\documentclass[border=10mm,12pt, varwidth]{amsart}
\usepackage[left=1in,top=1in,right=1in,bottom=1in]{geometry}
\usepackage{amsmath, amssymb, amsthm, bm}
\usepackage[square, numbers]{natbib}
\usepackage[colorlinks,citecolor=blue, linkcolor=blue, urlcolor=magenta]{hyperref}
\usepackage{stmaryrd}
\usepackage{setspace} 
\usepackage[mathscr]{eucal}
\usepackage[textsize=small]{todonotes}
\usepackage{graphicx}
\usepackage{latexsym}
\usepackage{psfrag}
\usepackage{epsfig}
\usepackage{enumitem}
\usepackage{comment}
\include{psbox}
\usepackage{amsfonts}
\usepackage{graphicx}
\usepackage{xcolor}
\usepackage{macro}

\newcommand{\bexp}{q}
\newcommand{\sexp}{q}

\newcommand{\gexp}{q}
\newcommand{\aex}{n}

\newcommand{\hexp}{\nu}
\newcommand{\holex}{\alpha}

\newcommand{\const}{\mathscr{C}}

\newcommand{\consta}{\mathscr{K}}

\newcommand{\loss}{\SC{L}}
\newcommand{\lfun}{\beta}
\newcommand{\gfun}{\Phi}

\newcommand{\occ}{\Gamma}
\newcommand{\inv}{\pi^{st}}

\newcommand{\mart}{\mathscr{M}}

\newcommand{\aml}{\rm{AMLE}}

\newcommand{\parsp}{\Theta}
\newcommand{\dfpsp}{\mathbb{S}}
\newcommand{\drpsp}{\mathbb{M}}

\newcommand{\mdrft}{A}
\newcommand{\drft}{\mu}
\newcommand{\dffun}{\s}
\newcommand{\diff}{\vas}
\newcommand{\drbd}{J}
\newcommand{\dibd}{J}

\newcommand{\pest}{\hat\drft^\vep_{\aml}}

\newcommand{\para}{\theta}
\newcommand{\apar}{\eta}

\newcommand{\auxp}{\alpha}
\newcommand{\pen}{\g}

\newcommand{\acov}{S}

\newcommand{\eqd}{\stackrel{dist}=}
\newcommand{\ert}{\stackrel{\vep \rt 0}\Rt}
\newcommand{\ewrt}{\stackrel{\vep \rt 0}\RT}
\newcommand{\prt}{\stackrel{\PP_{\para_0}}\Rt}

\numberwithin{equation}{section} 

\makeatletter
\@namedef{subjclassname@2020}{%
  \textup{2020} Mathematics Subject Classification}
\makeatother

\title[Asymptotic analysis of estimators of SDEs]{Asymptotic analysis of estimators of ergodic stochastic differential equations}

\author{Arnab Ganguly}
\address{Department of Mathematics, Louisiana State University, USA.}
\email{aganguly@lsu.edu}
\thanks{Research of A. Ganguly is supported in part by NSF DMS - 1855788, NSF DMS - 2246815 and Simons Foundation (via Travel Support for Mathematicians)}



\begin{document}

\begin{abstract}
The paper studies asymptotic properties of estimators of multidimensional stochastic differential equations driven by Brownian motions from high-frequency discrete data. Consistency and central limit properties of a class of estimators of the diffusion parameter and an approximate maximum likelihood estimator of the drift parameter based on a discretized likelihood function have been established in a suitable scaling regime involving  the time-gap between the observations and the overall time span. Our framework is more general than that typically considered in the literature and, thus, has the potential to be applicable to a wider range of stochastic models.

\end{abstract}
\date{\today}

\maketitle

\setcounter{equation}{0}
\renewcommand {\theequation}{\arabic{section}.\arabic{equation}}
\section{Introduction.}\label{intro}
The temporal dynamics of a variety of systems arising from systems biology, environmental science, engineering, physics, medicine can be be effectively modeled using Stochastic Differential Equations (SDEs) \cite{VK92, Wilk06, CWG09, DoSa13, AnKu15, GaAl15}.  In modern financial mathematics, SDEs play a key role in modeling short-term interest rates, asset pricing, option pricing, and their associated volatilities \cite{Sun00} . Understanding behaviors of these systems requires not just building mathematical models but integrating it with available data. Developing such well calibrated models hinges on the precise estimation of the underlying parameters in SDEs. The primary goal of this paper is to develop limit theorems which can be used to assess the accuracy of a class of such estimators for these parameters in a general framework.

Consider a $d$-dimensional parametric SDE driven by  a $d$-dimensional Brownian motion $W$, 
\begin{align}\label{eq:SDE0}
	X(t) &= x_0 + \int_0^t b(\drft, X(s))ds+  \int_0^t \dffun(\diff,X(s)) dW(s), \qquad x_0 \in \R^d,
\end{align}	
where the parameter $\para \dfeq (\drft, \diff) \in \parsp \dfeq \drpsp\times\dfpsp \subset \R^{n_{0}}\times \R^{n_{1}\times n_{2}}$.   We assume that for each $\para = (\drft, \diff) \in \Theta$, $X$ is ergodic with unique stationary (invariant) distribution $\inv_\theta$. We are interested in the consistency and central limit properties of of an estimator $\hat \para = (\hat \drft, \hat \diff)$ of the parameter $\para$. 

Although computational methods for parametric estimation is well studied at least in the case of It\^o diffusions (see \cite{Chib01, Era01, RoSt01, DuGa02, GoWi05, Bis08, GoWi08, ArOp11, DoSa13, CsOp13, BlSo14, FPR08, SuGa16, WGBS17} for a limited list), the accompanying theoretical results on consistency and central limit properties of the estimators  are relatively less explored in the literature. The books \cite{Kut03, Bis08} are standard references on various limiting properties of estimator of the drift parameter based for one-dimensional It\^o diffusion based on continuous data. While \cite{Bis08} considers the case of one-dimensional drift parameter only for additive SDEs,  \cite{Kut03} considers the more general case of multidimensional parameter estimation for multiplicative one-dimensional SDEs with known diffusion coefficient. Proofs of consistency of maximum likelihood estimator (MLE) of the drift parameter for one-dimensional  ergodic diffusions with known diffusion coefficient can also be found in \cite{Van01, WeSh16}; the second paper also proved asymptotic normality of the MLE (also see \cite{Van05}). 

These works focus on estimators of drift parameter based on a continuous trajectory of $X$. Of course, in most real cases, a continuous realization of $X$ is not possible, and data is available in a discretized form $\BX_{\tilde t_0:\tilde t_m} \dfeq (X(\tilde t_0), X(\tilde t_1), X(\tilde t_2), \hdots, X(\tilde t_m))$ at observation time points $\{\tilde t_i:i=0,1,2,\hdots,m\}$ over an interval $[0,T]$. The conventional idea is that if the data is high frequency in nature, that is, when $\tilde \Delta = \tilde t_i-\tilde t_{i-1}$ is small, the error due to discretization of the of the continuous path-based estimator, which in some cases is of the form  $\hat \drft_T = F_T(T^{-1}\int_0^{T} f(X(s)) dX(s), T^{-1}\int_0^T g(X(s)) ds)$, should be small.  However the effect of the time gap $\tilde\Delta$ on the error is much more subtle, and the resulting error analysis has to be done carefully. To understand the issues associated with naive discretization, notice that  $\tilde\Delta$ enters  the expression of   $\hat\drft^{\text{Disc}}_T$, the discretized version of $\hat\drft_T$,  in a nonlinear way (for instance, $F_{T}$ in most cases is nonlinear). Thus even if the discretization errors for approximating the integrals like $T^{-1}\int_0^T f(X(s)) dX(s)$ over a fixed interval $[0,T]$  are $O(\tilde\Delta)$, the precise order of error in terms of  $\Delta$ is for $\hat\drft^{\text{Disc}}_T$ remains unclear. The second important point is that the constant in $O(\tilde\Delta)$ bound above grows (typically, exponentially) with $T$. Thus as $T\rt \infty$, the error bound for a {\em fixed time-gap} $\tilde\Delta$ in general goes to infinity! This holds even if $X$ is ergodic with stationary distribution $\pi$, and $\f{1}{T}\int_0^T g(X(s)) ds$ has a finite limit $\int g(x) \pi(dx)$. This issue has been previously highlighted by the author in \cite{GaSu21} in the context of approximation of stationary distributions of SDEs (also see \cite{Tal90, GrTa13}).  This means that even when the abstract estimator $\hat\drft_T$ is consistent and converges to to the true value $\drft_0$ as $T\rt \infty$, its computable version $\hat\drft^{\text{Disc}}_T$ can potentially  diverge, rendering such estimators practically useless!

It is also important to realize that continuous-path-based estimators are difficult to interpret from a computational perspective. To see this, consider  the MLE of the parameter $\drft$ obtained by maximizing the likelihood function, which when the diffusion function $\dffun$ is known (that is, the diffusion parameter $\diff$ in \eqref{eq:SDE0} is known),  is given by
\begin{align}\label{eq:like-0}
L(\drft) = \exp \lf\{\int_0^T  a^{-1}(\diff,X(t))b(\drft, X(t)) \cdot dX(t) - \f{1}{2} \int_0^T b^T(\drft, X(t))a^{-1}(\diff,X(t))b(\drft, X(t))dt\ri\},
\end{align}
where $a(\diff,x) \dfeq \dffun(\diff, x)\dffun^T(\diff, x)$. 
Note that \eqref{eq:like-0} is an abstract expression and cannot be used for computation purposes in most cases. Even if data is available as an entire continuous path / realization, $X(\om) = x \in C([0,T])$, (where $\om$ is explicitly shown to emphasize that the path is a particular realization of the process $X$), neither the estimator (in most scenarios) nor the likelihood function $L(\cdot)$ can be computed for any given value of $\drft$ using \eqref{eq:like-0}. This is due to the fact that the stochastic integral is not defined  ``$\omega$ by $\omega$'', that is,   $\lf(\int_{0}^T f(X(s)) dX(s)\ri) (\om) \neq \int_{0}^T f(X(\om)(s)) dX(\om)(s)$; the right side is not even defined! Furthermore, note that if $\hat \drft_T = \hat \drft (X_{[0,T]})$ is an estimator of $\drft$, plugging in $\hat \drft_T$ in $L(\cdot)$ is problematic as the stochastic integral $\int_0^T a(X_t)^{-1}b(\hat \drft_T, X(t)) \cdot dX(t)$ is not defined (in the usual It\^o sense). This is because $\hat \drft_T$ depends on the whole path $X_{[0,T]}$ making the integrand not adapted to the underlying filtration $\{\salg_t\}$. Unfortunately, these crucial facts are sometimes overlooked resulting in serious gaps in a few existing papers; for example,  \cite{WeSh16} erroneously claimed martingale property of integrals like $\int_0^tf(\hat\para_T, X(s))dW(s)$. In the case of one-dimensional diffusions, a workaround exists by re-expressing the likelihood without the stochastic integral (see Kutoyants \cite{Kut03}). But this technique does not work in general for multidimensional SDEs.

This highlights the necessity of using discretization to obtain a computable likelihood for inference of $\theta$ even when a continuous path from the SDE is available! But as mentioned naive discretization using a fixed $\tilde\Delta$ can lead to unreliable estimators. This suggests that the asymptotic analysis needs to be done in a proper scaling regime involving $T$ and $\tilde\Delta$.
Quite a few papers in the literature have considered asymptotics of MLE under Euler type discretization of the SDEs \cite{BLSR87, Flo89, Yos92, Kes97, BLSR99}. \cite{Flo89, Kes97}  consider the case of one-dimensional SDEs. \cite{BLSR87,  BLSR99}  considered multidimensional with constant diffusion coefficient, this was improved in \cite{Yos92} which considered diffusion coefficient of the form $\dffun(\diff, x) = \dffun_0(x)\diff$.  The assumptions in these studies are quite restrictive, including conditions such as compactness of the parameter space, linear growth and Lipschitz continuity of the drift coefficients (with growth and Lipschitz constants independent of the parameter), strong differentiability conditions of the coefficients with respect to the parameter, etc. 

The goal of this paper is to study the limiting properties of  estimators of $\para=(\drft,\diff)$ for high-frequency data in a much more general framework. We consider a multidimensional parametric It\^o SDE of the form \eqref{eq:SDE0} with general drift and diffusion coefficient.  Under the assumption that a consistent estimator $\hat \diff^\vep$ of $\diff$ is available, we proved consistency and asymptotic normality of an approximate MLE, $\hat \drft^\vep_{\aml}$, of the drift parameter $\drft$ in an appropriate scaling regime involving time-gap $\tilde \Delta$ and time span $T=\vep^{-1}$ (c.f. Theorem \ref{th:const-AMLE-0} and Theorem \ref{th:aml-clt-0}). Next, we considered two specific forms of the diffusion function $\dffun(\diff,\cdot)$, defined a suitable estimator $\hat \diff^\vep$ and established its consistency and central limit asymptotics (c.f. Theorem \ref{th:const-diff-0} and Theorem \ref{th:clt-diff-0}). The conditions assumed in the paper for these results are much more general than those typically found  in the literature. In particular, the coefficients are required to satisfy only H\"older type continuity conditions in $x$ and the parameter argument while also simultaneously accommodating polynomial-type growth with respect to a Lypaunov function $V$, that ensures ergodicity (see Condition \ref{cond:SDE} and Condition \ref{cond:SDE-coeff} for details). Consequently, our findings apply to a considerably wider class of SDEs than those addressed in existing works. Future endeavor will address the case of sparse and noisy datasets. 

Before we describe the structure of the paper, we note that there are several additional works on parameter estimation for It\^o diffusion processes from discrete data, aside from those already mentioned, that employ alternative methods and lie outside the primary focus of this paper. A partial list includes include score function and estimating function based approach of \cite{BiSo95, KeSo99, Sor00},  different approximations of transition density function \cite{Saha02, BaJu05, ChCh11}, likelihood based inference using exact simulation of diffusions \cite{BPRF06}. These papers primarily focus on one-dimensional SDEs, and the methods are not easily implementable for multidimensional SDEs. A\"it-Sahalia in \cite{Saha08} extended his prior work on Hermite expansion based density approximation method to a certain class of multidimensional class of SDEs which can be converted to SDEs with constant diffusion coefficient via  Lamperti transform. 

\np
{\em Outline:} The layout of the rest of the paper is as follows. The mathematical framework and the required conditions are described in Section \ref{sec:math-frame}. Section \ref{sec:main-results} presents the main results of the paper, while Section \ref{sec:aux-results} contains some auxiliary results needed for their proofs. The proofs of the main results are provided in Section \ref{sec:main-proofs}. Finally Appendix \ref{sec:appendix} collects some additional supplementary materials.

\np
{\it Notations:} \textbullet\ $\R^{m\times n}$ denotes the space of all $m\times n$ matrices; $\R^{n\times n}_{\geq 0}, $ and $\R^{n\times n}_{>0}$ respectively denote the space of all $n\times n$ positive semi-definite (p.s.d.) and postive definite (p.d.) matrices. All p.s.d. matrices are assumed to be symmetric. \textbullet\ For a  p.s.d. matrix $A$, the square root, $A^{1/2}$, denotes the unique p.s.d matrix satisfying $A^{1/2}A^{1/2} = A$. For a square matrix $A$, $A_{\mathrm{sym}} \dfeq (A+A^T)/2.$ \textbullet\ $\ve_{m\times n}(A) \in \R^{mn}$ will denote the vectorization of an $m\times n$ matrix $A$.  \textbullet\ $C(\mathbb A, \R^d)$ denotes the space of continuous functions from $\mathbb A \subset \R^n \rt \R^d$ topologized by uniform convergence over compact sets.
 \textbullet\ For a differentiable function $f:\R^d \rt \R^{d_1}$, the $d\times d_1$ matrix $Df$, defined by $(Df)_{i,j} = \partial_j f_i$, denotes the first-derivative or Jacobian matrix. \textbullet\ For a function $f:\mathbb \R^n \rt \R^d$, 
 $$\argmin\{f(u): u\in \mathbb A\} \dfeq \{u_*: f(u_*) \leq f(u), \ u \in \mathbb A\}.$$
\textbullet\ $\SC{M}^+(\mathbb E)$ and $\SC{M}^+_1(\mathbb E)$ respectively denote the space of all non-negative measures and probability measures on the measurable space $(\mathbb E, \SC{E})$ equipped with topology of weak convergence.   \textbullet\ Convergence in probability with respect to a probability measure $\PP$ is denoted by $\cdot \stackrel{\PP}\Rt \cdot$. \textbullet\ $X \eqd X'$ means that the two random variables $X$ and $X'$ have the same distribution. \textbullet\ $\No_d(\mfk{m},\Sigma)$ denotes the normal distribution on $\R^d$ with mean vector $\mfk{m}$ and $d\times d$ covariance matrix $\Sigma$. \textbullet\ $\leb$ denotes the Lebesgue measure.

\section{Mathematical Framework}\label{sec:math-frame}
The standard approach to estimate the parameter $\apar \in \mathbb{H}$ of a stochastic model  is by solving the minimization problem $\min_{\apar \in \mathbb{H}} \loss^\vep(\apar)$ for a suitable (random) loss function $\loss^\vep: \mathbb{H} \rt \R$ and setting $\hat \apar^\vep = \argmin_{\apar \in \mathbb{H}} \loss^\vep(\apar).$ This of course requires the inherent assumption that a global minimum exists. The loss function certainly depends on the data from the model, which in case of the model \eqref{eq:SDE0}, is the time series data, $\BX_{\tilde t_0:\tilde t_m}$ ($t_m=\vep^{-1}$) from  interval $[0,\vep^{-1}]$. In other words, $\loss^\vep(\apar) \equiv \loss^\vep(\apar |\BX_{\tilde t_0:\tilde t_m})$, but for convenience we will often suppress the data part from the notation and simply write $\loss^\vep(\apar)$. A common choice of the loss function is the negative log-likelihood function of the parameter (when available), which results in the maximum likelihood estimate (MLE) of $\apar$. 

$\hat \apar^\vep$ is a consistent estimator of the parameter $\apar$ if $\hat \apar^\vep \stackrel{\PP_{\apar_0}} \Rt \apar_0$ for every $\apar_0$ (that is, when the true value of the parameter is $\apar_0$).

\subsection{Estimation of the drift parameter through approximate likelihood function}\label{sec:drift-est}
We work on a probability space  $(\Omega, \SC{F}, \PP)$ with $\PP$ belonging to a parametric family of probability measures $\lf\{\PP_{\para}: \para =(\drft, \diff) \in \parsp = \drpsp\times\dfpsp \ri\}$, and under $\PP_{\para}$, $X$ satisfies the SDE  \eqref{eq:SDE0}. We now describe a natural approach to estimation of $\para = (\drft, \diff)$.

\np
{\em Case of known diffusion parameter:} When the diffusion parameter, $\diff$, is known, the unknown parameter $\para$ to be estimated is just the drift parameter $\drft$.

 Let $\BX_{\tilde t_0:\tilde t_m} \dfeq (X(\tilde t_0), X(\tilde t_1), \hdots, X(\tilde t_m))$ be the discrete data from the interval $[0,\vep^{-1}]$. We assume that $\tilde t_m=\vep^{-1}$. 
A Riemann sum-based discretization of \eqref{eq:like-0} gives the discretized log-likelihood function of $\drft$, as $\ell^{m,\tilde\Delta}(\drft|\diff, \BX_{\tilde t_0:\tilde t_m})$, where for  fixed $\bx_{0:m} =(x_0,x_1,x_2,\hdots,x_m) \in \R^{d\times {(m+1)}}$, $\diff \in \dfpsp$, $0\leq \delta \leq 1$ the function $\ell^{m,\delta}(\cdot|\diff, \bx_{0:m}):\R^{n_0} \rt \R$  is defined as 
\begin{align}\label{eq:disc-like-0}
\begin{aligned}
\ell^{m,\delta}_{D}(\drft |\diff, \bx_{0:m})=&\ \sum_{i=1}^m a^{-1}(\diff,x_{i-1})b(\drft, x_{i-1})\cdot \lf(x_i-x_{i-1})\ri) \\
& \hs{.2cm}- \f{\delta }{2} \sum_{i=1}^m  b^T(\drft, x_{i-1})a^{-1}(\diff,x_{i-1})b(\drft, x_{i-1}).\\
\end{aligned}
\end{align}
%
Since our asymptotics will feature $\vep \rt 0$ and the time-gap, $\tilde\Delta$, will be scaled with inverse of the time span, $\vep$, it is convenient to denote $\ell^{m, \tilde\Delta}_D(\drft|\diff, \BX_{\tilde t_0:\tilde t_m})$ by $\ell^\vep_D(\drft|\diff, \BX_{\tilde t_0:\tilde t_m})$ for $m=(\tilde\Delta\vep)^{-1}$. We are tacitly assuming here that $(\tilde\Delta\vep)^{-1}$ is an integer, which of course does not result in any loss of generality. 

Note that $\ell^\vep_D(\drft|\diff, \BX_{\tilde t_0:\tilde t_m})\equiv \ell^{m=(\tilde\Delta\vep)^{-1}, \tilde\Delta}_D(\drft|\diff, \BX_{\tilde t_0:\tilde t_m}) $ has the following  equivalent integral representation which is often more convenient to work with for the purpose of analysis.
\begin{align}\label{eq:disc-like-int-0}
\begin{aligned}
 \ell^\vep_{D}(\drft |\diff, \BX_{\tilde t_0:\tilde t_m}) =&\ \int_0^{\vep^{-1}}   a^{-1}(\diff,X(\tilde{\vr_\vep}(s)))b(\drft, X(\tilde{\vr_\vep}(s)))\cdot dX(s)\\
& \hs{.2cm}  - \f{1 }{2} \int_0^{\vep^{-1}} b^T(\drft, X(\tilde{\vr_\vep}(s)))a^{-1}(\diff,X(\tilde{\vr_\vep}(s)))b(\drft, X(\tilde{\vr_\vep}(s))) ds,
\end{aligned}
\end{align}
where the step function $\tilde\vr_\vep$ corresponding to the partition $\{\tilde t_i\}$ is defined as $\tilde\vr_\vep(s) = \tilde t_i$ if $\tilde t_i \leq s < \tilde t_{i+1}.$

 Taking  the (discretized) negative log-likelihood, $- \ell^\vep_{D}(\drft |\diff, \BX_{\tilde t_0:\tilde t_m})$ (c.f. \eqref{eq:disc-like-0}) or equivalently its scaled version, $- \vep \ell^\vep_{D}(\drft |\diff, \BX_{\tilde t_0:\tilde t_m})$ as the loss function gives an (approximate) MLE of the drift parameter, $\drft$.  If $\drft$ is in a high-dimensional space, it is more appropriate to consider a penalized or regularized version of this loss function with 
\begin{align}\label{eq:pen-like}
\loss^\vep_D(\drft) \dfeq -\vep \ell^\vep_{D}(\drft |\diff, \BX_{\tilde t_0:\tilde t_m}) + \pen^\vep(\drft)
\end{align}
and define $\hat\drft^\vep_{\aml}$ as $\dst\hat\drft^\vep_{\aml} \dfeq \argmin_{\drft \in \drpsp} \lf\{-\vep\ell^\vep_{D}(\drft)+ \pen^\vep(\drft)\ri\}.$
A common choice of the penalty function involving $L^p$-norm of $\drft$ is of the form, $\pen^\vep(\drft) = \vep^{\alpha}  \|\drft\|^{p'}_p, \ \alpha >0, \ p,p'\geq 0.$ 

\np
{\em Case of unknown diffusion parameter:} When the diffusion parameter, $\diff$, is unknown, two diffusions corresponding to two different values of $\diff$ are singular; consequently, Girsanov's theorem cannot be used to get a likelihood function for $\drft$. A natural approach for the estimation problem in this case involves an approximate maximum likelihood estimation of the drift parameter, $\drft$, by minimizing the (penalized) negative of an approximate likelihood function, $\ell^\vep_A$, obtained by replacing the $\diff$ parameter with a consistent estimator, $\hat \diff^\vep$, in $ \ell^\vep_{D}(\drft |\diff, \BX_{\tilde t_0:\tilde t_m})$ (see \eqref{eq:disc-like-0}, or equivalently, \eqref{eq:disc-like-int-0}).

In other words, let $\hat\diff^\vep$ be a consistent estimator of $\diff$, that is, $\hat\diff^\vep \prt \diff_0$ for any $\para_0=(\drft_0, \diff_0)$. We assume that there exists a (deterministic) measurable function $F^\vep:\R^{d\times (m+1)} \rt \R^{n_1}$ such that $\hat\diff^\vep = F^\vep(\BX_{\tilde t_0:\tilde t_m})$. 

Consider the estimator $\hat\drft^\vep_{\aml}$ defined by
\begin{align}\label{eq:est-pen}
\hat\drft^\vep_{\aml} \in &\ \argmin_{\drft \in \drpsp} \lf\{-\vep\ell^\vep_{A}(\drft)+ \pen^\vep(\drft)\ri\}
\end{align}
where $\ell^\vep_{A}(\drft) \equiv\ \ell^\vep_{A}(\drft |\BX_{\tilde t_0:\tilde t_m}) \dfeq \  \ell^\vep_{D}(\drft |\hat\diff^\vep,\BX_{\tilde t_0:\tilde t_m})$, that is, 
\begin{align}\label{eq:approx-like-int-0}
\begin{aligned}
\ell^\vep_{A}(\drft) \dfeq &\
 \int_0^{\vep^{-1}}  a^{-1}(\hat\diff^\vep,X(\tilde{\vr_\vep}(s)))b(\drft, X(\tilde{\vr_\vep}(s)))\cdot dX(s)\\
&\ \hs{.2cm}  - \f{1 }{2} \int_0^{\vep^{-1}} b^T(\drft, X(\tilde{\vr_\vep}(s)))a^{-1}(\hat\diff^\vep,X(\tilde{\vr_\vep}(s)))b(\drft, X(\tilde{\vr_\vep}(s))) ds.
\end{aligned}
\end{align}
Notice that there are no adaptability issues in plugging in $\hat\diff^\vep$ in place of $\diff$ in the `stochastic integral' term, as here it is just a finite sum.


Suppose the mappings $\drft \rt \pen^\vep(\drft)$ and $\drft \rt b(\drft,x)$  for each $x$ are differentiable, then  $\ell^\vep_{A}$ is differentiable in $\drft$.  Consequently, if $\drpsp$ is open and convex, $\hat\drft^\vep_{\aml}$ is a critical point of the loss function, $\loss^\vep_A(\cdot) \dfeq -\vep\ell^\vep(\cdot)+\pen^\vep(\cdot)$, that is, the solution of the {\em penalized likelihood equation}, 
\begin{align}\label{eq:lik-crit}
0=\nabla_{\drft}\loss^\vep_A(\drft) \equiv - \vep \nabla_{\drft}\ell^\vep_{A}(\drft)+\nabla_\drft \pen^\vep(\drft),
\end{align}
where 
\begin{align}\label{eq:deriv-like-0}
\begin{aligned}
\nabla_{\drft}\ell^{\vep}_A(\drft) =&\ \int_0^{\vep^{-1}}D^T_{\drft} b(\drft, X(\tilde\vr_\vep(s)))a^{-1}(\hat\diff^\vep,X(\tilde{\vr_\vep}(s))) dX(s)\\
	& \hs{.2cm} -\int_0^{\vep^{-1}} D^T_{\drft} b(\drft, X(\tilde\vr_\vep(s))) a^{-1}(\hat\diff^\vep,X(\tilde{\vr_\vep}(s))) b(\drft, X(\tilde\vr_\vep(s)))\ ds.
\end{aligned}
\end{align}
The above expression is obtained by simply
 interchanging differentiation and the integral in \eqref{eq:approx-like-int-0}. This is allowed without any extra condition as the integral in  \eqref{eq:approx-like-int-0} is simply a finite sum. 



\subsection{Estimation of the diffusion parameter}\label{sec:diff-est}
One approach to estimate the diffusion parameter $\diff$ is via the matrix-valued quadratic variation, which satisfies $\mqd{X}_t =  \int_0^t a(\diff,X(s))ds$, where recall that $a(\diff,x) = \dffun(\diff,x)\dffun^T(\diff,x)$. Thus going to their discretized versions we define the estimator $\hat\diff^\vep$ through the equation
\begin{align}\label{eq:diff-est-def} 
\mqd{X}^{D,\tilde\vr_\vep}_{\vep^{-1}} \stackrel{set}=  \int_0^{\vep^{-1}} a(\hat\diff^\vep,X(s))ds
 \end{align}
 where $\mqd{X}^{D,\tilde\vr_\vep}$ is the discretized quadratic variation process of $X$ according to the partition $\{\tilde t_i\}$ or equivalently its corresponding step function $\tilde\vr_\vep$, that is,
 \begin{align}\label{eq:quad-disc-X}
 \mqd{X}^{D,\tilde\vr_\vep}_t \dfeq \sum_{i=0}^{[t/\tilde\Delta]-1} (X(\tilde t_{i+1}) - X(\tilde t_i))(X(\tilde t_{i+1}) - X(\tilde t_i))^T.
 \end{align}
 
We consider two  particular forms of the diffusion parameter $\dffun$ in this paper with the parameter space $\dfpsp = \R^{d\times d}_{>0}$:

\np
\textbullet\ {\em Form 1}: $\dffun$ is of the form $\dffun(\diff,x) = \lf(a_0(x)\diff\ri)^{1/2}$ or $\lf(\diff a_0(x)\ri)^{1/2}$, where $\diff \in \R^{d\times d}_{>0}$ is the unknown parameter, $a_0: \R^d \rt \R^{d\times d}_{>0}$ is a  known function. Recall for a p.d. matrix $A$, $A^{1/2}$ denotes the unique p.d. square root of $A$. 

The treatments of both these types are essentially the same. For specificity, we focus on the case, $\dffun(\diff,x) = \lf(a_0(x)\diff\ri)^{1/2}$ for which $a(\diff,x) = a_0(x)\diff $  leading to the estimator $\hat\diff^\vep \equiv \hat\diff^\vep(\BX_{\tilde t_0:\tilde t_m}) $ defined by
\begin{align}\label{eq:diff-est-1}
\hat\diff^\vep \dfeq \lf(\int_0^{\vep^{-1}} a_0(X(\tilde\vr_\vep(s)))ds\ri)^{-1}\mqd{X}^{D,\tilde\vr_\vep}_{\vep^{-1}}.
\end{align}

\np
\textbullet\ {\em Form 2}: $\dffun(\diff,x) = \dffun_0(x)\kappa$, where the parameter $\kappa \in \R^{d\times n_1}$ is such that $\diff = \kappa\kappa^T \in \R^{d\times d}_{>0}$ and $\dffun_0: \R^d \rt \R^{d\times d}_{>0}$ is a  known function. Note for a fixed $\diff$, any $\kappa$ satisfying $\diff=\kappa\kappa^T$ lead to distributionally equal SDEs, and thus the unknown parameter to be estimated in this case is $\diff$. Here
$a(\diff,x) = \dffun_0(x)\diff \dffun^T_0(x)$, and \eqref{eq:diff-est-def} leads to the following estimator defining equation
\begin{align*}
\mqd{X}^{D,\tilde\vr_\vep}_{\vep^{-1}} =  \int_0^{\vep^{-1}} \dffun_0(X(\tilde\vr_\vep(s)))\hat\diff^\vep\dffun^T_0(X(\tilde\vr_\vep(s)))ds.
\end{align*}
Using the $\ve_{d\times d}$ operator we get the estimator $\hat\diff^\vep \equiv \hat\diff^\vep(\BX_{\tilde t_0:\tilde t_m})$ as
\begin{align}\label{eq:diff-est-2}
\ve_{d\times d}\lf(\hat \diff^\vep\ri) \dfeq \lf(\int_0^{\vep^{-1}} \dffun_0(X(\tilde\vr_\vep(s))) \ot \dffun_0(X(\tilde\vr_\vep(s)))ds\ri)^{-1} \ve_{d\times d} \lf(\mqd{X}^{D,\tilde\vr_\vep}_{\vep^{-1}}\ri).
\end{align}

\subsection{Assumptions}
We collect all the conditions required on the SDE \eqref{eq:SDE0}. Not every condition will be needed for each result presented in the paper. We will note the specific ones that are required for each individual result. We will assume that $b:\R^{n_{0}}\times \R^d \rt \R^d$ and $\dffun:\R^{n_1\times n_2}\times \R^d \rt \R^{d\times d}$. This allows us to plug in estimators $\hat\drft^\vep$ and $\hat\diff^\vep$ in $b(\cdot,\cdot)$ and $\dffun(\cdot,\cdot)$ even when $(\hat\drft^\vep, \hat\diff^\vep)$ lie outside $\parsp = \drpsp\times \dfpsp$. 

Now, the various assumptions we will impose only concern $b\Big|_{\drpsp\times \R^d}$ and $\dffun\Big|_{\dfpsp\times \R^d}$, the restrictions of $b$ and $\dffun$ to $\drpsp\times \R^d$ and $\dfpsp\times \R^d$, respectively. Throughout, we will assume that $a(\diff,x) = \dffun(\diff,x)\dffun(\diff,x)^T$ is p.d. for each $\para=(\drft,\diff) \in \parsp = \drpsp\times \dfpsp$ and $x$ in the state-space of the SDE $X$. 

The following form of Lyapunov condition is needed to establish ergodicty of the process $X^\vep$ and related integral moment like bounds.
\begin{condition} \label{cond:SDE}  There exists a Lyapunov function $V: \R^d \rt [1,\infty)$ such that for every $\para = (\drft,\diff) \in \drpsp\times\dfpsp$,
	\begin{enumerate}[label=(\roman*), ref=(\roman*)] 
           \item \label{cond:item:growth-rec}   $\dst\beta_*(\theta) \equiv -\limsup_{\|x\| \rt \infty}\SC{L}_\theta V(x)>0;$

	\item \label{cond:item:growth-lyap-a} $\nabla^TV(x)a(\diff, x)\nabla V(x) = o(V|\SC{L}_\theta V|)$ as $\|x\| \rt \infty$, that is, 
	$$ \f{\nabla^TV(x)a(\diff,x)\nabla V(x)}{V|\SC{L}_\theta V|} \ \stackrel{\|x\| \rt \infty}\Rt 0;$$
	
	\item \label{cond:item:growth-lyap-b}	for some exponent $\gexp_{V,0}$ and constant $J_{V,0}>0$ such that
	$$\max\lf\{\|\nabla V(x)\|, \|D^2 V(x)\|\ri\}  \leq J_{V,0}  V(x)^{\gexp_{V,0}};$$ 	

\item \label{cond:item:growth-gen-convex}	for some exponent $\gexp_{V,1}$ and constant $J_{V,1}>0$ such that for any $0\leq u\leq 1$
	$$  V((1-u)x+uy) \leq J_{V,1}  \lf(V(x)^{\gexp_{V,1}}+V(y)^{\gexp_{V,1}}\ri), \quad x,y \in \R^d.$$ 	

	\end{enumerate}
\end{condition}
Notice the requirement $V(x) \geq 1$ is without loss of generality in the sense that if $\bar V$ is a non-negative function satisfying the above condition, then so does $V \equiv 1+\bar V$.
	
\begin{remark}[Existence of stationary distribution]\label{rem:exst-stat-dist}
{\rm
Condition \ref{cond:SDE}:\ref{cond:item:growth-rec} \& \ref{cond:item:growth-lyap-a} and positive definiteness of the $a=\s\s^T$ matrix imply that for every $\para_0=(\drft_0,\diff_0) \in \parsp$, the process $X$ admits a unique stationary distribution, $\inv_{\para_0}$ (see Lemma \ref{lem:tight-occ} for details).
 }   
\end{remark}

Of course, Condition \ref{cond:SDE}:\ref{cond:item:growth-gen-convex} holds if $V$ is quasiconvex, but, in general, it holds for a much bigger class of functions.

\begin{remark}\label{rem:SDE-cond} {\rm
Condition \ref{cond:SDE} implies the following:
\begin{enumerate}[label=(\alph*), ref=(\alph*)]
\item \label{rem:SDE-gen} There exists a constant $R_0(\para)$ such that
\begin{align*}
\SC{L}_\para V(x) \leq -\beta_*(\para)/2, \quad \text{ for all } \|x\| \geq R_0(\para).
\end{align*}
In particular, $\SC{L}_\para V(x) <0$ for $\|x\| \geq R_0(\para)$.

\item \label{rem:SDE-diff} For every constant $K>0$, there exists a constant $R_K(\para) \geq R_0(\para)$ such that
\begin{align*}
\nabla^TV(x)a(\diff,x)\nabla V(x) \leq \f{1}{K} V|\SC{L}_\para V| =- \f{1}{K} V \SC{L}_\para V, \quad \text{ for all } \|x\| \geq R_K(\para).
\end{align*}
\end{enumerate}
}
\end{remark}

\begin{example}\label{exam:rec-cond-b}{\rm 
A large class of ergodic SDEs with drift and diffusion coefficients $b$ and $\dffun$  satisfy the following conditions:
\begin{enumerate}[label=(\roman*)]
\item for some nonnegative constants $\mfk{c}_0(\drft), \bexp_0$ and $B_0(\drft)$, $\<x, b(\drft,x)\> \leq - \mfk{c}_0(\drft)\|x\|^{\bexp_0}$ for all $\|x\| \geq B_0(\drft)$.
\item  for some constants $\mfk{c}_1(\diff) \geq 0$ and $\sexp_{\dffun,0}$, $\|\dffun(\diff,x)\| \leq \mfk{c}_1(\diff)(1+\|x\|)^{\sexp_{\dffun,0}}$  with the parameter $\para=(\drft, \diff)$ such that $\sexp_{\dffun,0} < \f{1}{2}\bexp_0$. 
\end{enumerate}

Here the function $V$, defined by $V(x) =1+\f{1}{2}\|x\|^2$, is a Lypaunov function satisfying Condition \ref{cond:SDE}. Now of course it is clear from the choice of $V$ that Condition \ref{cond:SDE}:\ref{cond:item:growth-lyap-b} \& \ref{cond:item:growth-gen-convex} hold. For the other two conditions, notice that since $\sexp_{\dffun,0}< \f{1}{2}\bexp_0$,   $\SC{L}_\para V(x) = \<x, b(\drft,x)\> + \f{1}{2} \|\dffun(\diff,x)\|^2$ satisfies
\begin{align*}
\SC{L}_\theta V(x)  \leq -\f{1}{2} \mfk{c}_{2}(\theta)\|x\|^{\bexp_0}, \quad \|x\| \geq B_1(\theta)
\end{align*}
for some constant $\mfk{c}_{2}(\theta)$ and sufficiently large $B_1(\para)$. In particular, $\SC{L}_\theta V(x) \rt -\infty$ as $\|x\|\rt \infty$ and thus Condition \ref{cond:SDE}:\ref{cond:item:growth-rec} holds. Furthermore, for $\|x\| \geq B_1(\para)$ and some constant $\mfk{c}_3(\para) \geq 0$,
\begin{align*}
\f{\nabla^TV(x)a(\diff,x)\nabla V(x)}{V|\SC{L}_\theta V|} = \f{\nabla^TV(x)a(\diff,x)\nabla V(x)}{-V \SC{L}_\theta V} \leq \mfk{c}_3(\para)\f{(1+\|x\|)^{\sexp_{\dffun,0}+2}}{\|x\|^{\bexp_0+2}} \stackrel{\|x\| \rt \infty}\Rt 0,
\end{align*}	
and thus Condition \ref{cond:SDE}:\ref{cond:item:growth-lyap-a} also holds.

}
\end{example}

Before we list the various growth and regularity conditions needed on the drift and diffusion functions, $b$ and $\dffun$, we first define a local H\"older type condition on parameters that we often need the functions to satisfy.

\begin{definition}\label{def:Holder}
Let $\B,\mathbb{U}$ be normed spaces,  $F:\B \times \R^d \rt \mathbb{U}$, $V: \R^d \rt [1,\infty)$ measurable functions. The mapping $(\apar,x) \in \B \times \R^d \Rt F(\apar,x) \in \mathbb{U}$ satisfies a {\bf local H\"older-type condition in $\apar$} with growth function $V$, if for every compact set $\cmpt \subset \B$, there exist constant $\consta_{\cmpt}$, H\"older exponent $0\leq \hexp_{F,0} \leq 1$ and growth exponent $\gexp_{F,0}\geq 0$ (each potentially depending on $\cmpt$) such that
\begin{align*}
\|F(\apar_1, x) - F(\apar_2,x)\|_{\mathbb{U}} \leq \consta_{\cmpt} V(x)^{\gexp_{F,0}}\|\apar_1-\apar_2\|^{\hexp_{F,0}}_{\B}, \quad \apar_1, \apar_2 \in \cmpt.
\end{align*}
\end{definition}

Recall that a function $f: \mathbb A \subset \R^{m} \rt \R^n$ is {\em locally bounded} if for any compact set $\mathbb A_0 \subset \mathbb A$, $ \sup\limits_{x \in \mathbb A_0}\|f(x)\| < \infty.$
\begin{condition} \label{cond:SDE-coeff}  The drift and the diffusion coefficients $b$ and $\dffun$ satisfy the following:
\begin{enumerate}[label=(\roman*), ref=(\roman*), leftmargin=*]
\item \label{cond:b-growth} for some locally bounded  functions $\drbd_{b,0}, \drbd_{b,1}, \drbd'_{b,0} : \drpsp \rt [0,\infty),$ and exponents (not depending on $\drft$), $\bexp_{b,0}, \bexp_{b,1}, \bexp'_{b} \geq 0$ and $0<\holex_b \leq 1$, 
\begin{align*}
&\ \hypertarget{subitem-a}{(a)}\ \|b(\drft, x)\| \leq \drbd_{b,0}(\drft) V(x)^{\bexp_{b,0}},\quad \hypertarget{subitem-b-lip}{(b)}\   \|b(\drft, x) - b(\drft, x') \| \leq \drbd_{b,1}(\drft) \lf(V(x)^{\bexp_{b,1}} + V(x')^{\bexp_{b,1}}\ri)\|x-x'\|^{\holex_b}\\
&\ \hypertarget{subitem-c}{(c)}\ \|D_{\drft}b(\drft, x)\| \leq \drbd'_{b,0}(\drft) V(x)^{\bexp'_{b,0}}
\end{align*}

\item \label{cond:b-lip-para} the mapping $(\mu, x) \in \drpsp \times \R^d \Rt b(\drft,x) \in \R^d$ satisfies a local H\"older-type condition in $\drft$ with growth function $V$ in the sense of Definition \ref{def:Holder}.


\item \label{cond:Db-lip-para} the mapping $(\mu, x) \in \drpsp \times \R^d \Rt D_{\drft}b(\drft,x) \in \R^{d\times n_0}$ satisfies a local H\"older-type condition in $\drft$ with growth function $V$ in the sense of Definition \ref{def:Holder}.


\item \label{cond:sig-growth} for some locally bounded  functions $\dibd_{\dffun,0}, \dibd_{\dffun,1}, \dibd_{a,0}, \dibd'_{a,0}, \dibd'_{a,1}: \dfpsp \rt [0,\infty)$ and exponents (not depending on $\diff$), $\sexp_{\dffun,0}, \sexp_{\dffun,1}, \sexp_{a,0}, \sexp'_{a,0}, \sexp'_{a,1} \geq 0$ and $0<\holex_{\dffun}, \holex'_{a} \leq 1$,
\begin{align*}
&\ \hypertarget{subitem-s-a}{(a)}\ \|\dffun(\diff,x)\| \leq \dibd_{\dffun,0}(\diff)V(x)^{\sexp_{\dffun,0}},\ \ 
 \hypertarget{subitem-s-b}{(b)}\  \|\dffun(\diff,x) - \dffun(\diff,x')\|  \leq \dibd_{\dffun,1}(\diff)\lf(V(x)^{\sexp_{\dffun,1}} + V(x)^{\sexp_{\dffun,1}}\ri) \|x-x'\|^{\holex_{\dffun}};\\
 &\ \hypertarget{subitem-s-c}{(c)}\   \|a^{-1}(\diff,x)\| \leq \dibd_{a,0}(\diff)V(x)^{\sexp_{a,0}}, \quad  \hypertarget{subitem-s-d}{(d)}\   \|D_x a(\diff,x)\| \leq \dibd'_{a,0}(\diff)V(x)^{\sexp'_{a,0}}, \\ 
 &\ \hypertarget{subitem-s-e}{(e)}\  \|D_x a(\diff,x) - D_x a(\diff,x')\|  \leq \dibd'_{a,1}(\diff)\lf(V(x)^{\sexp'_{a,1}} + V(x)^{\sexp'_{a,1}}\ri) \|x-x'\|^{\holex'_{a}};
 \end{align*}
\item \label{cond:diff-inv-lip-para} 
the mapping $(\diff, x) \in \dfpsp \times \R^d \Rt a^{-1}(\diff,x) \in \R^{d\times d}$ satisfies a local H\"older-type condition in $\diff$ with growth function $V$ in the sense of Definition \ref{def:Holder}.
 \end{enumerate}
\end{condition}

\begin{remark} {\rm
Note that in Condition \ref{cond:SDE-coeff} above, \ref{cond:b-growth}-\hyperlink{subitem-c}{(c)} implies local Lipschitz-type continuity of the mapping $(\drft,x) \rt b(\drft,x)$ in $\drft$ and thus implies \ref{cond:b-lip-para}. The reason we have both listed here is because they are used in two separate results --- the  weaker assumption of local H\"older continuity of $b(\cdot,x)$ in \ref{cond:b-lip-para} is enough for the consistency of the approximate MLE estimator,  $\hat\drft^\vep_{\aml}$, while the stronger  \ref{cond:b-growth}-\hyperlink{subitem-c}{(c)} is needed to prove the associated CLT in a more general setting. Similarly, \ref{cond:sig-growth}-\hyperlink{subitem-s-d}{(d)} implies Lipschitz-type continuity of the mapping $(\diff,x) \rt a(\drft,x)$ in $x$.
}
\end{remark}

\section{Limit theorems for estimators}\label{sec:main-results}

\subsection{General result on consistency}
We first establish the following result on convergence of minimas, which leads to a general result on consistency of estimators. 
\begin{theorem} \label{th:const-gen}
Let $(\Omega, \SC{F}, \PP)$ be a probability space, and $\{\loss^\vep, \vep>0\}$ a family of stochastic processes (random fields) with path space, $C(\mathbb A, \R)$, where $\mathbb A \subset \R^{d_0}$ is measurable. 
Assume the following conditions hold.
\begin{enumerate}[label=(\roman*)]
\item  for each $\vep>0$, $\loss^\vep(\cdot)$ admits a global minimum;
\item there exists a deterministic function $\loss^0: \mathbb A \rt [0,\infty)$ such that $\loss^\vep(\cdot) \stackrel{\PP}\Rt \loss^0(\cdot)$ in $C(\mathbb A, [0,\infty))$ as $\vep \rt 0$;
\item $\loss^0(\cdot)$ has a unique (global) minimum at $\auxp=\auxp_*$.
\end{enumerate}
Let $\{\hat{\auxp}^\vep\}$ be a family of random variables such that for each $\vep>0$, 
\begin{align}\label{eq:aux-par-est}
\hat \auxp^\vep \in \argmin\lf\{\loss^\vep(\alpha): \alpha \in \mathbb A\ri\}.
\end{align}
Assume that one of the following conditions hold.
\begin{enumerate}[label=(\roman*)]
\setcounter{enumi}{3}
\item Either
\begin{enumerate}[label=(\alph*), ref=\theenumi{}-(\alph*)]
\item \label{cond:est-tight}
 $\mathbb A$ is closed and $\{\hat{\auxp}^\vep\}$  is tight; or,
 \item \label{cond:parsp} $\mathbb A$ is open, convex, and for each $\vep$ and $\om$, the mapping $\auxp \rt \loss^\vep(\auxp, \om)$ is differentiable, and the equation $\nabla_\auxp \loss^\vep(\auxp, \om)=0$ has a unique solution. 
\end{enumerate}
\end{enumerate}
Then  $\hat{\auxp}^\vep \stackrel{\PP}\Rt \auxp_*$ as $\vep \rt 0$.
 
\end{theorem}

\begin{remark}{\rm
Clearly, the  Conditions \ref{cond:est-tight} and \ref{cond:parsp} of Theorem \ref{th:const-gen} are not mutually exclusive, and tightness of $\{\hat{\auxp}^\vep\}$ also holds under Condition \ref{cond:parsp}. Now in some cases tightness of $\{\hat{\auxp}^\vep\}$ can be directly verified. The simplest instance of this is when the  space $\mathbb A$ is already compact. In other cases the form of the minimizer, $\hat{\auxp}^\vep$, might already be known from which tightness can be derived by standard moment-based estimates. Interestingly, we will show in the proof of Theorem \ref{th:const-gen}  that under condition \ref{cond:parsp}  the convergence of $\hat{\auxp}^\vep$ can be directly established without invoking tightness.}
\end{remark}

\np
{\em Connection to consistency:} For applications to consistency, consider a stochastic model defined on the probability space $(\Omega, \SC{F}, \PP)$ with $\PP$ belonging to a parametric family of probability measures $\{\PP_{\alpha}: \auxp \in \mathbb A \}$. In this context, for each $\vep>0$, $\loss^\vep:\mathbb A \subset \R^{d_0} \rt [0,\infty]$  denotes a suitable (random) loss function of the parameter $\auxp \in \mathbb A $ of the model depending on the data of ``size" measured in terms of $\vep^{-1}$. Thus $ \hat\auxp^\vep$, in this case, is an estimate of the parameter $\auxp$ of the model based on the  minimization of the loss function $\loss^\vep$.

Suppose that $\loss^\vep$ converges to a deterministic function, $\loss^0_{\auxp_0}: \mathbb A \rt \R$ when the true value of the parameter is $\auxp_0$, that is, for every $\auxp_0 \in \mathbb A$, $\loss^\vep(\cdot) \stackrel{\PP_{\auxp_0}}\Rt \loss^0_{\auxp_0}(\cdot)$ in $C(\mathbb A, \R)$. Assume that for each $\auxp_0$, $\loss^0_{\auxp_0}$ has a unique minimum at $\auxp_0$. Then Theorem \ref{th:const-gen} shows that under condition \ref{cond:est-tight} or \ref{cond:parsp}, $\hat\auxp^\vep$ is consistent, that is, it converges to the true parameter $\auxp_0$ (under $\PP_{\auxp_0}).$ The limiting function $\loss^0_{\auxp_0}$ can be thought of as an ideal loss function for the model corresponding to the true parameter $\auxp_0$ that is minimized only at $\auxp_0$. Obviously, $\loss^0_{\auxp_0}$ is not computable as it depends on the unknown parameter $\auxp_0$.

\subsection{Consistency and CLT of drift estimator} We will apply Theorem \ref{th:const-gen} to prove consistency of the estimator of the drift parameter, $\hat \drft^\vep_{\aml}$ (c.f \eqref{eq:est-pen}) of general SDE model \eqref{eq:SDE0}, under the assumption that a consistent estimator $\hat \diff^\vep$ of the parameter $\diff$ is available. We then discuss the asymptotics of $\hat \diff^\vep$ given by \eqref{eq:diff-est-1} and \eqref{eq:diff-est-1} for the two specific forms of diffusion function $\dffun$  discussed in Section \ref{sec:diff-est}.


 
 

\begin{theorem}[Consistency of the approximate maximum likelihood estimator] \label{th:const-AMLE-0}
Let $X$ be the solution to the SDE, \eqref{eq:SDE0}. For each $\vep>0$, let $\BX_{\tilde t_0:\tilde t_m} = (X(\tilde t_0), X(\tilde t_1), X(\tilde t_2), \hdots, X(\tilde t_m))$ be data from $X$ from the interval $[t_0=0, \tilde t_m \equiv \vep^{-1}]$ at the observation time points $\{\tilde t_i:i=0,1,2,\hdots,m\}$ with $\tilde \Delta \equiv \tilde\Delta(\vep) \dfeq \tilde t_i - \tilde t_{i-1}$ denoting the time gap between two successive observations. Assume that either the drift parameter space, $\dfpsp$, is open or  $\dfpsp$ is closed and each $\hat \diff^\vep \equiv \hat \diff^\vep(\BX_{\tilde t_0:\tilde t_m})$ takes values in $\dfpsp$.  Further assume that the following conditions hold.

\begin{enumerate}[label=(\Alph*), ref=(\Alph*)]
\item \label{cond:MLE-SDE} Condition \ref{cond:SDE}: \ref{cond:item:growth-rec} - \ref{cond:item:growth-lyap-b} and Condition \ref{cond:SDE-coeff}: \ref{cond:b-growth}-\hyperlink{subitem-a}{(a)},\ref{cond:b-growth}-\hyperlink{subitem-b}{(b)}, \ref{cond:b-lip-para}, \ref{cond:sig-growth}-\hyperlink{subitem-s-a}{(a)}, \ref{cond:sig-growth}-\hyperlink{subitem-s-b}{(b)}, \ref{cond:sig-growth}-\hyperlink{subitem-s-c}{(c)}  \& \ref{cond:diff-inv-lip-para};
\item \label{cond:MLE-diff-est} for each $\para_0 =(\drft_0,\diff_0)$, $\hat \diff^\vep \equiv \hat \diff^\vep(\BX_{\tilde t_0: \tilde t_m}) $ is consistent;
\item \label{cond:MLE-lik-unique} for each $\para_0$, $\bar \L_{\para_0}: \Theta \rt \R$ defined by \begin{align} \label{eq:cen-lik-lim}
\bar \L_{\para_0}(\drft) \dfeq\ \f{1}{2} \int_0^1(b(\drft,x) - b(\drft_0,x))^T a^{-1}(\diff_0,x)(b(\drft,x) - b(\drft_0,x)) \inv_{\para_0}(dx),
\end{align}
has a unique minimum at $\para=\para_0$;
\item \label{cond:MLE-penalty} $\pen^\vep \rt 0$ in $C(\drpsp, \R)$, as $\vep \rt 0$;
\item \label{cond:MLE-psp} either
\begin{enumerate}[label = (\alph*),ref=\theenumi{}-(\alph*)]
	\item \label{cond:est-tight-1}
 the $\drft$-parameter space $\drpsp$ is closed, and $\{\pest \equiv \pest(\BX_{\tilde t_0: \tilde t_m})\}$ is tight; or,
 
  \item \label{cond:parsp-1} $\drpsp$ is open, convex, and for every $\bx_{0:m} =(x_0,x_1,x_2,\hdots,x_m) \in \R^{d\times {(m+1)}}$, $\diff \in \dfpsp$, $0\leq \delta \leq 1$,  the likelihood equation 
  $$-m^{-1}\nabla_\drft \ell^{m,\delta}_{D}(\drft |\diff, \bx_{0:m})+\pen^{1/m}(\drft)=0$$
   has a unique solution.
  \end{enumerate}
\item  $\tilde \Delta(\vep) \ert 0$. 
\end{enumerate} 
Then for each $\para_0=(\drft_0,\diff_0)$, $\hat\drft^\vep \stackrel{\PP_{\theta_0}} \Rt \drft_0$ as $\vep \rt 0$.
 
\end{theorem}

\begin{remark}{\rm
 Since $\bar \L_{\para_0}(\drft) \geq 0$ for all $\drft \in \drpsp$ and $\bar \L_{\para_0}(\drft_0)=0$, $\drft_0$ is clearly a point of minimum of $\bar \L_{\para_0}.$ For uniqueness of $\drft_0$ as the point of minima, suppose that $b$ is identifiable in the parameter argument in the sense that if $b(\drft, x) = b(\drft', x)$ for a.a $x$, then $\drft=\drft'$. Suppose in addition that for all $\diff \in \dfpsp$ a.a $x$, the matrix $a(\diff, x)$ is positive definite and $\inv_{\para_0}$  is absolutely continuous with respect to $\leb$, with the invariant density $\f{d\inv_{\para_0}}{d\leb}>0$ a.s. If $\tilde \drft_0$ is another point of minimum for $\bar \L_{\para_0}(\cdot)$, then $\bar \L_{\para_0}(\tilde \drft_0) =0$, and consequently, the above assumptions imply $b(\drft_0,\cdot) = b(\tilde\drft_0,\cdot)$ a.s. The identifiability assumption on of $b$ then assures that $\drft_0 = \tilde\drft_0$.

}
\end{remark}
The following result on central limit asymptotics of the estimator of the drift term is proved in a bit more general setting. Specifically, it holds for any solution of the likelihood equation \eqref{eq:lik-crit}, provided it is consistent. Since Theorem \ref{th:const-AMLE-0} establishes consistency of $\pest$, in particular, Theorem \ref{th:aml-clt-0} below  gives the CLT for $\pest$ when it solves the likelihood equation (see assumption \ref{cond:parsp-1} above).

Recall the functions $\drbd'_{b,0}:  \drpsp \rt [0,\infty)$ and $\dibd_{a,0}: \R^{d\times d}_{>0} \rt [0,\infty)$ from Condition \ref{cond:SDE-coeff}.

\begin{theorem}[CLT for estimator of the drift parameter]\label{th:aml-clt-0}
Let $X$, $\BX_{\tilde t_0:\tilde t_m}$ be as before, $\drpsp \subset \R^{n_0}$  open and convex, and 
 $\hat\drft^\vep \equiv \hat\drft^\vep(\BX_{\tilde t_0:\tilde t_m}) $  a solution to the penalized likelihood equation, \eqref{eq:lik-crit}. Assume that either the drift parameter space, $\dfpsp$, is open or  $\dfpsp$ is closed and each $\hat \diff^\vep \equiv \hat \diff^\vep(\BX_{\tilde t_0:\tilde t_m})$ takes values in $\dfpsp$.
Further assume that
\begin{enumerate}[label=(\Alph*), ref=(\Alph*)]
\item \label{cond:CLT-SDE} Condition \ref{cond:SDE} and Condition \ref{cond:SDE-coeff}: \ref{cond:b-growth}, \ref{cond:Db-lip-para}, \ref{cond:sig-growth}-\hyperlink{subitem-s-a}{(a)}, \ref{cond:sig-growth}-\hyperlink{subitem-s-b}{(b)}, \ref{cond:sig-growth}-\hyperlink{subitem-s-c}{(c)} \& \ref{cond:diff-inv-lip-para} hold;
\item \label{cond:CLT-const} for each $\para_0 =(\drft_0,\diff_0)$, both $\hat \drft^\vep\equiv \hat \drft^\vep(\BX_{\tilde t_0:\tilde t_m})$ and $\hat \diff^\vep \equiv \hat \diff^\vep(\BX_{\tilde t_0:\tilde t_m})$ are consistent;
\item  \label{cond:CLT-bds-cont} the functions, $\drbd'_{b,0}:  \drpsp \rt [0,\infty)$ and $\dibd_{a,0}: \R^{d\times d}_{>0} \rt [0,\infty)$ are continuous;
\item \label{cond:clt-conv-penalty} $\vep^{-1/2}\nabla_{\drft} \pen^\vep \rt 0$ in $C(\drpsp, \R^{n_0})$ as $\vep \rt0$
\item $\tilde\Delta(\vep)$ is such that $\tilde\Delta(\vep)/\vep \ert 0$.
\end{enumerate}
 Then under $\PP_{\para_0}$, with $\para_0=(\drft_0,\diff_0)$, $\hat \drft^\vep \equiv \hat \drft^\vep(\BX_{\tilde t_0:\tilde t_m})$ satisfies the following: as $\vep \rt 0$
\begin{align}\label{eq:drft-clt}
\vep^{-1/2}(\hat \drft^\vep - \drft_0) \RT \No_{n_0}(0,\Sigma^{-1}_{\para_0}),
\end{align}
where 
$\dst \Sigma_{\para_0} = \int_{\R^d} D^T_\drft b(\drft_0,  x)a^{-1}(\diff_0,x)D_\drft b(\drft_0, x) \inv_{\theta_0}(dx).$
\end{theorem}

Notice that \ref{cond:clt-conv-penalty} in the above theorem holds for penalty function of the form $\pen^\vep(\drft) = \vep^{\alpha+1/2}\|\drft\|^p, \ p\geq 1, \alpha>0.$

\subsection{Consistency and CLT of diffusion estimator}\label{sec:diff-est-lim}
We will establish asymptotics of the estimator $\hat \diff^\vep$ for two specific forms of the diffusion parameter, $\diff$, discussed in Section \ref{sec:diff-est}.
\begin{theorem}\label{th:const-diff-0} Suppose the diffusion coefficient $\dffun$ is given by either Form 1 or Form 2, as mentioned in Section \ref{sec:diff-est}. Accordingly, define the estimator, $\hat\diff^\vep \equiv \hat\diff^\vep(\BX_{\tilde t_0:\tilde t_m})$, by \eqref{eq:diff-est-1} or by \eqref{eq:diff-est-2}. Assume that
\begin{enumerate}[label=(\Alph*), ref=(\Alph*)]
\item \label{cond:LLN-diff}
Condition \ref{cond:SDE}: \ref{cond:item:growth-rec} - \ref{cond:item:growth-lyap-b} and Condition \ref{cond:SDE-coeff}: \ref{cond:b-growth}-\hyperlink{subitem-a}{(a)}, \ref{cond:sig-growth}-\hyperlink{subitem-s-a}{(a)}, \ref{cond:sig-growth}-\hyperlink{subitem-s-b}{(b)} hold;
\item \label{cond:LLN-diff-stepsize} $\tilde \Delta(\vep) \ert 0$.
\end{enumerate}
Then for each $\para_0=(\drft_0,\diff_0)$, 
$\hat \diff^\vep \stackrel{\PP_{\theta_0}} \Rt \diff_0$ as $\vep \rt 0$..
\end{theorem}

\np
Recall the H\"older exponents $\holex_b, \holex_a'$ from Condition \ref{cond:SDE-coeff}.
\begin{theorem}[CLT of diffusion estimator]\label{th:clt-diff-0}
Define the estimator, $\hat\diff^\vep \equiv \hat\diff^\vep(\BX_{\tilde t_0:\tilde t_m})$, by \eqref{eq:diff-est-1} or by \eqref{eq:diff-est-2} depending on whether $\dffun$ is given by Form 1 or Form 2.
Suppose that
\begin{enumerate}[label=(\Alph*), ref=(\Alph*)]
\item \label{cond:CLT-diff} Condition \ref{cond:SDE}: \ref{cond:item:growth-rec} - \ref{cond:item:growth-lyap-b} and Condition \ref{cond:SDE-coeff}: \ref{cond:b-growth}-\hyperlink{subitem-a}{(a)}, \ref{cond:b-growth}-\hyperlink{subitem-b-lip}{(b)}, \ref{cond:sig-growth}-\hyperlink{subitem-s-a}{(a)}, \ref{cond:sig-growth}-\hyperlink{subitem-s-b}{(b)}, \ref{cond:sig-growth}-\hyperlink{subitem-s-d}{(d)}, \ref{cond:sig-growth}-\hyperlink{subitem-s-e}{(e)} hold;
\vs{.1cm}
\item \label{cond:CLT-diff-stepsize} $\tilde\Delta(\vep)^{\holex_b \wedge \holex'_a}/\vep \ert 0$;
\end{enumerate}
Let $\zeta$ be an $\R^{d\times d}$-valued random variable such that  
$$ \ve_{d\times d}(\zeta) \sim \lf(\f{1}{2}\int a(\diff_0,x) \ot a(\diff_0,x) \pi^{st}_{\para_0}(dx)\ri)^{1/2} \No_{d^2}(0,I).$$
Then the following hold for $\hat \diff^\vep \equiv \hat \diff^\vep(\BX_{\tilde t_0:\tilde t_m})$ under $\PP_{\para_0}$ as $\vep \rt 0$.
\begin{enumerate}[label = (\roman*),ref=(\roman*)]
\item \label{item:clt-diff-form1} If $\dffun$ if of Form 1, that is, $\dffun(\diff,x) = \lf(a_0(x)\diff\ri)^{1/2}$, then 
\begin{align*}
    \Delta(\vep)^{-1/2}(\hat \diff^\vep -\diff_0) \RT  2\lf(\int a_0(x)\inv_{\theta_0}(dx)\ri)^{-1}(\zeta)_{\mathrm{sym}}.
\end{align*}

\item \label{item:clt-diff-form2} If $\dffun$ if of Form 2, that is, $\dffun(\diff,x) = \dffun_0(x)\kappa$ with $\kappa\kappa^T = \diff$, then 
\begin{align*}
    \Delta(\vep)^{-1/2}(\hat \diff^\vep -\diff_0) \RT  2\lf(\int \dffun_0(x) \ot \dffun_0(x)\inv_{\theta_0}(dx)\ri)^{-1}(\zeta)_{\mathrm{sym}}.
\end{align*}
\end{enumerate}
\end{theorem}
As mentioned in the notation paragraph of the introduction, for a square matrix $A$, $(A)_{\mathrm{sym}} = \f{1}{2}(A+A^T)$.

\begin{remark}\label{rem:diff-two-forms} {\rm 
It is now instructive to describe the situations when the diffusion function $\dffun$ given by the two forms considered in Section \ref{sec:diff-est} satisfy Condition \ref{cond:SDE}:\ref{cond:sig-growth} \& \ref{cond:diff-inv-lip-para}. It is immediate that \ref{cond:SDE}:\ref{cond:sig-growth} hold for $\dffun$ if $a_0$ or $\dffun_0$ and the corresponding derivative $D_x a_0$ or $D_x \dffun_0$ (depending on whether $\dffun$ is given by Form 1 or Form 2) satisfy the growth conditions, Condition \ref{cond:SDE}:\ref{cond:sig-growth}-\hyperlink{subitem-s-a}{(a)}, \ref{cond:sig-growth}-\hyperlink{subitem-s-d}{(d)} and the H\"older-type continuity conditions (in $x$-variable), Condition \ref{cond:SDE}:\ref{cond:sig-growth}-\hyperlink{subitem-s-b}{(b)}, \ref{cond:sig-growth}-\hyperlink{subitem-s-e}{(e)}. Next note that under the assumption that $a_0$ and $\dffun_0$ are invertible, Condition \ref{cond:SDE}:\ref{cond:diff-inv-lip-para} holds for both the forms of $\dffun$, as the mapping $\diff \in \R^{d\times d}_{>0} \rt \diff^{-1} \in \R^{d\times d}_{>0} $ is locally Lipschitz continuous.
To see this, let $\dfpsp_0 \subset \R^{d\times d}_{>0}$ be a compact set. Since $\diff \in \R^{d\times d}_{\geq 0}  \rt \eig_{(1)}(\diff) \in [0,\infty)$ is continuous (e.g. see \cite[Corollary 6.3.8]{HoJo13}, \cite{Lax07}), $\eig_{(1),*}(\dfpsp_0) \equiv \inf_{\diff \in \dfpsp_0} \eig_{(1)}(\diff) >0$, where recall that $\eig_{(1)}(A)$ denotes the smallest eigenvalue of symmetric matrix $A$. Thus for $\diff_1, \diff_2 \in \dfpsp_0,$
\begin{align*}
\|\diff_1^{-1} - \diff_2^{-1}\| = \|\diff_1^{-1}(\diff_1 -\diff_2)\diff_2^{-1}\| \leq 1/\lf(\eig_{(1)}(\diff_1)\eig_{(1)}(\diff_2)\ri)\|\diff_1-\diff_2\| \leq \lf(1/\eig^2_{(1),*}(\dfpsp_0)\ri)\|\diff_1-\diff_2\|.
\end{align*}
    }
\end{remark}

\begin{example}\label{ex:SDE-lin-para}
We now consider a large class of parametric SDEs where the drift function $b$ of \eqref{eq:SDE0} is of the form $b(\mu, x) = B_0(x)\drft$, that is, the corresponding SDE is of the form
\begin{align}\label{eq:SDE-lin-para}
X(t) = X(0)+\int_0^t B_0(X(s))\drft\ ds+ \int_0^t\dffun(\diff, X(s))\ dW(s).
\end{align}	
 Here $B_0: \R^d \rt \R^{d\times n_0}$  is known, and $\drft \in \R^{n_0}$ is unknown. We assume the function $B_0$ and $\dffun$ are such that the conditions of Theorem \ref{th:const-AMLE-0} and Theorem \ref{th:aml-clt-0} hold. 
 
 Let $\hat \diff^\vep$ be an available consistent estimator of the diffusion parameter $\diff$. Then the likelihood equation \eqref{eq:lik-crit} in this case gives
\begin{align}\label{eq:lik-eq-lin-para}
\begin{aligned}
\int_0^{\vep^{-1}}B_0^T(X(\tilde{\vr_\vep}(s)))&\ a^{-1}(\hat\diff^\vep,X(\tilde{\vr_\vep}(s))) dX(s)\\
&\ = \int_0^{\vep^{-1}} B^T_0(X(\tilde\vr_\vep(s))) a^{-1}(\hat\diff^\vep,X(\tilde{\vr_\vep}(s))) B_0( X(\tilde\vr_\vep(s)))\ \drft\ ds.
\end{aligned}
\end{align}
Solving we get $\hat\drft^\vep_{\aml}$ as
\begin{align}\label{eq:SDE-lin-para-est}
 \begin{aligned}
    \hat\drft^\vep_{\aml}=&\ \lf(\int_0^{\vep^{-1}} B^T_0( X(\tilde\vr_\vep(s))) a^{-1}(\hat\diff^\vep,X(\tilde{\vr_\vep}(s))) B_0(X(\tilde\vr_\vep(s)))\  ds\ri)^{-1}\\
    & \hs{.9cm}\ \lf(\int_0^{\vep^{-1}}B_0^T(X(\tilde{\vr_\vep}(s)))a^{-1}(\hat\diff^\vep,X(\tilde{\vr_\vep}(s))) dX(s)\ri).
    \end{aligned}
\end{align}
 If the true parameter is $\para_0 = (\drft_0,\diff_0)$, then under the assumption that $\tilde \Delta(\vep) \rt 0$, $\hat\drft^\vep_{\aml}$ is consistent, and under the assumption that $\tilde \Delta(\vep)/\vep \rt 0$,  Theorem \ref{th:aml-clt-0} guarantees a CLT for $\hat\drft^\vep_{\aml}$ with  limiting covariance matrix $\Sigma^{-1}_{\para_0}$, where
$$\Sigma_{\para_0} = \int B^T_0(x)a^{-1}(\diff_0, x)B_0(x) \inv_{\para_0}(dx).$$

Notice that this setup also applies to SDEs of the form 
\begin{align}\label{eq:SDE-lin-para-2}
X(t) = X(0)+\int_0^t \mdrft \beta_0(X(s))\ ds+ \int_0^t\dffun(\diff, X(s))\ dW(s),
\end{align}
where $\mdrft \in \R^{d\times m_0}$ and $\beta_0: \R^d \rt \R^{m_0}$.
Observe that \eqref{eq:SDE-lin-para-2} can be written in the form of \eqref{eq:SDE-lin-para} with
$$B_0(x) = \beta^T_0(x)\ot I_{d\times d},\quad \drft = \ve_{d\times m_0} (\mdrft).$$
It follows from \eqref{eq:SDE-lin-para-est} that in this case
\begin{align}\label{eq:SDE-lin-para-est-2}
\begin{aligned}
\hat\drft^\vep_{\aml}=&\ \lf(\int_0^{\vep^{-1}} \beta_0( X(\tilde\vr_\vep(s)))\beta^T_0( X(\tilde\vr_\vep(s))) \ot a^{-1}(\hat\diff^\vep,X(\tilde{\vr_\vep}(s)))\ ds\ri)^{-1}\\
    & \hs{.9cm}\ \lf(\int_0^{\vep^{-1}}\beta_0(X(\tilde{\vr_\vep}(s))) \ot a^{-1}(\hat\diff^\vep,X(\tilde{\vr_\vep}(s))) dX(s)\ri),
 \end{aligned}
\end{align} 
with the inverse of the limiting covaraince matrix of the corresponding CLT given by 
\begin{align}\label{eq:SDE-lin-lim-cov}
	\Sigma_{\para_0} = \int \beta_0(x)\beta^{T}_{0}(x) \ot a^{-1}(\diff_0, x) \inv_{\para_0}(dx).
\end{align}

\np
{\bf Multidimensional OU process:} As a special case, consider the multidimensional Ornstein-Uhlenbeck (OU) SDE given by
\begin{align}\label{eq:SDE-multi-OU}
X(t) = X(0)+\int_0^t (g-HX(s))\ ds+ \kappa W(t),
\end{align}
where $g \in \R^{d}$, $H \in \R^{d\times d}$ and $\diff \in \R^{d\times d}.$ We assume that $H$ and $\diff \equiv \kappa\kappa^T$ are invertible and that the real part of eigenvalues of $H$ are positive, so that $X$ admits a stationary distribution given by $\No_{d}(H^{-1}g, F),$ where the covariance matrix $F$ satisfies the Lyapunov equation: $HF+FH^{T}=\diff.$ The diffusion matrix $\diff=\kappa\kappa^T$ can be estimated as 
\begin{align*}
\ve_{d\times d}(\hat\diff^\vep) =\ve_{d\times d} \lf(\mqd{X}^{D,\tilde\vr_\vep}_{\vep^{-1}}\ri).
\end{align*}
 Now letting $\mdrft=[g, H]$, $\drft = \ve_{d\times (d+1)}([g, H])$ and $\beta_0(x) = \begin{pmatrix}
1\\
x
\end{pmatrix}$,
$x \in \R^d$, we get from \eqref{eq:SDE-lin-para-est-2}
\begin{align*}
\hat\drft^\vep_{\aml}=&\ \lf(\int_0^{\vep^{-1}} \begin{pmatrix}
1 & X^T(\tilde{\vr_\vep}(s))\\
X(\tilde{\vr_\vep}(s)) & X(\tilde{\vr_\vep}(s))X^T(\tilde{\vr_\vep}(s))
\end{pmatrix} \ ds\ri)^{-1}\ot \hat\diff^\vep  \\
& \hs{1.9cm} \int_0^{\vep^{-1}}\begin{pmatrix}
1\\
X(\tilde{\vr_\vep}(s)) 
\end{pmatrix} \ot (\hat\diff^\vep)^{-1} \ dX(s) \\
=&\ \lf(\int_0^{\vep^{-1}} \begin{pmatrix}
1 & X^T(\tilde{\vr_\vep}(s))\\
X(\tilde{\vr_\vep}(s)) & X(\tilde{\vr_\vep}(s))X^T(\tilde{\vr_\vep}(s))
\end{pmatrix} \ ds\ri)^{-1}\int_0^{\vep^{-1}}\begin{pmatrix}
1\\
X(\tilde{\vr_\vep}(s)) 
\end{pmatrix} \ot I_{d\times d}\ dX(s).
\end{align*} 
It follows from \eqref{eq:SDE-lin-lim-cov} that the inverse of the limiting covaraince matrix of the corresponding CLT of $\hat \drft^\vep$ is given by 
\begin{align*}
\Sigma_{\para_0} = \begin{pmatrix}
1 & g^TH^{-T}\\
H^{-1}g & H^{-1}gg^TH^{-T}+F
\end{pmatrix} \ot \diff^{-1}
\end{align*}

\end{example}

\subsection{A scaling regime and equivalent formulation of the results} \label{sec:new-scaling}
For certain technical conveniences, we work under the transformation $t \rt t/\vep$. A simple change of variable formula shows  that the dynamics of  $X(\cdot/\vep)$ is same as that of the following SDE:
\begin{align}\label{eq:SDE1}
	X^\vep(t) &= x_0 + \f{1}{\vep}\int_0^t b(\drft, X^\vep(s))ds+  \f{1}{\sqrt \vep}\int_0^t\dffun(\diff, X^\vep(s)) dW(s),
\end{align}
in the sense that under $\PP_\para$ with $\para =(\drft,\diff)$, $X(\cdot/\vep) \eqd X^\vep$ in $C([0,T],\R^d)$.
In particular,  
\begin{align*}
\BX_{\tilde t_0:\tilde t_m} \equiv  (X(\tilde t_0), X(\tilde t_1), X(\tilde t_2), \hdots, X(\tilde t_m)) \eqd  (X^\vep(t_0), X^\vep(t_1), X^\vep(t_2), \hdots, X^\vep(t_m)) \equiv \BX^\vep_{t_0:t_m},
\end{align*}
where
$\BX_{\tilde t_0:\tilde t_m}$ on the left side is the data from the original process $X$ over the time interval $[0,\vep^{-1}]$ at time points $\{\tilde t_i:i=0,1,2,\hdots,m\}$ ($\tilde t_m = \vep^{-1}$) with frequency $\tilde \Delta = \tilde t_i - \tilde t_{i-1}$, and $\BX^\vep_{t_0:t_m}$ on  the right side is the data from the (scaled) process $X^\vep$ over the time interval $[0,1]$ at time points $\{t_i = \vep\tilde t_i:i=1,2,\hdots,m\}$ with frequency 
$$ \Delta(\vep) \equiv t_i-t_{i-1}= \vep(\tilde t_i - \tilde t_{i-1}) = \vep \tilde \Delta(\vep).$$
Consequently, for any measurable function $F^\vep: \R^{d\times (m+1)} \rt \R^{d'}$
\begin{align}\label{eq:eqv-diff-est}
F^\vep(\BX_{\tilde t_0:\tilde t_m}) \eqd  F^\vep(\BX^\vep_{t_0:t_m}).
\end{align}
and for  measurable functions $f_i:\R^d \rt \R^{d^f_i}$ and $g_i:\R^d \rt \R^{d_0\times d^g_i}$, $i=1,2$
\begin{align*}
& \lf( \int_{0}^{\vep^{-1}t}f_1(X(s))ds, \int_{0}^{\vep^{-1}t}f_2(X(\tilde{\vr_\vep}(s)))ds, \int_{0}^{\vep^{-1} t}g_1(X(s))dX(s), \int_{0}^{\vep^{-1} t}g_2(X(\tilde{\vr_\vep}(s)))dX(s)\ri)\\
 & = \lf(\vep^{-1}\int_{0}^{t}f_1(X(r/\vep))ds,\vep^{-1}\int_{0}^{t}f_2(X(\tilde{\vr_\vep}(r/\vep)))dr, \int_{0}^{t}g_1(X(r/\vep))dX(r/\vep),  \ri. \\
&  \hs{1.2cm}\lf. \int_{0}^{t}g_2(X(\tilde{\vr_\vep}(r/\vep)))dX(r/\vep) \ri)\\
  &\eqd  \lf(\vep^{-1}\int_{0}^{t}f(X^\vep(r))dr,  \vep^{-1} \int_{0}^{t}f(X^\vep(\vr_\vep(r)))dr, \int_{0}^{t}g(X^\vep(r))dX^\vep(r),  \int_{0}^{t}g(X^\vep(\vr_\vep(r)))dX^\vep(r)\ri),
%
\end{align*}
where as in \eqref{eq:disc-like-int-0} for $s \in [0,\vep^{-1}]$, $\tilde\vr_\vep(s) = \tilde t_i$ if $\tilde t_i \leq s < \tilde t_{i+1},$ and for $r \in [0,1]$, $\vr_\vep(r) = \tilde t_i/\vep \equiv t_i$ if $t_i \leq r < t_{i+1}$. Clearly, $\tilde\vr_\vep(r/\vep) = \vr_\vep(r)/\vep$.
In fact the equality in distribution in the previous display holds at a process level (and not merely for each fixed $t$). 

Recall the estimation of the drift parameter in Section \ref{sec:drift-est} assumed availability of a (consistent) estimator, $\hat\diff^\vep \equiv \hat\diff^\vep(\BX_{\tilde t_0:\tilde t_m})$, of the diffusion parameter $\diff$, and it was assumed that it is a measurable function of the data $\BX_{\tilde t_0:\tilde t_m}$. It readily follows from the above paragraph that under $\PP_{\para}$, for every $\vep>0$, 
\begin{align}\label{eq:sc-lik-eq}
\hat\diff^\vep(\BX_{\tilde t_0:\tilde t_m}) \eqd  \hat\diff^\vep(\BX^\vep_{t_0:t_m}), \quad \ell^\vep_{D}(\drft |\diff, \BX_{\tilde t_0:\tilde t_m}) \ \eqd \ \ell^\vep_{D}(\drft |\diff, \BX^\vep_{t_0: t_m}).
\end{align}
Here $\hat\diff^\vep(\BX_{\tilde t_0:\tilde t_m})$ and $\ell^\vep_{D}(\drft |\diff, \BX_{\tilde t_0:\tilde t_m})$ are respectively the estimator of the parameter $\diff$ and the likelihood function of $\drft$ (given by \eqref{eq:disc-like-int-0}) based on $\BX_{\tilde t_0:\tilde t_m}$ (recall $\tilde t_m= \vep^{-1}$), the data from the original process $X$ over the time interval $[0,\vep^{-1}]$, while $\hat\diff^\vep(\BX^\vep_{t_0:t_m})$ and $\ell^\vep_{D}(\drft |\diff, \BX^\vep_{t_0: t_m})$ are, respectively, the corresponding estimator of $\diff$ and the likelihood function of $\drft$ based on $\BX^\vep_{t_0:t_m}$, the data from the process $X^\vep$ over the time interval $[0,1]$. $\ell^\vep_{D}(\drft |\diff, \BX^\vep_{t_0: t_m})$ is of course given by
\begin{align}\label{eq:approx-like-int}
\begin{aligned}
\ell^\vep_{D}(\drft |\diff,\BX^\vep_{t_0: t_m}) =&\ \int_0^{1}  a(\diff,X^\vep(\vr_\vep(s))^{-1}b(\drft, X^\vep(\vr_\vep(s)))\cdot dX^\vep(s)\\
& \hs{.2cm}  - \f{1 }{2\vep} \int_0^{1} b^T(\drft, X^\vep(\vr_\vep(s)))a^{-1}(\diff,X^\vep(\vr_\vep(s)))b(\drft, X^\vep(\vr_\vep(s))) ds.
\end{aligned}
\end{align}
Here by a slight abuse of notation we continued to use $\hat\diff^\vep$ and $\ell^\vep_D$ to denote the estimator of $\diff$ and the likelihood corresponding to data from the (scaled) process $X^\vep$.

In the same spirit, we continue to denote by $\pest= \pest(\BX^\vep_{t_0: t_m})$ the AMLE corresponding to the (penalized) likelihood of data from $X^\vep$, that is,
\begin{align}\label{eq:est-pen-scaled}
\pest(\BX^\vep_{t_0: t_m}) \in&\ \argmin_{\drft \in \drpsp} \lf\{-\vep\ell^\vep_{A}(\drft|\BX^\vep_{t_0: t_m}))+ \pen^\vep(\drft)\ri\}
\end{align}
 with $\ell^\vep_A$ now obtained by replacing $\diff$ by $\hat \diff^\vep = \diff^\vep(\BX^\vep_{t_0:t_m})$ in \eqref{eq:approx-like-int}, that is, $\ell^\vep_{A}(\drft |\BX^\vep_{t_0: t_m}) \dfeq \ell^\vep_{D}(\drft |\hat\diff^\vep,\BX^\vep_{t_0: t_m})$. \eqref{eq:sc-lik-eq} immediately implies that
$$ \pest(\BX_{\tilde t_0:\tilde t_m})\ \eqd \ \pest(\BX^\vep_{t_0: t_m})$$
when comparing the corresponding points of minimum.
As before, when the mappings the mappings $\drft \rt \pen^\vep(\drft)$ and $\drft \rt b(\drft,x)$ for each $x \in \R^d$ are differentiable, and $\drpsp$ is open and convex, $\pest(\BX^\vep_{t_1: t_m})$ is a critical point of $\loss^\vep(\drft)$, that is, the solution of the penalized likelihood equation
\begin{align}\label{eq:lik-crit-scaled}
0=\nabla_{\drft}\loss^\vep_A(\drft | \BX^\vep_{t_0: t_m}) \equiv - \vep \nabla_{\drft}\ell^\vep_{A}(\drft|\BX^\vep_{t_0: t_m})+\nabla_\drft \pen^\vep(\drft),
\end{align}
with $\nabla_{\drft}\ell^\vep_{A}(\drft|\BX^\vep_{t_0: t_m})$ now given by
\begin{align}\label{eq:deriv-like}
\begin{aligned}
\nabla_{\drft}\ell^{\vep}_A(\drft|\BX^\vep_{t_0: t_m}) =&\ \int_0^{1}D^T_\drft b(\drft, X^\vep(\vr_\vep(s))) a^{-1}(\hat\diff^\vep, X^\vep(\vr_\vep(s))) dX^\vep(s)\\
	& \hs{.2cm} -\f{1}{\vep}\int_0^{1} D^T_\drft b(\drft, X^\vep(\vr_\vep(s))) a^{-1}(\hat\diff^\vep,X^\vep(\vr_\vep(s))) b(\drft, X^\vep(\vr_\vep(s)))\ ds.
\end{aligned}
\end{align}

It is now clear that  Theorem \ref{th:const-AMLE-0} and Theorem \ref{th:aml-clt-0} are equivalent to Theorem \ref{th:const-AMLE} and Theorem \ref{th:aml-clt} below, respectively, and the latter two results are what we prove in this paper.

\begin{theorem} \label{th:const-AMLE}
For each $\vep>0$, let $X^\vep$ be the solution to the SDE, \eqref{eq:SDE1}, and let $\BX^\vep_{ t_0:t_m} = (X^\vep(t_0), X^\vep(t_1), X^\vep(t_2), \hdots, X^\vep(t_m))$ be data from $X^\vep$ from the interval $[0, 1]$ at the observation time points $\{ t_i:i=1,2,\hdots,m\}$ with $ \Delta \equiv \Delta(\vep) \dfeq t_i - t_{i-1}$ denoting the time gap between two successive observations.  Assume that the assumption \ref{cond:MLE-SDE}, \ref{cond:MLE-diff-est} (with $\hat \diff^\vep \equiv \hat \diff^\vep(\BX^\vep_{t_0: t_m})$), \ref{cond:MLE-lik-unique}, \ref{cond:MLE-penalty}, \ref{cond:MLE-psp} (with $\pest \equiv \pest(\BX^\vep_{t_0: t_m})$ in part \ref{cond:est-tight-1}) of Theorem \ref{th:const-AMLE-0} hold, and $\dfpsp$ satisfies the condition of Theorem \ref{th:const-AMLE-0}.
Further assume that $\Delta(\vep)/\vep \ert 0$. Then for each $\para_0$, $\pest \stackrel{\PP_{\theta_0}} \Rt \drft_0$ as $\vep \rt 0$.
\end{theorem}

\begin{theorem}\label{th:aml-clt}
Let $\hat\drft^\vep \equiv \hat\drft^\vep(\BX^\vep_{ t_0:t_m}) $ be a solution to the penalized likelihood equation, \eqref{eq:lik-crit-scaled}.  Assume that the assumptions \ref{cond:CLT-SDE}, \ref{cond:CLT-const} (with $\hat\drft^\vep \equiv \hat\drft^\vep(\BX^\vep_{ t_0:t_m})$ and $\hat\diff^\vep \equiv \hat\diff^\vep(\BX^\vep_{ t_0:t_m})$), \ref{cond:CLT-bds-cont} and \ref{cond:clt-conv-penalty}  of Theorem \ref{th:aml-clt-0} hold, and $\drpsp$ and $\dfpsp$ satisfy the conditions of Theorem \ref{th:aml-clt-0}. Further assume that $\Delta(\vep)$ is such that $\Delta(\vep)/\vep^2\rt 0$ as $\vep \rt 0$. Then under $\PP_{\para_0}$, $\hat\drft^\vep \equiv \hat\drft^\vep(\BX^\vep_{ t_0:t_m})$ satisfies \eqref{eq:drft-clt}.

\end{theorem}

We now consider the estimation of diffusion parameter $\diff$ in this scaling regime for the two forms mentioned in Section \ref{sec:diff-est}. Similar to the discretized quadratic variation of $X$ introduced in \eqref{eq:quad-disc-X}, for each $\vep>0$, let $\mqd{X^\vep}^{D,\vr_\vep}$ denote the discretized matrix-quadratic variation of the process $X^\vep$ defined by
 \begin{align}\label{eq:quad-disc-X-ep}
 \mqd{X^\vep}^{D,\vr_\vep}_t \dfeq \sum_{i=0}^{[t/\Delta]-1} (X^\vep(t_{i+1}) - X^\vep(t_i))(X^\vep(t_{i+1}) - X^\vep(t_i))^T.
 \end{align}
It is immediate that  
 $$\mqd{X}^{D,\tilde\vr_\vep}_{\sbullet/\vep} \ \eqd \ \mqd{X^\vep}^{D,\vr_\vep}_{\sbullet}.$$
Hence for Form 1 of $\dffun$, specifically, for $\dffun(\diff,x) = \lf(a_0(x)\diff\ri)^{1/2}$,
\begin{align}\label{eq:diff-est-1-sc}
\begin{aligned}
\hat\diff^\vep(\BX_{\tilde t_0:\tilde t_m}) = &\ \lf(\int_0^{\vep^{-1}} a_0(X(\tilde\vr_\vep(s)))ds\ri)^{-1}\mqd{X}^{D,\tilde\vr_\vep}_{\vep^{-1}}\\
\eqd&\ \lf(\vep^{-1}\int_0^{1} a_0(X^\vep(\vr_\vep(s)))ds\ri)^{-1}\mqd{X^\vep}^{D,\vr_\vep}_1 \equiv \hat\diff^\vep(\BX^\vep_{ t_0: t_m}),
\end{aligned}
\end{align}
while for Form 2 of $\dffun$, that is, for $\dffun(\diff,x) = \dffun_0(x)\kappa$ with $\diff=\kappa\kappa^T$,
\begin{align}\label{eq:diff-est-2-sc}
\begin{aligned}
\ve_{d}\lf(\hat \diff^\vep(\BX_{\tilde t_0:\tilde t_m})\ri) =&\ \lf(\int_0^{\vep^{-1}} \dffun_0(X(\tilde\vr_\vep(s))) \ot \dffun_0(X(\tilde\vr_\vep(s)))ds\ri)^{-1} \ve_{d} \lf(\mqd{X}^{D,\tilde\vr_\vep}_{\vep^{-1}}\ri)\\
\eqd&\ \lf(\vep^{-1}\int_0^{1} \dffun_0(X^\vep(\vr_\vep(s))) \ot \dffun_0(X^\vep(\vr_\vep(s)))ds\ri)^{-1} \ve_{d} \lf(\mqd{X^\vep}^{D,\vr_\vep}_{1}\ri)\\
\equiv &\ \ve_d\lf(\hat\diff^\vep(\BX^\vep_{ t_0: t_m})\ri).
\end{aligned}
\end{align}

As in the case of drift estimation, it is again obvious that proving Theorem \ref{th:const-diff-0} and Theorem \ref{th:clt-diff-0} is equivalent to showing that those respective assertions hold for $\hat \diff^\vep \equiv \hat\diff^\vep(\BX^\vep_{ t_0: t_m})$ with $\tilde \Delta(\vep)$ replaced by $\Delta(\vep)/\vep$ in the conditions of those theorems.

\section{Auxiliary Results}\label{sec:aux-results}

\begin{proposition}\label{prop:int-bd}
	Suppose that  $ X^\vep$ satisfies \eqref{eq:SDE1} and that $\Delta(\vep)/\vep \rt 0$. Assume that 
 \begin{itemize}
 \item the drift and the diffusion functions, $b$ and $a=\dffun\dffun^T$, are locally bounded;
 \item Condition \ref{cond:SDE}-\ref{cond:item:growth-rec} \& \ref{cond:item:growth-lyap-a}  hold for $\para_0 =(\drft_0,\diff_0) \in \drpsp\times R^{d\times d}_{>0}$. 
\end{itemize}
 Then for all $p\geq 0$
	$$\sup_{0<\vep \leq 1} \EE_{\para_0}\lf[\int_0^1 V^p(X^\vep(t) dt \ri] <\infty. $$
\end{proposition}

\begin{proof} 

Let $p\geq 1$. By It\^o's lemma for $0\leq t \leq 1$,
	\begin{align} \non
		V^{p+1}(X^\vep(t)) =&\ V^{p+1}(x_0)+\int_0^t (p+1) V^{p}(X^\vep(s))\nabla^TV(X^\vep(s))dX^\vep(s)\\ \non
		&\ + \tr \int_0^t (p+1)V^{p}(X^\vep(s))D^2V(X^\vep(s))d[X^\vep](s)\\ \non
		&\ +\tr \int_0^t (p+1)pV^{p-1}(X^\vep(s))\nabla V(X^\vep(s))\nabla^TV(X^\vep(s))d\mqd{X^\vep}(s)\\ \non
		=&\ V^{p+1}(x_0)+\vep^{-1/2}\mart^\vep_p(t)+ \vep^{-1}\int_0^t (p+1) V^{p}(X^\vep(s)) \SC{L}_{\para_0}V(X^\vep(s)) ds\\ \label{eq:v-ito}
		&\ + \vep^{-1}\tr \int_0^t (p+1)pV^{p-1}(X^\vep(s))\nabla V(X^\vep(s))\nabla^TV(X^\vep(s))a(\diff_0, X^\vep(s))ds,
	\end{align}
where $\mart^\vep_p(t) \equiv \int_0^t  (p+1) V^{p}(X^\vep(s))\nabla^TV(X^\vep(s))\dffun(\diff_0,X^\vep(s))dW(s)$
is a martingale.  By splitting  the last term (without the factor $\vep^{-1}$) according as $\|X^\vep(s)\| >R_{2p}(\theta_0)$ (c.f. Remark \ref{rem:SDE-cond})  or not, it can be estimated as follows:
\begin{align*}
A^\vep(t) \equiv&\ \tr \int_0^t (p+1)pV^{p-1}(X^\vep(s))\nabla V(X^\vep(s))\nabla^TV(X^\vep(s))a(\diff_0, X^\vep(s))ds\\
= &\  \int_0^t (p+1)pV^{p-1}(X^\vep(s))\nabla^TV(X^\vep(s))a(\diff_0,X^\vep(s))\nabla V(X^\vep(s)) ds \\
\leq &\ \const_{\ref*{prop:int-bd},0}(p) - \f{1}{2}\int_0^t (p+1)V^{p}(X^\vep(s))\SC{L}_{\para_0} V(X^\vep(s))1_{\{\|X^\vep(s)\| \geq R_{2p}(\para_0)\}}
\end{align*}
for some constant, $\const_{\ref*{prop:int-bd},0}(p)$. Thus rearranging and multiplying multiplying throughout by $\vep$ we get,
\begin{align*}
-\f{1}{2}\EE_{\para_0}\int_0^t  (p+1) V^{p}&(X^\vep(s)) \SC{L}_{\para_0}V(X^\vep(s))  1_{\{\|X^\vep(s)\| \geq R_{2p}(\para_0)\}} \leq\  \vep V^{p+1}(x_0)+ \const_{\ref*{prop:int-bd},0}(p).
\end{align*}
Now splitting the integral on the left according as $\|X^\vep(s)\| >R_{0}(\para_0) \vee R_{2p}(\para_0) $ we get (see Remark \ref{rem:SDE-cond}-\ref{rem:SDE-gen})   
\begin{align*}
\EE_{\para_0}\int_0^1   V^{p}(X^\vep(s)) 1_{\{\|X^\vep(s)\| \geq R_{0}(\para_0) \vee R_{2p}(\para_0)\}} ds \leq  \f{2( V^{p+1}(x_0)+ \const_{\ref*{prop:int-bd},0}(p))}{(p+1)\lfun_*(\theta_0)}.
\end{align*}
This proves the assertion.
\end{proof}	

\subsection{Auxiliary Lemma on $\psi$ and $\phi$}
Unless otherwise specified, throughout we will fix the parameter value, $\para_0 = (\drft_0,\diff_0)$, and work on the probability space $(\Omega, \SC{F}, \PP_{\para_0})$.
%
%
%
%

\begin{lemma} \label{lem:fun-diff-est}
Suppose that  $ X^\vep$ satisfies \eqref{eq:SDE1}. Assume that Condition \ref{cond:SDE}-\ref{cond:item:growth-rec} \& \ref{cond:item:growth-lyap-a}   and Condition \ref{cond:SDE-coeff}: \ref{cond:b-growth}-\hyperlink{subitem-a}{(a)} \& \ref{cond:sig-growth}-\hyperlink{subitem-s-a}{(a)} hold for $\drft_0, \diff_0$. 
Let $\psi:\R^{d} \rt \R^{d_1}$  be a twice differentiable function  satisfying the following growth condition:
$$\max\lf\{ \|D\psi(x)\|, \|D^2\psi(x)\|\ri\}  \leq J_0 V(x)^{\gexp_0}$$
for some constant $J_0>0$ and exponent $\gexp_1\geq 0$.
Then for any $m>0$ the following hold: for some constant $\consta_{1,m}$ 
$$\EE_{\para_0}\lf[\int_0^1\|\psi(X^\vep(s)) - \psi(X^\vep(\vr_\vep(s))) \|^m ds\ri] \leq \consta_{1,m}  (\Delta(\vep)/\vep)^{m/2}.$$
\end{lemma}

\begin{proof}

We only consider the case $m\geq 1$. The assertion for $0<m<1$ easily follows from the previous case.
First observe that the condition on $\psi$ together with the growth assumptions on $b$ and $a$ (see  Condition \ref{cond:SDE-coeff}: \ref{cond:b-growth}-\hyperlink{subitem-a}{(a)} \& \ref{cond:sig-growth}-\hyperlink{subitem-s-a}{(a)}) implies that for some constant $\const_{\ref*{lem:fun-diff-est},0}$ and exponent $q_1 = \gexp_0(\bexp_{b,0}+2\sexp_{\dffun,0})$
\begin{align}\label{eq:mmt-gen}
\|\SC{L}_{\para_0} \psi(x)\|\leq \const_{\ref*{lem:fun-diff-est},0} V(x)^{\gexp_1},
\end{align}
where 
$\SC{L}_{\para_0}\psi(x) = \lf(\SC{L}_{\para_0}\psi_1(x)), \SC{L}_{\para_0}\psi_2(x), \hdots, \SC{L}\psi_{d_1}(x)\ri)^T.$
Now to estimate the term, $\SC{R}^{m,\vep}_0 \equiv  \int_0^ 1 \|\psi(X^\vep(s)) - \psi(X^\vep(\vr_\vep(s)))\|^m\ ds,$ simply  observe that by It\^o's lemma (applied to each component of $\psi$),
\begin{align}\label{eq:psi-ito}
\psi(X^\vep(s)) - \psi(X^\vep(\vr_\vep(s))) = \f{1}{\vep}\int_{\vr_\vep(s)}^s \SC{L}_{\para_0}\psi(X^\vep(r))\ dr+\f{1}{\sqrt \vep} \int_{\vr_\vep(s)}^s D \psi(X^\vep(r)) \s(X^\vep(r)) dW(r)
\end{align}
 By H\"older's inequality, Lemma \ref{lem:simp-int-bd} and the assumptions on $\psi, b$ and $\dffun$, the quantity
$\SC{A}^{m,\vep} \equiv \int_0^1 \lf\| \int_{\vr_\vep(s)}^s \SC{L}_{\para_0}\psi(X^\vep(r))\ dr\ri\|^m \ ds$ can be estimated as
\begin{align*}
\SC{A}^{m,\vep} \leq &\ \Delta^{m-1}  \int_0^1 \int_{\vr_\vep(s)}^s \lf\|\SC{L}_{\para_0}\psi(X^\vep(r))\ri\|^m\ dr \ ds \ 
\leq  \Delta^{m} \int_0^1  \lf\|\SC{L}_{\para_0}\psi(X^\vep(s))\ri\|^m \ ds\\
\leq &\ \Delta^m \const_{\ref*{lem:fun-diff-est},1,m} \int_0^1V(X^\vep(s))^{m\gexp_1} \ ds.
\end{align*}
Similar calculations along with an additional use of Burkholder-Davis-Gundy (BDG) inequality show that the term $M^\vep_0 = \int_0^1\lf\|\int_{\vr_\vep(s)}^s D\psi(X^\vep(r)) \dffun(\diff, X^\vep(r)) dW(r)\ri\|^m ds $ can be estimated as
\begin{align*}
\EE_{\para_0}\lf[M^\vep_0\ri] \leq&\  \Delta^{m/2} \const_{\ref*{lem:fun-diff-est},2,m}  \EE_{\para_0} \int_0^1V(X^\vep(s))^{m\gexp_2} \ ds
\end{align*}
for some constant $\const_{\ref*{lem:fun-diff-est},2,m}$ and exponent $\gexp_2 = 2\gexp_0 \sexp_{\dffun,0}$. 
Thus by Proposition \ref{prop:int-bd} we get from \eqref{eq:psi-ito} that 
\begin{align*}
\EE_{\para_0}[\SC{R}_0^{m,\vep} ] \leq&\ \const_{\ref*{lem:fun-diff-est},3,m} \lf(\f{\Delta}{\vep}\ri)^{m/2}\EE_{\para_0} \int_0^1V(X^\vep(s))^{m(\gexp_1\vee \gexp_2)} \ ds 
\leq  \consta_{1,m} \lf(\f{\Delta}{\vep}\ri)^{m/2},
\end{align*}
where $\const_{\ref*{lem:fun-diff-est},3,m}$ and $ \consta_{1,m}$ are suitable constants.  
\end{proof}

\begin{remark}\label{rem:x-diff-est} {\rm
In particular, the above result holds for $\psi(x) =x$, that is, under Condition \ref{cond:SDE}-\ref{cond:item:growth-rec} \& \ref{cond:item:growth-lyap-a}   and Condition \ref{cond:SDE-coeff}: \ref{cond:b-growth}-\hyperlink{subitem-a}{(a)} \& \ref{cond:sig-growth}-\hyperlink{subitem-s-a}{(a)}, for any $m>0$,
$$\EE_{\para_0}\lf[\int_0^1\|X^\vep(s) - X^\vep(\vr_\vep(s)) \|^m ds\ri] \leq  \consta_{1,m} (\Delta/\vep)^{m/2},$$
for some constant $\consta_{1,m}$.
}
\end{remark}

The following result immediately follows from Proposition \ref{prop:int-bd} and Lemma \ref{lem:fun-diff-est}. 
%
\begin{corollary}\label{cor:int-bd-V}
Suppose that Condition \ref{cond:SDE}: \ref{cond:item:growth-rec} - \ref{cond:item:growth-lyap-b} and Condition \ref{cond:SDE-coeff}:\ref{cond:b-growth}-\hyperlink{subitem-a}{(a)} \& \ref{cond:sig-growth}-\hyperlink{subitem-s-a}{(a)} hold. 
Then for any exponent $p>0$ and $\para_0 \in \Theta$
\begin{align*}
	\sup_{0<\vep\leq 1} \EE_{\para_0}\lf[\int_0^T V^p(X^\vep(\vr_\vep(t))) dt \ri] <\infty.
\end{align*}
In particular, 
$$V_{p, \para_0}^* \dfeq  \sup_{0<\vep \leq 1} \lf(\EE_{\para_0}\lf[\int_0^1 V^p(X^\vep(t) dt \ri] \vee \EE_{\para_0}\lf[\int_0^T V(X^\vep(\vr_\vep(t)))^p dt \ri]\ri) < \infty. $$
\end{corollary}	

Notice that the above corollary, which needed $V$ to satisfy the additional assumption of Condition \ref{cond:SDE}-\ref{cond:item:growth-lyap-b} compared to Lemma \ref{lem:fun-diff-est}, helps to extend that lemma to cover the functions $\psi$ which might not be differentiable but satisfy a local H\"older type condition. This is summarized in the following corollary whose proof  is an immediate consequence of Corollary \ref{cor:int-bd-V}, Remark \ref{rem:x-diff-est} and H\"older's inequality.
 In fact, the corollary below considers functions $\psi$ and $\phi$ that depend on a parameter $\apar$ and notes the dependency of the bounding constants on the parameter $\apar$.  This can be easily done by tracking the constants in the proof of Proposition \ref{prop:int-bd} and Lemma \ref{lem:fun-diff-est}. Such estimates involving functions depending on a parameter are necessary for establishing our main results. 

\begin{corollary}\label{cor:diff-est-1}
Let $\phi: \R^{d'}\times \R^d \rt \R^{d_0 \times d_1}$ and $\psi: \R^{d'}\times\R^{d} \rt \R^{d_1}$ be two functions. Suppose that Condition \ref{cond:SDE}: \ref{cond:item:growth-rec} - \ref{cond:item:growth-lyap-b} and Condition \ref{cond:SDE-coeff}:\ref{cond:b-growth}-\hyperlink{subitem-a}{(a)} \& \ref{cond:sig-growth}-\hyperlink{subitem-s-a}{(a)} hold. Fix a parameter  $\apar \in \R^{d'}$.

\begin{enumerate}[label=(\roman*), ref=(\roman*)]
\item \label{item:bd-est:bd-1} Suppose $\phi$ and $\psi$ satisfy the following growth condition: for a function  $J_0: \R^{d'} \rt [0,\infty)$,
$$\max\lf\{\|\phi(\apar, x)\|,\|\psi(\apar, x)\|\ri\}  \leq J_0(\apar) V(x)^{\gexp_0}.$$
Then for some constant (not depending on $\apar$) $\consta^1_{m,m_0}\geq 0$,
 \begin{align*}
&\ \EE_{\para_0}\lf[\int_0^1\|\phi(\apar, X^\vep(\vr_\vep(s)))\|^{m_0}\|\psi(\apar, X^\vep(s))\|^m ds\ri] \leq \consta^1_{m,m_0} J_0^{m+m_0}(\apar).
\end{align*}

\item \label{item:diff-est:diff-1} Suppose $\phi$ satisfies the growth condition  in \ref{item:bd-est:bd-1}, and $\psi$ satisfies the following H\"older type condition:  for some function   $J_1: \R^{d'} \rt [0,\infty)$, H\"older exponent $0<\holex_{\psi}\leq 1,$ and growth exponent $\gexp_1\geq 0$ 


$$ \|\psi(\apar, x) - \psi(\apar, x')\| \leq  J_1(\apar) (V(x)^{\gexp_1}+V(x')^{\gexp_1})\|x-x'\|^{\holex_{\psi}}.$$

\vs{.2cm}
\np
Then for any $m>0, m_0>0$,
 \begin{align*}
 \EE_{\para_0}\Big[\int_0^1\|\phi(\apar, X^\vep(\vr_\vep(s)))\|^{m_0} & \|\psi(\apar, X^\vep(s))- \psi(\apar, X^\vep(\vr_\vep(s)))\|^m ds\Big]\\
 \leq &\ \consta^2_{m,m_0} J^{m_0}_0(\apar)J_1^{m}(\apar)(\Delta/\vep)^{m\holex_{\psi}/2},
\end{align*}
where $\consta^2_{m,m_0} \geq 0$ is a constant not depending on $\apar$.

\item \label{item:diff-est:diff-2}  Let $\SC{B} \subset \R^{d'}$ be open, and suppose $\{\tilde \apar^\vep\}$ is a family of $\R^{d'}$-valued random variables such that $\tilde \apar^\vep \prt \tilde \apar^0$ as $\vep \rt 0$, where the limiting random variable $\tilde \apar^0$ takes values in $\SC{B}$. Suppose that for any $\apar' \in \SC{B}$, $\phi$ satisfies the growth condition  in \ref{item:bd-est:bd-1} and $\psi$ satisfies the H\"older type condition in \ref{item:diff-est:diff-1}, for a.a $\om$, the mapping $\apar' \in \SC{B} \rt J_0(\apar')$ is locally bounded. 
 Further assume that $\Delta(\vep)$ is such that $\Delta(\vep)^{\holex_{\psi}}/\vep^{1+\holex_{\psi}} \rt 0$.
Then  as $\vep \rt 0,$
\begin{align*}
\vep^{-1/2}\int_0^1\|\phi(\tilde \apar^\vep, X^\vep(\vr_\vep(s)))\|\|\psi(\apar, X^\vep(s))- \psi(\apar, X^\vep(\vr_\vep(s)))\| ds \prt 0.
\end{align*}
\end{enumerate}
\end{corollary}

The decay rate of $(\Delta/\vep)^{m/2}$ in the bounds in the first half of the above corollary is however not enough for CLT of the estimator of diffusion-parameter. We need a finer version giving better decay rate which is the content of Lemma \ref{lem:fun-diff-est-2}. This however requires more careful analysis. We first prove the following lemma. 

\begin{lemma} \label{lem:disc-mart}
Suppose that  $ X^\vep$ satisfies \eqref{eq:SDE1}. Assume that Condition \ref{cond:SDE}: \ref{cond:item:growth-rec} - \ref{cond:item:growth-lyap-b} and Condition \ref{cond:SDE-coeff}: \ref{cond:b-growth}-\hyperlink{subitem-a}{(a)} \& \ref{cond:sig-growth}-\hyperlink{subitem-s-a}{(a)} hold for $\drft_0, \diff_0$. 
Let  $g:\R^{d} \rt \R^{d_1\times d}$ satisfying the following growth condition: $\|g(x)\|  \leq \consta_{g,0} V(x)^{\gexp_{g,0}}$ for some constant $\consta_{g,0}>0$ and exponent $\gexp_{g,0}\geq 0$ .

Define the process $M_{g}^\vep$ by
\begin{align*}
M_{g}^\vep(t) = \int_0^t g(X^\vep(\vr_\vep(s))) \lf(W(s) - W(\vr_\vep(s))\ri) ds.
\end{align*}
Then the following hold.
\begin{enumerate}[label=(\roman*), ref=(\roman*)]
\item \label{item:g-mart} The induced discrete-time process $\lf\{M^\vep_{g}(t_k): k=0,1,2,\hdots,m\ri\}$ is a (discrete-time) martingale.
\item \label{item:g-mart-bd} For some constant $\consta_{g,1,m} \geq 0$
 \begin{align*}
\EE_{\para_0}\lf(\lf\|\int_0^t g(X^\vep(\vr_\vep(s))) \lf(W(s) - W(\vr_\vep(s))\ri) ds\ri\|^m\ri)& \leq \consta_{g,1,m} \Delta(\vep)^m, \quad 0\leq t\leq 1.
\end{align*}
Under the additional assumption that $g$ satisfies the following H\"older type continuity: for some constant $\consta_{g,1} \geq 0$ and exponents $\gexp_{g,1} \geq 0$, $0\leq \holex_{g} \leq 1$
$$ \|g(x) - g(x')\| \leq  \consta_{g,1}  (V(x)^{\gexp_{g,1}}+V(x')^{\gexp_{g,1}})\|x-x'\|^{\holex_{g}},$$ 
we have for some constant $\consta_{g,2,m} \geq 0$
\begin{align*}
\EE_{\para_0}\lf(\lf\|\int_0^t g(X^\vep(s)) \lf(W(s) - W(\vr_\vep(s))\ri) ds\ri\|^m\ri)& \leq \consta_{g,2,m} \Delta(\vep)^{m(1+\holex_g)/2}/\vep^{m/2}.\
 \end{align*}
\end{enumerate}
\end{lemma}

\begin{proof}
\ref{item:g-mart} Notice that
\begin{align*}
\EE_{\para_0}(M_{g}^\vep(t_{k+1})|\SC{F}_{t_k}) =&\ M_{g}^\vep(t_k) +  \EE_{\para_0} \lf(\int_{t_k}^{t_{k+1}}g(X^\vep(\vr_\vep(s))) \lf(W(s) - W(\vr_\vep(s))\ri) ds \Big|\SC{F}_{t_k}\ri)\\
=&\ \ M_{g}^\vep(t_k) + g(X^\vep(t_k)) \int_{t_k}^{t_{k+1}}\EE_{\para_0}\lf(W(s) - W(t_k) \big|\SC{F}_{t_k}\ri)\ ds\\
=&\ M_{g}^\vep(t_k)+0= M_{g}^\vep(t_k).
\end{align*}

\np
\ref{item:g-mart-bd} First notice that it is enough to prove that the estimate holds for any partition point $t=t_n$, $n\leq m =1/\Delta(\vep)$. By the BDG inequality on the martingale $\{M^\vep_{g}(t_k)\}$ followed by an application of H\"older's inequality, we get for any $n \leq m$
\begin{align*}
\EE_{\para_0}\lf(\|M^\vep_{g}(t_n)\|^m\ri) \leq&\ \const_{\ref*{lem:disc-mart},0,m} \EE_{\para_0}\lf(\sum_{k=0}^{n-1}\lf\|\int_{t_k}^{t_{k+1}}g(X^\vep(\vr_\vep(s))) \lf(W(s) - W(\vr_\vep(s))\ri) ds\ri\|^2\ri)^{m/2}\\
\leq&\ \const_{\ref*{lem:disc-mart},1,m} n^{\f{m}{2}-1} \sum_{k=0}^{n-1}\EE_{\para_0} \lf(\lf\|g(X^\vep(t_k))\int_{t_k}^{t_{k+1}} \lf(W(s) - W(t_k)\ri) ds\ri\|^m\ri)\\
\leq &\ \const_{\ref*{lem:disc-mart},2,m} \Delta(\vep)^{-(\f{m}{2}-1)} \Delta(\vep)^{m-1} \\
&\ \hs{.9cm} \times \sum_{k=0}^{n-1} \EE_{\para_0}\lf[\|g(X^\vep(t_k))\|^m \int_{t_k}^{t_{k+1}} \EE_{\para_0}\lf(\lf\|W(s) - W(t_k)\ri\|^m \big | \SC{F}_{t_k}\ri)ds\ri]\\
\leq &\ \const_{\ref*{lem:disc-mart},3,m} \Delta(\vep)^{-(\f{m}{2}-1)} \Delta(\vep)^{m-1}  \Delta(\vep)^{m/2} \EE_{\para_0}\int_0^1 \|g(X^\vep(\vr_\vep(s)))\|^m \ ds
\leq \consta_{g,1,m} \Delta(\vep)^m.
\end{align*}
The last step above used Lemma \ref{lem:fun-diff-est} (also see Corollary \ref{cor:int-bd-V}) and the assumption on $g$.
For the second bound, write
\begin{align*}
\int_0^t g(X^\vep(s)) \lf(W(s) - W(\vr_\vep(s))\ri) ds = M^\vep_g(t)+R_g^\vep(t),
\end{align*}
where 
$$R_g^\vep(t) = \int_0^t \lf(g(X^\vep(s)) - g(X^\vep(\vr_\vep(s))\ri) \lf(W(s) - W(\vr_\vep(s))\ri) ds.$$
By H\"older's inequality and Corollary \ref{cor:diff-est-1}-\ref{item:diff-est:diff-1},
\begin{align*}
\EE_{\para_0}\lf(\|R_g^\vep(t)\|^m\ri)\leq&\ \lf(\EE_{\para_0}\int_0^t\| g(X^\vep(s)) - g(X^\vep(\vr_\vep(s))\|^{m/(m-1)}\ ds\ri)^{m-1}\EE_{\para_0}\int_0^t\|W(s) - W(\vr_\vep(s))\|^m\ ds\\
\leq&\ \consta_{g,2,m}(\Delta(\vep)/\vep)^{m\holex_g/2}\Delta(\vep)^{m/2},
\end{align*}
for some constant $\consta_{g,2,m}$.
This along with the first bound establishes the second one.
\end{proof}

\begin{lemma} \label{lem:fun-diff-est-2}
Suppose that  $ X^\vep$ satisfies \eqref{eq:SDE1}. Assume that Condition \ref{cond:SDE}  and Condition \ref{cond:SDE-coeff}: \ref{cond:b-growth}-\hyperlink{subitem-a}{(a)} \& \ref{cond:sig-growth}-\hyperlink{subitem-s-a}{(a)}, \ref{cond:sig-growth}-\hyperlink{subitem-s-b}{(b)}  hold for $\para_0=(\drft_0, \diff_0)$. 
Let $\phi:\R^d \rt \R^{d_0\times d_1}$ and $\psi:\R^{d} \rt \R^{d_1}$  be two functions  satisfying the following conditions: for some constants $J_0, J'_1$ and exponents $\gexp0, \gexp_1 \geq 0$
\begin{itemize}
\item $\psi$ is differentiable;
\item for some constant $J_0$ and exponent $\gexp_0$
\begin{itemize}
\item $\max\lf\{\|\phi(x)\|,\|\psi(x)\|, \|D\psi(x)\|, \ri\}  \leq J_0 V(x)^{\gexp_0}$;
\item (local H\"older type continuity)  $\|D\psi(x) - D\psi(x')\| \leq J'_1(V(x)^{\gexp_1}+V(x')^{\gexp_1})\|x-x'\|^{\holex'_\psi}$ for some exponent $0<\holex'_\psi\leq 1$.
\end{itemize}
\end{itemize}
Then for some constant $\consta_{3,m}$,
\begin{align*}
\EE_{\para_0}\lf[\lf\|\int_0^1 \phi(X^\vep(\vr_\vep(s)))\big(\psi(X^\vep(s)) - \psi(X^\vep(\vr_\vep(s)))\big) ds \ri\|^m \ri] \leq \consta_{3,m} (\Delta(\vep)/\vep)^{m(1+ \holex'_\psi \wedge \holex_{\dffun})/2},
\end{align*}
and under the additional assumption that $\|\phi(x) - \phi(x')\| \leq J_1(V(x)^{\gexp_1}+V(x')^{\gexp_1})\|x-x'\|^{\holex_\phi}$ with $0\leq \holex_\phi \leq 1,$
\begin{align*}
\EE_{\para_0}\lf[\lf\|\int_0^1 \phi(X^\vep(s))\big(\psi(X^\vep(s)) - \psi(X^\vep(\vr_\vep(s)))\big) ds \ri\|^m \ri] \ \leq \consta_{4,m}  (\Delta(\vep)/\vep)^{m(1+ \holex'_\psi \wedge \holex_{\dffun}\wedge \holex_\phi)/2},
\end{align*}
for some constant $\consta_{4,m}$.

\end{lemma}
\begin{proof}
By Taylor's expansion
\begin{align*}
\psi(X^\vep(s)) - \psi(X^\vep(\vr_\vep(s))) = \lf(\int_0^1 D\psi((1-u)X^\vep(\vr_\vep(s)) + u X^\vep(s)) \ du\ri) \lf(X^\vep(s)) -X^\vep(\vr_\vep(s))\ri).
\end{align*}
Hence
\begin{align}\label{eq:phi-psi-int}
\begin{aligned}
\int_0^1 & \phi(X^\vep(\vr_\vep(s))) \big(\psi(X^\vep(s)) - \psi(X^\vep(\vr_\vep(s)))\big) ds\\
=&\ \int_0^1\Bigg[\vep^{-1} \int_0^1  \phi(X^\vep(\vr_\vep(s))) D\psi((1-u)X^\vep(\vr_\vep(s)) + u X^\vep(s))\int_{\vr_\vep(s)}^s b(\drft_0, X^\vep(r))\ dr\ ds\\
&\ +\vep^{-1/2} \int_0^1 \phi(X^\vep(\vr_\vep(s)))\   D\psi((1-u)X^\vep(\vr_\vep(s)) + u X^\vep(s))\int_{\vr_\vep(s)}^s \dffun(\diff_0,X^\vep(r)) dW(r)\ ds\Bigg]\ du.
\end{aligned}
\end{align}
Note that by the growth condition on $\phi$, $D\psi$, and $b$, Condition \ref{cond:SDE}:\ref{cond:item:growth-gen-convex},  H\"older's inequality, Lemma \ref{lem:simp-int-bd},  and Corollary \ref{cor:int-bd-V}, the first term (without the $\vep^{-1}$ factor) inside the $u$-integral,
$$\SC{I}^\vep_0(u,t) \equiv \int_0^t\phi(X^\vep(\vr_\vep(s))) D\psi((1-u)X^\vep(\vr_\vep(s)) + u X^\vep(s))\int_{\vr_\vep(s)}^s b(\drft_0, X^\vep(r))\ dr\ ds$$
can be estimated as
\begin{align}\label{eq:I_0-est}
\begin{aligned}
\EE_{\para_0}(\|\SC{I}^\vep_0(u,t)\|^m) \leq&\ \const_{\ref*{lem:fun-diff-est-2},0,m}\lf(\EE_{\para_0} \int_0^t\|\phi(X^\vep(\vr_\vep(s)))\| (V(X^\vep(s))^{\gexp_0\gexp_{V,1}}+V(X^\vep(\vr_\vep(s)))^{\gexp_0\gexp_{V,1}})^{m/(m-1)}\ ds\ri)^{m-1} \\
& \ \times \EE_{\para_0}\int_0^t \lf\|\int_{\vr_\vep(s)}^s b(\drft_0,X^\vep(r))\ dr \ri\|^m\ ds\\
\leq&\ \const_{\ref*{lem:fun-diff-est-2},1,m}\Delta^{m-1}(\vep)\EE_{\para_0}\int_0^t \int_{\vr_\vep(s)}^s \|b(\drft_0,X^\vep(r))\|^m dr \ ds\\
\leq & \const_{\ref*{lem:fun-diff-est-2},1,m}\drbd^m_{b,0}(\drft_0) \Delta(\vep)^m  \EE_{\para_0}\int_0^t V^{m\bexp_{b,0}}(X^\vep(s))\ ds\ 
\leq \ \const_{\ref*{lem:fun-diff-est-2},2,m} \Delta(\vep)^m,
\end{aligned}
\end{align}
where the constants $\const_{\ref*{lem:fun-diff-est-2},1,m}, \const_{\ref*{lem:fun-diff-est-2},2,m}$ are independent of $u$.
Next write the second term inside the $u$-integral of \eqref{eq:phi-psi-int} (without the $\vep^{-1/2}$ factor), $$\SC{I}^\vep_1(u,t) \equiv \int_0^t \phi(X^\vep(\vr_\vep(s)))\   D\psi((1-u)X^\vep(\vr_\vep(s)) + u X^\vep(s))\int_{\vr_\vep(s)}^s \dffun(\diff_0,X^\vep(r)) dW(r)\ ds$$
as
$$\SC{I}^\vep_1(u,t) = M^\vep_{1,0}(t)+\SC{I}^\vep_{1,1}(u,t) + \SC{I}^\vep_{1,2}(t)$$
where
\begin{align*}
\SC{M}^\vep_{1,0}(t) \equiv &\ \int_0^t\phi(X^\vep(\vr_\vep(s)))D \psi(X^\vep(\vr_\vep(s))) \lf(\int_{\vr_\vep(s)}^s  \dffun(\diff_0,X^\vep(\vr_\vep(r))) dW(r)\ri) \ ds\\
=&\ \int_0^t\phi(X^\vep(\vr_\vep(s)))D \psi(X^\vep(\vr_\vep(s))) \dffun(\diff_0,X^\vep(\vr_\vep(s)))   \lf(W(s) - W(\vr_\vep(s))\ri) \ ds,\\
\SC{I}^\vep_{1,1}(u,t) \equiv &\ \int_0^t\phi(X^\vep(\vr_\vep(s)))\big(D\psi((1-u)X^\vep(\vr_\vep(s)) + u X^\vep(s)) - D \psi(X^\vep(\vr_\vep(s)))\big) \\
& \hs{.1 cm} \times \lf(\int_{\vr_\vep(s)}^s  \dffun(\diff_0,X^\vep(r)) dW(r)\ri) \ ds,\\
\SC{I}^\vep_{1,2}(t) \equiv &\ \int_0^t\phi(X^\vep(\vr_\vep(s)))D \psi(X^\vep(\vr_\vep(s))) \lf(\int_{\vr_\vep(s)}^s\dffun(\diff_0,X^\vep(r))- \dffun(\diff_0,X^\vep(\vr_\vep(r))) dW(r)\ri) ds.
\end{align*}
 The estimate on $\SC{M}^\vep_{1,0}(t)$ is more subtle, and a straightforward use of H\"older's and BDG inequalities will lead to a cruder estimate that will not suit our needs. This is where we need Lemma \ref{lem:disc-mart} which gives for some constant $\const_{\ref*{lem:fun-diff-est-2},3,m}$
\begin{align}\label{eq:M_1_0-est}
\EE_{\para_0}\lf(\|\SC{M}^\vep_{1,0}(t)\|^m\ri) \leq \const_{\ref*{lem:fun-diff-est-2},3,m} \Delta(\vep)^m.
\end{align}
Similar to the methods used for estimating the term $\SC{I}^\vep_0$, a combination of  local H\"older type continuity conditions on $D\psi$ and $\dffun(\diff_0,\cdot)$,  H\"older's and BDG inequalities and Lemma \ref{lem:fun-diff-est} easily show that for some constant $\const_{\ref*{lem:fun-diff-est-2},4,m} \geq 0$,
\begin{align}\label{eq:I_1_1-est}
\begin{aligned}
\EE_{\para_0}\lf(\lf\|\SC{I}^\vep_{1,1}(t)\ri\|^m\ri)\leq&\ \const_{\ref*{lem:fun-diff-est-2},4,m} \Delta(\vep)^{m(1+\holex'_{\psi})/2}/\vep^{m\holex'_{\psi}/2},\\
\EE_{\para_0}\lf(\lf\|\SC{I}^\vep_{1,2}(t)\ri\|^m\ri)\leq &\ \const_{\ref*{lem:fun-diff-est-2},4,m} \Delta(\vep)^{m(1+\holex_{\dffun})/2}/\vep^{m\holex_{\dffun}/2}.
\end{aligned}
\end{align}
The assertion now follows from \eqref{eq:phi-psi-int}, \eqref{eq:I_0-est},  \eqref{eq:M_1_0-est} and \eqref{eq:I_1_1-est}.

For the second estimate, first notice because of the assumption on $D\psi$, $\psi$ satisfies the  H\"older-type condition of Corollary \ref{cor:diff-est-1} with $\holex_\psi=1$. Now the second estimate easily follows from the first by writing 
\begin{align*}
\int_0^1 \phi(X^\vep(s))\big(\psi(X^\vep(s)) - \psi(X^\vep(\vr_\vep(s)))\big) ds =&\ \int_0^1 \phi(X^\vep(\vr_\vep(s)))\big(\psi(X^\vep(s)) - \psi(X^\vep(\vr_\vep(s)))\big) ds +R^\vep_{\phi,\psi}
\end{align*}
and estimating the term $R^\vep_{\phi,\psi} \equiv \int_0^1 \big(\phi(X^\vep(s))-\phi(X^\vep(\vr_\vep(s)))\big)\big(\psi(X^\vep(s)) - \psi(X^\vep(\vr_\vep(s))) \big)ds$
by H\"older's inequality and Corollary \ref{cor:diff-est-1} as
\begin{align*}
\EE_{\para_0}\|R^\vep_{\phi,\psi}\|^m \leq \const_{\ref*{lem:fun-diff-est-2},5,m} (\Delta(\vep)/\vep)^{m(1+\holex_\phi)/2}.
\end{align*}
\end{proof}


Next, denote the occupation measures of the process, $X^\vep$, and its discretized version,  $X^\vep\circ\vr_\vep$, on $R^d\times[0,1]$ respectively by $\occ^\vep$ and $\occ^{D,\vep}$, where they are defined by
$$\occ^\vep(A\times[0,t]) = \int_{\R^d\times [0,t]}1_{\{X^\vep(s) \in A\}}ds, \quad \occ^{D,\vep}(A\times[0,t]) = \int_{\R^d\times [0,t]}1_{\{X^\vep(\vr_\vep(s)) \in A\}}ds,$$
with $A \stackrel{m'ble}\subset \R^d$ and $ 0\leq t \leq 1.$
Notice that 
$$\occ^\vep (\cdot \times [0,1]) \eqd \vep \int_0^{\vep^{-1}}1_{\{ X(s)\ \in \ \cdot\}}\ ds$$
where recall $X$ is the solution of the SDE \eqref{eq:SDE0}.

\begin{lemma} \label{lem:tight-occ}
Suppose that  Condition \ref{cond:SDE}-\ref{cond:item:growth-rec} \& \ref{cond:item:growth-lyap-a} and Condition \ref{cond:SDE-coeff}: \ref{cond:b-growth}-\hyperlink{subitem-a}{(a)} \& \ref{cond:sig-growth}-\hyperlink{subitem-s-a}{(a)} hold.  Then for any $\para_0 =(\drft_0,\diff_0)$, $\occ^\vep$ is a tight family of random measures, $X$ defined by \eqref{eq:SDE0} has a unique stationary distribution $\inv_{\para_0}$, and $\occ^\vep \stackrel{\PP_{\para_0}}\rt\inv_{\para_0}\ot \leb$, that is, for any $\kappa>0$, 
$\PP_{\para_0}\lf[d_{\textsc{LP}}(\occ^\vep, \inv_{\para_0}\ot \leb) \geq \kappa\ri] \stackrel{\vep \rt 0} \rt 0,$ where $d_{\textsc{LP}}$ denotes the Levy-Prohorov metric on $\SC{M}^+_1(\R^d\times [0,1])$, the space of probability measures on $\R^d\times [0,1]$. 

If $\Delta(\vep)/\vep \stackrel{\vep \rt 0}\Rt 0$, and Condition \ref{cond:SDE}: \ref{cond:item:growth-rec} - \ref{cond:item:growth-lyap-b} and Condition \ref{cond:SDE-coeff}:\ref{cond:b-growth}-\hyperlink{subitem-a}{(a)}, \ref{cond:b-growth}-\hyperlink{subitem-b-lip}{(b)} \& \ref{cond:sig-growth}-\hyperlink{subitem-s-a}{(a)},\ref{cond:sig-growth}-\hyperlink{subitem-s-b}{(b)} hold,  then the same assertions hold for $\occ^{D,\vep}$

\end{lemma}

\begin{proof}
Since by the  assumption on $V$, its sublevel sets, $\{x:V(x) \leq R\}$ are compact,  to prove tightness of $\occ^{D,\vep}$ we need to show that for every $\eta>0$, there exists a constant $R_\eta$ such that
\begin{align}\label{eq:occ-tight}
\limsup_{\vep \rt 0}\EE_{\para_0}\occ^{D,\vep}\lf(\{x \in R^d: V(x) > R_\eta\}\times[0,1]\ri) \leq \eta.
\end{align}
Since by Proposition \ref{prop:int-bd}, $\EE_{\para_0} \int_{\R^d\times[0,1]}V(x) \occ^{D,\vep}(dx\times ds) = \EE_{\para_0} \int_0^1 V(X^\vep(\vr_\vep(s)) ds \leq V^*_1<\infty$ for all $\vep\leq \vep_0$ (for some $\vep_0>0$), \eqref{eq:occ-tight} follows from Markov inequality. Similar steps show that $\occ^\vep$ is also tight, and hence by Krylov-Bogoliubov theorem $X$ has a stationary distribution $\inv_{\para_0}$, that is, $\int \SC{L}_{\para_0} g(x)\inv_{\para_0}(dx) =0$ for any $g \in C^2_b(\R^d, \R).$ Since the diffusion coefficient $ a(\diff_0,x)$ is positive-definite, $\inv_{\para_0}$ is the unique stationary distribution of $X$.

We now show that $\occ^{D,\vep} \prt \inv_{\para_0}\ot \leb$ as $\vep \rt 0$. The proof of the corresponding convergence for $\occ^\vep$ is similar and in fact slightly easier.

Let $\occ^0$ be a limit point of $\occ^{D,\vep}$ in $\SC{M}^+_1(\R^d\times [0,1])$, that is, $\occ^{D,\vep} \RT \occ^0$ along a subsequence. For convenience, assume the convergence happens along the entire family.  We will show that $\occ^0 = \inv_{\para_0}\ot \l_{Leb}$. Toward this end, let  $g \in C^\infty_b(\R^d, \R)$.  Then by It\^o's lemma
\begin{align*}
g(X^\vep(t)) = g(X^\vep(0))+ \f{1}{\vep}\int_0^t \SC{L}_{\para_0}g(X^\vep(s)) ds+ \f{1}{\sqrt \vep} \int_0^t \nabla^T g(X^\vep(s)) a(\diff_0,X^\vep(s)) dW(s).
\end{align*}
Thus for $0\leq t \leq 1$,	
\begin{align}
\label{eq:g-exp-ito}
\begin{aligned}
&\int_{\R^d\times [0,t]} \SC{L}_{\para_0} g(x) \occ^{D,\vep}(dx\times ds) = \ \int_{\R^d\times [0,t]} \SC{L}_{\para_0} g(X^\vep(\vr_\vep(s))) ds \\
& = \ -\vep(g(X^\vep(t)) - g(X^\vep(0))) - \int_0^t \Big(\SC{L}_{\para_0}g(X^\vep(s)) - \SC{L}_{\para_0}g(X^\vep(\vr_\vep(s)))\Big)ds - \sqrt \vep M_g^\vep(t),
\end{aligned}
\end{align}	
where $M_g^\vep(t) \equiv \int_0^t \nabla^T g(X^\vep(s)) a(\diff_0,X^\vep(s)) dW(s)$ is a martingale. Since $g$ is bounded, $\vep(g(X^\vep(t)) - g(X^\vep(0))) \ert 0$ a.s. (and in $L^2(\PP_{\para_0})$), and since $\nabla g$ is also bounded it readily follows from  BDG inequality, growth condition on $\s$ (c.f. Condition  \ref{cond:SDE-coeff}-\ref{cond:sig-growth}-\hyperlink{subitem-s-a}{(a)}) and Proposition \ref{prop:int-bd} that $\sqrt \vep \sup_{t\leq 1} \|M_g^\vep(t)\| \ert 0$ in $L^2(\PP_{\para_0})$. Finally notice because of the local H\"older continuity assumptions on $b$ and $\dffun$ (see Condition \ref{cond:SDE-coeff}),  Corollary \ref{cor:diff-est-1}-\ref{item:diff-est:diff-1} shows that for some constant $\const_{\ref*{lem:tight-occ},0,m}$,
$$\EE_{\para_0}\lf[\int_0^1\|\SC{L}_{\para_0} g(X^\vep(s)) - \SC{L}_{\para_0} g(X^\vep(\vr_\vep(s))) \|^2 ds\ri] \leq \const_{\ref*{lem:tight-occ},0,m} (\Delta/\vep)^{\hexp_b\wedge \hexp_{\dffun}} \ert 0.$$
Therefore, taking $\vep \rt 0$ in \eqref{eq:g-exp-ito} we have 
$$\int_{\R^d \times [0,t]} \SC{L}g(x) \occ^0(dx\times ds) = 0,\quad  0\leq t \leq 1.$$
Clearly, $\occ^0(\R^d\times [0,t]) = t$, and thus the first marginal $\occ^0_{(1)}(\cdot) \equiv \occ^0(\R^d\times \cdot) = \leb(\cdot).$ Consequently writing $\occ^0(dx\times dt) = \occ^0_{(2|1)}(dx|t) dt$,  we see from the above equation that for almost all $0\leq t \leq 1$,
$$\int_{\R^d} \SC{L}g(x) \occ^0_{(2|1)}(dx|t) = 0, \quad g \in C^\infty_b(\R^d, \R).$$
It follows that for a.a $t$, $\occ^0_{(2|1)}(\cdot|t)$ is a stationary distribution of $X$, and by the assumption of uniqueness of stationary distribution, we have $\occ^0_{(2|1)}(\cdot|t) = \inv_{\para_0}.$ Since $\occ^0 = \inv_{\para_0} \ot \leb$ is deterministic, the convergence happens in probability. This proves the assertion.
\end{proof}

\begin{proposition} \label{prop:diff-est}
Let $\SC{B}^1 \subset \R^{d'_1}$ and $\SC{B}^2 \subset \R^{d'_2}$ be measurable subsets and  $\phi: \R^{d'_1}\times \R^{d'_2}\times \R^d \rt \R^{d_0 \times d_1}$ and $\psi: \R^{d'_1}\times \R^{d'_2}\times\R^{d} \rt \R^{d_1}$ be two functions such that their restrictions $\phi\Big|_{\SC{B}^1\times \SC{B}^2\times \R^d }$ and $\psi\Big|_{\SC{B}^1\times \SC{B}^2\times \R^d }$ satisfy the following  conditions:
\begin{itemize}
\item there exist functions $J_0, J_1:  \SC{B}^1\times \SC{B}^2  \rt [0,\infty)$, exponents $\gexp_0, \gexp_1\geq 0$ and $0 \leq \holex_{\psi} \leq 1$ such that for every $ \apar= (\apar^1,\apar^2) \in \SC{B}^1\times \SC{B}^2 ,$ 
\begin{itemize}
\item 
$\dst \max\lf\{\|\phi(\apar, x)\|,\|\psi(\apar, x)\|\ri\}  \leq J_0(\apar) V(x)^{\gexp_0};$

 \item $\dst \|\psi(\apar, x) - \psi(\apar, x')\| \leq  J_1(\apar) \lf(V(x)^{\gexp_1}+V(x')^{\gexp_1}\ri)\|x-x'\|^{\holex_{\psi}};$

\end{itemize}

\item $J_0$ is locally bounded;

\item (local H\"older type continuity w.r.t parameter) for any compact set  $\cmpt \subset \SC{B}^1\times \SC{B}^2$, there exist a constant $\consta_{\cmpt}$, Holder exponents $0\leq  \hexp_0 \leq 1$, and growth exponents $\gexp_2\geq 0$,  (each potentially depending on $\cmpt$) such that for every $\apar, \apar' \in \cmpt$
\begin{align*}
\max\lf \{\|\phi(\apar,x) - \phi(\apar',x)\|, \|\psi(\apar,x) - \phi(\apar',x)\|\ri\} \leq &\ \consta_{\cmpt} V(x)^{\gexp_2}\|\apar-\apar'\|^{\hexp_0}.
\end{align*}
\end{itemize}	
 Define the family of random maps, $\lf\{\Phi^\vep:\R^{d'_1}\times \R^{d'_2} \times \R^d \rt \R^{d_0}\ri\}$, and the deterministic map $\lf\{\Phi_{\para_0,0}:\R^{d'_1}\times \R^{d'_2} \times \R^d \rt \R^{d_0}\ri\}$ respectively by 
\begin{align*}
\Phi^\vep(\apar, t) = \int_0^t \phi(\apar, X^\vep(\vr_\vep(s))) \psi(\apar, X^\vep(s))ds, \quad \Phi_{\para_0,0}(\apar,t) = t \int_{\R^d} \phi(\apar, x)\psi(\apar, x))\inv_{\para_0}(dx)
\end{align*}	
Furthermore, assume that Condition \ref{cond:SDE}:\ref{cond:item:growth-rec} - \ref{cond:item:growth-lyap-b} and Condition \ref{cond:SDE-coeff}:\ref{cond:b-growth}-\hyperlink{subitem-a}{(a)}, \ref{cond:b-growth}-\hyperlink{subitem-b-lip}{(b)} \& \ref{cond:sig-growth}-\hyperlink{subitem-s-a}{(a)},\ref{cond:sig-growth}-\hyperlink{subitem-s-b}{(b)} hold and $\Delta(\vep)/\vep \ert 0$. Then the following hold under $\PP_{\para_0}$:

\begin{enumerate}[label=(\roman*), ref=(\roman*)]

\item  \label{item:diff-est:conv-1}

 The stochastic process (random field), $\Phi^{\vep}\Big|_{\SC{B}^1\times \SC{B}^2 }$ (viewed as a process over $\SC{B}^1\times \SC{B}^2  \times [0,1]$) is tight in  $C(\SC{B}^1\times \SC{B}^2 \times [0,1], \R^{d_0})$, and $\Phi^\vep\Big|_{\SC{B}^1\times \SC{B}^2 } \RT \Phi^{\para_0,0}\Big|_{\SC{B}^1\times \SC{B}^2 }$ in $C(\SC{B}^1\times \SC{B}^2  \times [0,1], \R^{d_0})$.

\item  \label{item:diff-est:conv-2}

Let  $\{\hat\apar^{1,\vep}\}$ be a family of $\R^{d'_1}$-valued random variables satisfying $\hat\apar^{1,\vep} \stackrel{\PP_{\para_0}} \rt  \apar^1_0$ as $\vep \rt 0$, where $\apar^1_0 \in \SC{B}_1$ is deterministic.Assume that either $\SC{B}^1$ is open in $\R^{d_1'}$ or $\SC{B}^1$ is closed and each $\hat\apar^{1,\vep}$ takes values in $\SC{B}^1$. . Then the process $\Phi^{\vep}(\hat\apar^{1,\vep},\cdot,\cdot)$ is tight in  $C(\SC{B}^2\times [0,1], \R^{d_0})$ and 
 $\Phi^{\vep}(\hat\apar^{1,\vep},\cdot,\cdot)\Big|_{\SC{B}^2 } \prt \Phi_{\para_0,0}(\apar^1_0,\cdot,\cdot)\Big|_{\SC{B}^2 }$ as $\vep \rt 0$.

\end{enumerate}

\end{proposition}

\begin{proof}

\np
\ref{item:diff-est:conv-1} 
The proof of this assertion is similar to (in fact simpler than) that of (i) of Proposition \ref{prop:gen-conv2}\ref{item:si-tgt}.

\np
\ref{item:diff-est:conv-2} The proof of tightness of $\{\Phi^\vep(\hat\apar^{1,\vep},\cdot, \cdot)\}$ is similar to Proposition \ref{prop:gen-conv2}.
Now let $(\tilde \Phi_0, \apar^1_0)$ be a limit point of $\{\lf(\Phi^\vep(\hat\apar^{1,\vep},\cdot, \cdot), \hat\apar^{1,\vep}\ri)\}$. Then along a subsequence as $\vep \rt 0$,
$$(\Phi^\vep(\hat\apar^{1,\vep},\cdot, \cdot), \hat\apar^{1,\vep}, \occ^{D,\vep}) \RT  (\tilde\Phi_0,\apar^1_0, \inv_{\para_0}\ot \l_{Leb})$$
 in $C(\SC{B}^2\times[0,1],\R^{d_0})\times \R^{d'_1}\times \mbb{M}_1(\R^d\times[0,1]))$. We will show that $\tilde \Phi_0(\cdot,\cdot) = \Phi_{\para_0,0}(\apar^1_0,\cdot,\cdot)$ a.s, that is, $\tilde \Phi_0(\apar^2,t) = \Phi^0(\apar^1_0,\apar^2, t)$ for each $\apar^2 \in \SC{B}^2, t \in [0,1]$, which, in particular, will show that the limit point does not depend on the choice of the subsequence thereby establishing the assertion. Towards this end, simply write
\begin{align*}
\Phi^\vep(\hat\apar^{1,\vep},\cdot, \cdot) =&\ \Phi^\vep(\apar^{1}_0,\cdot, t)+ \Phi^\vep(\hat\apar^{1,\vep},\cdot, \cdot) - \Phi^\vep(\apar^{1}_0,\cdot, \cdot).
\end{align*}
As in Proposition \ref{prop:gen-conv2}-\ref{item:si-conv}, $\Phi^\vep(\hat\apar^{1,\vep},\cdot, \cdot) - \Phi^\vep(\apar^{1}_0,\cdot, \cdot) \stackrel{\PP_{\para_0}} \rt 0$ in $C(\SC{B}^2\times[0,1],\R^d)$.
Note that tightness of $ \Phi^{\vep}(\apar^1_0,\cdot, \cdot)$ in $C(\SC{B}^2\times[0,1],\R^d)$ is an immediate consequence of part \ref{item:diff-est:conv-1}. To identify its limit as $\Phi_{\para_0, 0}(\apar^1_0,\cdot, \cdot)$,
observe that for any $\apar^2 \in \SC{B}^2$, 
\begin{align*}
\Phi^{\vep}(\apar^1_0,\apar^2, t)=&\ \int_0^t  \phi(\apar^1_0,\apar^2, X^\vep(\vr_\vep(s)))\psi(\apar^1_0,\apar^2,X^\vep(\vr_\vep(s)))ds +  \tilde{\scr{R}}^\vep(t)\\
=&\ \int_{\R^d\times [0,t]}  \phi(\apar^1_0,\apar^2, x)\psi(\apar^1_0,\apar^2,x)\occ^{D,\vep}(dx\times ds) +  \tilde{\scr{R}}^\vep(t).
\end{align*}
The assertion follows upon observing that the first term converges in probability (w.r.t $\PP_{\para_0}$) to $\Phi_{\para_0,0}(\apar^1_0,\apar^2,t) = t\int_{\R^d}  \phi(\apar^1_0,\apar^2, x)\psi(\apar^1_0,\apar^2,x)\inv_{\para_0}(dx)$
by Lemma \ref{lem:tight-occ} and Lemma \ref{lem:uni-int}, and 
\begin{align*}
\tilde{\scr{R}}^\vep(t)\equiv&\ \int_0^t \phi(\apar^1_0,\apar^2, X^\vep(\vr_\vep(s)))\lf(\psi(\apar^1_0,\apar^2, X^\vep(s)) - \psi(\apar^1_0,\apar^2, X^\vep(\vr_\vep(s)))\ri) \ ds \ert 0 \text{ in } L^2(\PP_{\para_0})
\end{align*}
because of Corollary \ref{cor:diff-est-1}-\ref{item:diff-est:diff-1}.

\end{proof}

We now prove convergence  results for the stochastic-integral version of $\Phi^\vep$ in the previous proposition; for this we need $\psi$ to be a matrix-valued function.

\begin{proposition}\label{prop:gen-conv2}
%
Let $\SC{B}^1 \subset \R^{d'_1}$ and $\SC{B}^2 \subset \R^{d'_2}$ be measurable subsets and  $\phi: \R^{d'_1}\times \R^{d'_2}\times \R^d \rt \R^{d_0 \times d_1}$ and $\psi: \R^{d'_1}\times \R^{d'_2}\times\R^{d} \rt \R^{d_1\times d}$ be two functions such that their restrictions $\phi\Big|_{\SC{B}^1\times \SC{B}^2\times \R^d }$ and $\psi\Big|_{\SC{B}^1\times \SC{B}^2\times \R^d }$ satisfy the following  conditions:

\begin{enumerate}[label=(\alph*), ref=(\alph*)]
\item  there exist functions $J_0, J_1:  \SC{B}^1\times \SC{B}^2\rt [0,\infty)$, exponents $\gexp_0, \gexp_1\geq 0$ and $0 \leq \holex_{\psi} \leq 1$ such that for every $ \apar= (\apar^1,\apar^2) \in \SC{B}^1\times \SC{B}^2,$ 
\begin{itemize}
\item 
$\dst \max\lf\{\|\phi(\apar, x)\|,\|\psi(\apar, x)\|\ri\}  \leq J_0(\apar) V(x)^{\gexp_0};$

 \item $\dst \|\psi(\apar, x) - \psi(\apar, x')\| \leq  J_1(\apar) \lf(V(x)^{\gexp_1}+V(x')^{\gexp_1}\ri)\|x-x'\|^{\holex_{\psi}};$

\end{itemize}

\item \label{item:gen-conv2:loc-bdd} $J_0$ is locally bounded;

\item (local H\"older type continuity w.r.t parameter) for any compact set  $\cmpt \subset \SC{B}^1\times \SC{B}^2$, there exist a constant $\consta_{\cmpt}$, Holder exponents $0\leq  \hexp_0 \leq 1$, and growth exponents $\gexp_2\geq 0$,  (each potentially depending on $\cmpt$) such that for every $\apar, \apar' \in \cmpt$
\begin{align*}
\max\lf \{\|\phi(\apar,x) - \phi(\apar',x)\|, \|\psi(\apar,x) - \phi(\apar',x)\|\ri\} \leq &\ \consta_{\cmpt} V(x)^{\gexp_2}\|\apar-\apar'\|^{\hexp_0}.
\end{align*}
\end{enumerate}
Define the family of random maps $\lf\{\Phi^{W,\vep}: \R^{d'_1}\times \R^{d'_2} \times \R^d \rt \R^{d_0}\ri\}$ by
\begin{align*}
\Phi^{W,\vep}(\apar, t) = \int_0^t \phi(\apar, X^\vep(\vr_\vep(s))) \psi(\apar, X^\vep(s)) dW(s).
\end{align*}
 Furthermore, assume that Condition \ref{cond:SDE}:\ref{cond:item:growth-rec} - \ref{cond:item:growth-lyap-b} and Condition \ref{cond:SDE-coeff}:\ref{cond:b-growth}-\hyperlink{subitem-a}{(a)}, \ref{cond:b-growth}-\hyperlink{subitem-b-lip}{(b)} \& \ref{cond:sig-growth}-\hyperlink{subitem-s-a}{(a)},\ref{cond:sig-growth}-\hyperlink{subitem-s-b}{(b)} hold and $\Delta(\vep)/\vep \ert 0$. Then the following hold.
\begin{enumerate}[label=(\roman*), ref=(\roman*)]
 
 \item \label{item:si-tgt} The stochastic process (random field), $\Phi^{W,\vep}\Big|_{\SC{B}^1\times \SC{B}^2 }$ (viewed as a process over $\SC{B}^1\times \SC{B}^2  \times [0,1]$) is tight in  $C(\SC{B}^1\times \SC{B}^2 \times [0,1], \R^{d_0})$.

 \item \label{item:si-conv} Let  $\{\hat\apar^{1,\vep}\}$ be a family of $\R^{d'_1}$-valued random variables satisfying $\hat\apar^{1,\vep} \stackrel{\PP_{\para_0}} \rt  \apar^1_0$ as $\vep \rt 0$, where $\apar^1_0 \in \SC{B}^1$ is deterministic. Assume that either $\SC{B}^1$ is open in $\R^{d_1'}$ or $\SC{B}^1$ is closed and each $\hat\apar^{1,\vep}$ takes values in $\SC{B}^1$.  The process $\Phi^{W,\vep}(\hat\apar^{1,\vep},\cdot,\cdot)$ is tight in $C(\SC{B}^2\times [0,1], \R^{d_0})$ and 
 $\Phi^{W,\vep}(\hat\apar^{1,\vep},\cdot,\cdot) \RT \Phi^W_0(\apar^1_0,\cdot,\cdot)$ as $\vep \rt 0$, where $\Phi^W_0(\apar^1_0,\apar^2,\cdot) \stackrel{d}=\lf(\scr{S}^{\phi_1,\psi_1}_{\para_0}(\apar^1_0, \apar^2)\ri)^{1/2}W$
and 
$$\dst\scr{S}^{\phi,\psi}_{\para_0}(\apar^1_0, \apar^2) \dfeq \int_{\R^d} \phi(\apar^1_0, \apar^2,x) \psi(\apar^1_0, \apar^2,x)\psi^T(\apar^1_0, \apar^2,x) \phi^T(\apar^1_0, \apar^2,x) \inv_{\para_0}(dx).$$
In particular, for a fixed $\apar^2$ as $\vep \rt 0$
 \begin{align*}
 \Phi^{W,\vep}(\hat\apar^{1,\vep}, \apar^2,1) = \int_0^1 & \phi(\hat\apar^{1,\vep}, \apar^2, X^\vep(\vr_\vep(s))) \psi(\hat\apar^{1,\vep}, \apar^2, X^\vep(s)) dW(s) \\
 & \RT \lf(\scr{S}^{\phi,\psi}_{\para_0}(\apar^1_0, \apar^2)\ri)^{1/2} \No(0,I). 
 \end{align*}
 \end{enumerate}
\end{proposition}

\begin{proof} 
\ref{item:si-tgt} Notice that for any fixed $\apar = (\apar^1, \apar^2) \in \SC{B}^1\times \SC{B}^2$, the family of $\R^{d_0}$-valued random variables $\{\Phi^{W,\vep}(\apar,1)\}$ is tight. Thus to prove the tightness at a process level over $\SC{B}^1\times \SC{B}^2\times[0,1]$, we need to show that for any compact set $\cmpt \subset \SC{B}^{1}\times \SC{B}^2$ and constants $\kappa>0, \beta>0$, there exist $0<\delta, \vep_0<1$ such that the quantity
 \begin{align}\label{eq:tight-result-2}
\SC{P}^{W,\vep, \kappa}_{\cmpt} (\delta)\equiv \PP\lf[\sup_{\apar, \apar' \in \cmpt \atop \|\apar-\apar'\|\leq \delta }\sup_{0\leq t\leq t+h\leq 1 \atop 0\leq h \leq \delta } \|\Phi^{W,\vep}(\apar', t+h) - \Phi^{W,\vep}(\apar,t)\| \geq \kappa\ri] \leq \beta, \quad \text{ for all } \vep \leq \vep_0.
\end{align}
By   \cite[Lemma 12.3]{Bil68} (also see \cite[Theorem 8.8, Chapter 3]{EK86}), this follows if we show that for any compact set $\cmpt \subset \SC{B}^1\times \SC{B}^2$, there exist exponents $m_0\geq 0$, $m_1 > 1$ and $m_2 > 1$ (each potentially depending on $\cmpt$)   and a constant $\bar\consta_{\cmpt,m}$ such that for all $\apar, \apar' \in \cmpt$ and   $t, t+h \in [0,1]$,
\begin{align}\label{eq:tight-mean}
\EE_{\para_0}\lf[ \|\Phi^{W,\vep}(\apar', t+h) - \Phi^{W,\vep}(\apar,t)\|^{m_0}\ri] \leq \bar\consta_{\cmpt,m}\lf(\|\apar - \apar'\|^{m_1}+ |h|^{m_2}\ri).
\end{align}
Toward this end notice that by BDG and H\"older's  inequalities  the term,\\ $E^{W,m_0,\vep}_{1} \equiv \EE_{\para_0}\lf\| \Phi^{W,\vep}(\apar', t+h) - \Phi^{W,\vep}(\apar', t) \ri\|^{m_0}$ can be estimated as
\begin{align}\non
E^{W,m_0,\vep}_{1} \leq &\ \const^{B}_m \EE_{\para_0}\lf[\int_{t}^{t+h} \|\phi(\apar', X^\vep(\vr_\vep(s)))\psi(\apar', X^\vep(s))\|^2 ds\ri]^{m_0/2}\\ \non
\leq &\ \const^{B}_m h^{m_0/2-1} \EE_{\para_0}\lf[\int_{0}^{1} \|\phi(\apar', X^\vep(\vr_\vep(s)))\psi(\apar', X^\vep(s))\|^{m_0} ds\ri]\\ 
\label{eq:phi-tight-0}
\leq&\ \const^{B}_m   J_0^{*,2m_0}(\cmpt) V_{2m_0\gexp_0, \para_0}^* h^{m_0/2-1},
\end{align}
where $\const^{B}_m$ is the Burkholder's constant, $J^{*,2m_0}_0(\cmpt) =\sup_{\apar'\in \cmpt} J_0^{2m_0}(\apar')$ (which is finite because of assumption \ref{item:gen-conv2:loc-bdd}) and the constant $ V_{2m_0\gexp_0, \para_0}^*$ is as in Corollary \ref{cor:int-bd-V} .
Next notice because of the assumption of local H\"older continuity on $\phi$ and $\psi$,
\begin{align*}
\|\phi(\apar',X^\vep(\vr_\vep(s))) & \psi(\apar',X^\vep(s)) - \phi(\apar,X^\vep(\vr_\vep(s)))\psi(\apar,X^\vep(s))\| \\
\leq &\ J^{*}_0(\cmpt)\consta_{\cmpt} V(X^\vep(s))^{\gexp_0\vee\gexp_2}V(X^\vep(\vr_\vep(s)))^{\gexp_0\vee\gexp_2}\|\apar-\apar'\|^{\hexp_0}, \quad \apar, \apar' \in \cmpt
\end{align*}
Thus again by BDG and H\"older's inequalities, the term $E^{W,m_0,\vep}_{2} \equiv \EE_{\para_0}\lf\|\Phi^{W,\vep}(\apar', t) - \Phi^{W,\vep}(\apar, t)\ri\|^{m_0}$ can be estimated as
\begin{align*}
E^{W,m_0,\vep}_{2}  \leq&\ \const^B_m  \EE_{\para_0}\lf[\int_{0}^{t} \|\phi(\apar',X^\vep(\vr_\vep(s)))\psi(\apar',X^\vep(s)) - \phi(\apar,X^\vep(\vr_\vep(s)))\psi(\apar,X^\vep(s))\|^2\ ds\ri]^{m_0/2} \\
\leq &\ \const^B_mJ^{*,m_0}(\cmpt)\consta^{m_0}_{\cmpt}\|\apar-\apar'\|^{m_0\hexp_0}\EE_{\para_0}\lf(\int_0^1 V(X^\vep(s))^{\gexp_0\vee\gexp_2}V(X^\vep(\vr_\vep(s)))^{\gexp_0\vee\gexp_2}\ ds\ri)\\
\leq &\ \const^B_mJ^{*,m_0}(\cmpt)\consta^{m_0}_{\cmpt} V^*_{2m_0(\gexp_0\vee\gexp_2), \para_0}\|\apar-\apar'\|^{m_0\hexp_0}.
\end{align*}
 
\eqref{eq:tight-mean} now follows by simply writing $\Phi^{W,\vep}(\apar', t+h) - \Phi^{W,\vep}(\apar,t) = \Phi^{W,\vep}(\apar', t+h) - \Phi^{W,\vep}(\apar', t) + \Phi^{W,\vep}(\apar', t) - \Phi^{W,\vep}(\apar, t)$ and choosing $m_0>4$ large enough to ensure $m_0\hexp_0 >1$. 

\np
\ref{item:si-conv} Write $\Phi^{W,\vep}(\hat\apar^{1,\vep},\cdot,\cdot) = \Phi^{W,\vep}(\apar^{1}_0,\cdot,\cdot)+ \Phi^{W,\vep}(\hat\apar^{1,\vep},\cdot,\cdot) -\Phi^{W,\vep}(\apar^{1}_0,\cdot,\cdot).$
We first prove that 
$$\Phi^{W,\vep}(\hat\apar^{1,\vep},\cdot,\cdot) -\Phi^{W,\vep}(\apar^{1}_0,\cdot,\cdot) \stackrel{\PP_{\para_0}} \Rt 0 \quad \text{ in } C(\SC{B}^2\times[0,1],\R^d)$$
 as $\vep \rt 0.$  For this we need to show that  for any compact set $\cmpt^2 \subset \SC{B}^2$, $\kappa>0, \tilde\beta>0$, there exists $0<\vep_0<1$ such that for all $\vep \leq \vep_0$, the quantity
\begin{align}\label{eq:prob-conv-zero}
\tilde{\SC{P}}^{W,\vep, \kappa} \equiv \PP\lf[\sup_{\apar^2 \in \cmpt^2 }\sup_{0\leq t\leq 1 } \|\Phi^{W,\vep}(\hat \apar^{1,\vep}, \apar^2, t) - \Phi^{W,\vep}(\apar^1_0, \apar^2, t)\| \geq \kappa\ri] \leq \tilde\beta.
\end{align}
We first claim that for any $0<\delta \leq 1$, there exist $\vep_0$ and a compact set $\SC{B}^{1,\delta}_0 \subset  \SC{B}^1$ containing $\apar^1_0$  such that
$$ \sup\lf\{\|\apar^1-\apar^1_0\|: \apar^1 \in \SC{B}^{1,\delta}_0  \ri\} \leq \delta,$$
 and for all $\vep \leq \vep_0$
\begin{align}\label{eq:eta-cmpt}
\PP_{\theta_0}\lf[\hat \apar^{1,\vep} \notin \SC{B}^{1,\delta}_0\ri] \leq \tilde\beta/2.
\end{align}
 To see this fix a $\delta$. 
  When $\SC{B}^1$ is open, choose $r_0$ so that $\overline{B(\apar^1_0, r_0)} \subset \SC{B}^1$, and take $ \SC{B}^{1,\delta}_0 \equiv \overline{B(\apar^1_0, r_0\wedge\delta)}.$  Since $\hat \apar^{1,\vep} \stackrel{\PP_{\para_0}}\rt \ \apar^1_0$, choose $\vep_0$ such that so that for all $\vep\leq \vep_0$,\\
 $\PP_{\theta_0}\lf[\hat \apar^{1,\vep} \notin \overline{B(\apar^1_0, r_0\wedge\delta)} \ri] \leq \tilde\beta/2.$

Now consider the case when $\SC{B}^1 \subset \R^{d'_1}$ is closed and the $\hat \apar^{1,\vep}$  take values in $\SC{B}^1$. 
 Let  $\vep_0$ be such that for all $\vep\leq \vep_0$, $\PP_{\theta_0}\lf[\hat \apar^{1,\vep} \notin \overline{B(\apar^1_0, \delta)} \ri] \leq \tilde\beta/2.$  Now choose $\SC{B}^{1,\delta}_0 \equiv  \SC{B}^1 \cap \overline{B(\apar^1_0, \delta)}$ which, because $\SC{B}^1$ is closed, is a compact subset of $\SC{B}^1$.

It is clear that in either case, we can choose the $\SC{B}^{1,\delta}_0$ to be non-decreasing in $\delta$, that is,  for $\delta<\delta'$, $\SC{B}^{1,\delta}_0 \subset \SC{B}^{1,\delta'}_0$




Next by (i), which shows tightness of $\Phi^{W,\vep}$ in $C(\SC{B}^1\times\SC{B}^2 \times [0,1], \R^{d_0})$, choose $0<\delta<1$ and  $\vep_0$ so  that  for all $\vep \leq \vep_0$ \eqref{eq:tight-result-2} holds with $\cmpt = \SC{B}^{1,1}_0 \times\cmpt^2$ and $\beta = \tilde \beta/2$. Now make $\vep_0$ smaller if necessary so that \eqref{eq:eta-cmpt} holds for this choice of $\delta$. As remarked, $\SC{B}^{1,\delta}_0 \subset \SC{B}^{1,1}_0$.
Consequently,
\begin{align*}
\tilde{\SC{P}}^{W,\vep, \kappa} \leq&\ \PP\lf[\hat \apar^{1,\vep} \in \SC{B}^{1,\delta}_0, \sup_{\apar^2 \in \cmpt^2}\sup_{0\leq t\leq 1 } \|\Phi^{W,\vep}(\hat \apar^{1,\vep}, \apar^2, t) - \Phi^{W,\vep}(\apar^1_0, \apar^2, t)\| \geq \kappa\ri]\\
&\ + \PP_{\theta_0}\lf[\hat \apar^{1,\vep} \notin \SC{B}^{1,\delta}_0 \ri]
\leq\ \SC{P}^{W,\vep, \kappa}_{\cmpt} (\delta) + \tilde\beta/2 \leq \tilde\beta.
\end{align*}
Tightness of $ \Phi^{W,\vep}(\apar^1_0,\cdot, \cdot)$ in $C(\SC{B}^2\times [0,1], \R^{d_0})$ is an immediate consequence of \ref{item:si-tgt}.  For its convergence observe that for any $\apar^2$, $\Phi^{W,\vep}(\apar^1_0,\apar^2, \cdot) $ is a martingale whose (matrix) quadratic variation $\mqd{\Phi^{W,\vep}(\apar^1_0,\apar^2, \cdot)}$ is given by
\begin{align*}
\mqd{\Phi^{W,\vep}(\apar^1_0,\apar^2, \cdot)}_t = \int_0^t \phi(\apar^1_0,\apar^2, X^\vep(\vr_\vep(s))) \psi(\apar^1_0,\apar^2, X^\vep(s))\psi^T(\apar^1_0,\apar^2, X^\vep(s)) \phi^T(\apar^1_0,\apar^2, X^\vep(\vr_\vep(s)))ds.
\end{align*}
By Proposition \ref{prop:diff-est}-\ref{item:diff-est:conv-1} $\mqd{\Phi^{W,\vep}(\apar^1_0,\apar^2,\cdot)}$ is tight  and  
 $\mqd{\Phi^{W,\vep}(\apar^1_0,\apar^2,\cdot)} \rt \mfk{S}^{\phi_1,\psi_1}_{\para_0}(\apar^1_0, \apar^2)$ in $C([0,1], R^{d_0\times d_0})$, where $\mfk{S}^{\phi_1,\psi_1}_{\para_0}(\apar^1_0, \apar^2)(t)\equiv t\scr{S}^{\phi_1,\psi_1}_{\para_0}(\apar^1_0, \apar^2).$ The assertion now follows from the martingale central limit theorem (see \cite[Chapter 7]{EK86}).

\end{proof}

\section{Proofs of consistency and central limit theorem}\label{sec:main-proofs}

\begin{proof}[{\bf Proof of Theorem \ref{th:const-gen}}] 
We first work under the condition \ref{cond:est-tight}. Since $\{\hat \auxp^\vep\}$ is tight, for every $\eta>0$, there exists a compact ball $\overline{ B(\auxp_*, R_{\eta})} \subset \R^{d_0}$ (note that $\auxp_*$ is deterministic) and $\vep_0$ such that 
\begin{align}\label{eq:tight-auxp}
\sup_{\vep \leq \vep_0} \PP(\hat \auxp^\vep \notin \overline{ B(\auxp_*, R_{\eta})}) = \sup_{\vep \leq \vep_0}  \PP(\hat \auxp^\vep \notin \overline{ B(\auxp_*, R_{\eta})} \cap \mathbb A ) \leq \eta.
\end{align}
 Since $\mathbb A$ is closed, for any $0<r<R$, $\overline{B(\auxp_*, R)}\cap B(\auxp_*,r)^c \cap \mathbb A$ is compact. By the hypothesis $\loss^\vep(\cdot) - \loss^\vep(\auxp_*)   \stackrel{\PP} \Rt  \bar{\loss}^0(\cdot) \dfeq \loss^0(\cdot) - \loss^0(\auxp_*) $ in $C(\mathbb A, \R)$ as $\vep \rt 0$. Thus  for any $0<r<R$ and $\delta>0$,
$$\PP\lf(\sup_{\auxp \in  \overline{B(\auxp_*, R)}\cap B(\auxp_*,r)^c \cap \mathbb A }\lf|\loss^\vep(\auxp) - \loss^\vep(\auxp_*)  - \bar\loss^0(\auxp)\ri| > \delta\ri) \rt 0,$$
 as $\vep \rt 0$.
For any $0<r<R$ and $\delta>0$, let $\bar{\loss}_{*}^{\auxp_*,r,R} = \min_\auxp \lf\{\bar{\loss}^0(\auxp): r\leq|\auxp - \auxp_*|\leq R, \ \auxp \in \mathbb A \ri\}$. Since $\auxp_*$ is the unique global minimum of $\bar{\loss}^0$ and $\bar{\loss}^0(\auxp_*)=0$,  $\bar{\loss}^0(\auxp) > 0$ for all $\auxp \neq \auxp_*$. Thus $\bar{\loss}_{*}^{\auxp_*,r,R} >0$. Let $\delta = \bar{\loss}_{*}^{\auxp_*,r,R} /2.$ 
Then it follows that for any $0<r<R$,  $\PP(\SC{D}_\vep^{r,R}) \ert 1$, where
\begin{align*}
\SC{D}_\vep^{r,R} = \lf\{\om:\loss^\vep(\auxp,\om) - \loss^\vep(\auxp_*,\om)  \geq \bar{\loss}_{*}^{\auxp_*,r,R} /2,\ r\leq|\auxp - \auxp_*|\leq R, \ \auxp \in \mathbb A \ri\}.
\end{align*}
Now notice that for any $R>0,$
\begin{align*}
\lf\{\om: \|\hat\auxp^\vep(\om)-\auxp_*\| > r\ri \} =&\ \lf\{ r< \|\hat\auxp^\vep(\om) - \auxp_*\| \leq R\ri\} \cup \lf\{\hat\auxp^\vep(\om) \notin \overline{ B(\auxp_*, R)}\ri\}\\
\subset&\  \lf(\SC{D}_\vep^{r,R} \ri)^c\cup \lf\{\hat\auxp^\vep \notin \overline{ B(\auxp_*, R)}\ri\}.
\end{align*}
To see this suppose that $\om$ is such that  $r< \|\hat\auxp^\vep(\om)-\auxp_*\| \leq R $ but $\om \notin \lf(\SC{D}_\vep^{r,R}\ri)^c$, that is, $\om \in \SC{D}_\vep^{r,R}$. But then $\loss^\vep(\auxp, \om) >  \loss^\vep(\auxp_*,\om)$ for all $\auxp \in \mathbb A$ satisfying $r\leq|\auxp-\auxp_*|\leq R$. In particular, we  have  $\loss^\vep(\hat\auxp^\vep(\om), \om) > \loss^\vep(\auxp_*, \om) $. This is impossible as  $\hat\auxp^\vep(\om)$ is a global minimizer of $\loss^\vep(\cdot, \om).$

Taking $R = R_{\eta}$, it follows from \eqref{eq:tight-auxp} that 
\begin{align*}
\lim_{\vep \rt 0} \PP\lf( \|\hat\auxp^\vep-\auxp_*\| > r\ri) \leq \eta.
\end{align*}
Since this is true for any $\eta$, the assertion follows.

We now work under condition \ref{cond:parsp}. Here the tightness of $\{\hat\auxp^\vep\}$ is not directly known. Since $\mathbb A$ is open, there exists an $R_0>0$ such that the compact ball $ \overline{ B(\auxp_*, R_0)} \subset \mathbb A$. Now since $\loss^\vep(\cdot, \om)$ is differentiable, and $\mathbb A$ is open and convex, $\nabla_\auxp \loss^\vep(\hat\auxp^\vep(\om), \om)=0$. Next we claim that for any $0<r<R_0$, 
\begin{align*}
\lf\{\om: \|\hat\auxp^\vep(\om)-\auxp_*\| \geq r\ri \} \subset \lf(\SC{D}_\vep^{r,R_0}\ri)^c.
\end{align*}
To see this suppose $\om \in \SC{D}_\vep^{r,R_0}$. Then $\loss^\vep(\auxp, \om) > \loss^\vep(\auxp_*, \om)$ for any $\auxp$ satisfying $r\leq|\auxp - \auxp_*|\leq R_0$ (note that such an $\auxp \in \mathbb A$). It follows that the minimum of $\loss^\vep(\cdot, \om)$ over the compact set $\overline {B(\auxp_*,r)} \subset \mathbb A$ cannot be attained at the boundary, $\{\auxp:\|\auxp-\auxp_*\|=r\}$; it  is attained at an interior point $\tilde \auxp^\vep(\om) \in B(\auxp_*,r)$. In other words, $\tilde \auxp^\vep(\om)$ is a point of local minimum of $\loss^\vep(\cdot, \om)$, and hence $\nabla_{\auxp} \loss^\vep(\tilde \auxp^\vep(\om), \om) =0$. But then condition \ref{cond:parsp} implies that $\hat\auxp^\vep(\om) = \tilde\auxp^\vep(\om) \in B(\auxp_*,r)$ which establishes the claim. The assertion now is immediate from the previously shown fact, $\PP\lf((\SC{D}_\vep^{r,R_0})^c\ri) \rt 0,$ which holds in this case as well because  $ \overline{B(\auxp_*, R_0)}\cap B(\auxp_*,r)^c \cap \mathbb A = \overline{B(\auxp_*, R_0)}\cap B(\auxp_*,r)^c$ is compact.
	 	
\end{proof}

\begin{lemma}\label{lem:conv-like}
For each $\vep>0$, let $X^\vep$ and $\ell^\vep_{A}(\cdot)$ be respectively given by \eqref{eq:SDE0} and \eqref{eq:approx-like-int}. Fix $\para_0 = (\drft_0, \diff_0)$ and define the function $\L_{\para_0} \in C(\drpsp, \R)$ by 
\begin{align*} 
\L_{\para_0}(\drft) = \int_0^1 b^T(\drft, x)a^{-1}(\diff_0,x) b(\drft_0, x) \inv_{\para_0}(dx)-\f{1}{2 }  \int_0^t b^T(\drft, x)a^{-1}(\diff_0,x)(x) b(\drft, x) \inv_{\para_0}(dx).
\end{align*}
Assume that $\Delta(\vep)/\vep \ert 0$ and Condition \ref{cond:MLE-SDE} of Theorem \ref{th:const-AMLE-0} hold. 
Then 
$\vep \ell^\vep_{A}(\cdot) \stackrel{P_{\theta_0}}\Rt \   \L_{\theta_0}(\cdot)$ in $C(\drpsp,\R).$
\end{lemma}

\begin{proof}
The result is a consequence of Proposition \ref{prop:diff-est} and Proposition \ref{prop:gen-conv2} with $\diff$ and $\drft$ playing the roles of $\apar^1$ and $\apar^2$.
Notice that under $\PP_{\para_0}$, $X^\vep$ satisfies \eqref{eq:SDE1} with $\theta = \theta_0$ and thus we have from \eqref{eq:sc-lik-eq}

\begin{align} \label{eq:log-lik}
\begin{aligned}
\vep \ell^\vep_{A}(\drft) =&\ \sqrt \vep\  \SC{I}^{\vep}(\drft,1) + \int_0^1 b^T(\drft, X^\vep(\vr_\vep(s)))a^{-1}(\hat\diff^\vep, X^\vep(\vr_\vep(s))) b(\drft_0, X^\vep(s))ds \\
&\hs{.2cm} -\f{1}{2 }  \int_0^1 b^T(\drft, X^\vep(\vr_\vep(s)))a^{-1}(\hat\diff^\vep, X^\vep(\vr_\vep(s))) b(\drft, X^\vep(\vr_\vep(s)))ds
\end{aligned}
\end{align}
where
\begin{align*}
 \SC{I}^{\vep}(\drft, t) = \int_0^t  b^T(\drft, X^\vep(\vr_\vep(s)))  a^{-1}(\hat\diff^\vep, X^\vep(\vr_\vep(s))) \dffun(\diff_0, X^\vep(s)) dW(s).
\end{align*}
Because of the assumption, $\hat \diff^\vep \stackrel{P_{\theta_0}} \rt \diff_0$, and the conditions on $b$ and $\dffun$,  tightness of $ \SC{I}(\cdot, 1)$ in $C(\drpsp,\R)$ follows from Proposition \ref{prop:gen-conv2}-\ref{item:si-conv} (with $\hat \diff^\vep$ and $\drft$ playing the roles of $\hat\apar^{1,\vep}$ and $\apar^2$), and it follows that $\sqrt{\vep}\SC{I}(\cdot, 1) \prt 0$. Tightness and convergence of the other two terms in \eqref{eq:log-lik} to their respective limits  is a consequence of Proposition \ref{prop:diff-est}-\ref{item:diff-est:conv-2}. 

\end{proof}

\begin{proof}[{\bf Proof of Theorem \ref{th:const-AMLE}}] 
By Lemma \ref{lem:conv-like}, $-\vep\ell^\vep_{A}+ \pen^\vep \prt \bar \L_{\para_0}$ in $C(\drpsp, [0,\infty)).$ The assetion now readily follows from Theorem \ref{th:const-gen}.
	 	
\end{proof}

\begin{proof}[{\bf Proof of Theorem \ref{th:aml-clt}}]
Notice that \eqref{eq:deriv-like} shows that under $\PP_{\para_0}$ with $\para_0=(\drft_0,\diff_0)$
\begin{align}
\begin{aligned}
\vep\ \nabla_{\drft}\ell^{\vep}_A(\drft|\BX^\vep_{t_0: t_m}) =& \int_0^1D^T_\drft b(\drft, X^\vep(\vr_\vep(s))) a^{-1}(\hat\diff^\vep, X^\vep(\vr_\vep(s))) b(\drft_0, X^\vep(s)) \ ds\\
& \hs{.2cm} -\int_0^1 D^T_\drft b(\drft, X^\vep(\vr_\vep(s))) a^{-1}(\hat\diff^\vep, X^\vep(\vr_\vep(s))) b(\drft, X^\vep(\vr_\vep(s)))\ ds \\
& \hs{.2cm}+\sqrt \vep \int_0^1 D^T_\drft b(\drft,  X^\vep(\vr_\vep(s)))a^{-1}(\hat\diff^\vep, X^\vep(\vr_\vep(s)))\dffun(\diff_0, X^\vep(s)) dW(s).
\end{aligned}
\label{eq:l-deriv}
\end{align}
where $(D_\drft b(\drft,x))_{i,j} = \partial_{\drft_j} b_i(\drft,x)$. Now $\vep\nabla_{\drft}\ell^{\vep}_A(\drft |\BX^\vep_{t_0: t_m}) - \nabla_\drft\pen^\vep(\drft)\Big|_{\drft= \hat \drft^\vep}=0$. Multiplying this equation by $\vep^{-1/2}$ we have
\begin{align}
\non
0=&\ \Up^\vep(1)+ \vep^{-1/2} \int_0^1D^T_\drft b(\hat \drft^\vep, X^\vep(\vr_\vep(s))) a^{-1}(\hat\diff^\vep, X^\vep(\vr_\vep(s))) b(\drft_0, X^\vep(s)) \ ds \\ \non
& -\vep^{-1/2}\int_0^1 D^T_\drft b(\hat \drft^\vep, X^\vep(\vr_\vep(s))) a^{-1}(\hat\diff^\vep, X^\vep(\vr_\vep(s))) b(\hat \drft^\vep, X^\vep(\vr_\vep(s)))\ ds - \vep^{-1/2}\nabla_\drft \pen^\vep(\hat \drft^\vep)\\ \non
=& \ \Up^\vep(1)+ \vep^{-1/2} \int_0^1D^T_\drft b(\hat \drft^\vep, X^\vep(\vr_\vep(s))) a^{-1}(\hat\diff^\vep, X^\vep(\vr_\vep(s))) \Big(b(\drft_0, X^\vep(s)) - b(\drft_0, X^\vep(\vr_\vep(s)))\Big) ds\\ \non
& \ -  \vep^{-1/2} \int_0^1 D^T_\drft b(\hat \drft^\vep, X^\vep(\vr_\vep(s))) a^{-1}(\hat\diff^\vep, X^\vep(\vr_\vep(s)))\Big(b(\hat \drft^\vep, X^\vep(\vr_\vep(s))) - b(\drft_0, X^\vep(\vr_\vep(s)))\Big) ds\\ 
& \ - \vep^{-1/2}\nabla_\drft \pen^\vep(\hat \drft^\vep) 
\equiv\ \Up^\vep(1) + \vep^{-1/2}\SC{I}^\vep_1 - \vep^{-1/2}\SC{I}^\vep_2- \vep^{-1/2}\nabla_\drft \pen^\vep(\hat \drft^\vep),
\label{eq:l-deriv}
\end{align}
where the process $\Up^\vep$ is given by
\begin{align*}
\Up^\vep(t) = \int_0^ t D^T_\drft b(\hat \drft^\vep,  X^\vep(\vr_\vep(s)))a^{-1}(\hat\diff^\vep, X^\vep(\vr_\vep(s))) \dffun(\diff_0, X^\vep(s)) dW(s).
\end{align*}
Since $(\hat \drft^\vep,\hat \diff^\vep) \stackrel{P_{\para_0}} \rt \para_0 = (\drft_0,\diff_0)$ by the assumption \ref{cond:CLT-const} of Theorem \ref{th:aml-clt-0},  it follows by   Proposition \ref{prop:gen-conv2}-\ref{item:si-conv} (with $(\hat\drft^\vep,\hat \diff^\vep) $ playing the roles of $\hat\apar^{1,\vep}$ and discarding the $\apar^2$ input there) that as $\vep \rt 0$,
\begin{align}\label{eq:conv-mart}
\Up^\vep(1) \rt \Sigma_{\theta_0}^{1/2}\xi, 
\end{align}
where $\xi \sim \No_{n_0}(0,I)$, and by Corollary \ref{cor:diff-est-1} - \ref{item:diff-est:diff-2}  that $\vep^{-1/2}\SC{I}^\vep_1 \rt 0.$ It is easy to see because of the assumption of consistency of $\hat\drft^\vep$ and \ref{cond:clt-conv-penalty} that $\vep^{-1/2}\nabla_\drft \pen^\vep(\hat \drft^\vep) \prt 0.$

Next consider the term $\vep^{-1/2}\SC{I}^\vep_2 $. First order Taylor expansion of  $b(\cdot,X^\vep(\vr_\vep(s)))$ around $\drft_0$ gives
\begin{align*}
b(\hat \drft^\vep,X^\vep(\vr_\vep(s))) - b(\drft_0, X^\vep(\vr_\vep(s))) =&\ \lf(\int_0^1\lf(D_\drft b(\drft_0+u(\hat \drft^\vep-\drft_0), X^\vep(\vr_\vep(s)))\ri) du\ri) \\
& \times  (\hat \drft^\vep-\drft_0),
\end{align*}
and thus  \eqref{eq:l-deriv} after an interchange of integral gives
\begin{align}
\vep^{-1/2}\SC{I}^\vep_2 = \lf(\int_0^1 G^\vep(u)\ du\ri) \vep^{-1/2} (\hat \drft^\vep-\drft_0) = \Up^\vep(1)+\vep^{-1/2}\SC{I}^\vep_1- \vep^{-1/2}\nabla_\drft \pen^\vep(\hat \drft^\vep), 
\label{eq:CLT-prelim-eq}
\end{align}
where 
$$ G^\vep(u) = \int_0^1 D^T_\drft b(\hat \drft^\vep,  X^\vep(\vr_\vep(s))) a^{-1}(\hat\diff^\vep,X^\vep(\vr_\vep(s)))D_\drft b(\drft_0+u(\hat \drft^\vep-\drft_0), X^\vep(\vr_\vep(s))) ds.$$
We will show that
\begin{align*}
	\int_0^1 G^\vep(u)\ du \prt \Sigma_{\para_0}.
\end{align*}
By going to a subsequence if necessary we assume that $(\hat\drft^\vep, \hat\diff^\vep, \occ^{D,\vep} ) \ert (\drft_0, \diff_0, \inv_{\para_0}\ot \leb )$ a.s. 
From the proof of Proposition \ref{prop:gen-conv2}-\ref{item:si-conv}, it is easy to see that there exist compact sets $\dfpsp_0\subset \dfpsp$ and $\overline{B(\drft_0,r_0)} \subset \drpsp$ (for some suitable $r_0$)  and some $\vep_0 \equiv \vep_0(\om)$ such that for all $\vep < \vep_0$, $ \hat\drft^\vep(\om) \in \overline{B(\drft_0,r_0)}, \ \hat\diff^\vep(\om) \in \dfpsp_0$ a.s.

Now because of Condition \ref{cond:SDE-coeff}, the mapping 
$$(\drft, \diff, x) \rt D^T_\drft b(\drft,  x)a^{-1}(\diff,x)D_\drft b(\drft_0+u(\drft-\drft_0), x)$$
satisfies the conditions (on $\phi$) of Proposition \ref{prop:diff-est}, and it shows, with $(\hat \drft^\vep,\hat \diff^\vep)$ playing the role of $\hat \apar^{1,\vep}$,  that $G^\vep(u) \rt \Sigma_{\theta_0}$ as $\vep \rt 0$ for every $u$. Next, by Condition \ref{cond:SDE-coeff} for all $\vep < \vep_0,$
\begin{align*}
|G^\vep(u)| \leq &\  \drbd'_{b,0}(\hat \drft^\vep)\drbd'_{b,0}(\drft_0+u(\hat \drft^\vep-\drft_0))\dibd_{a,0}(\hat\diff^\vep) \int_0^1V(X^\vep(\vr_\vep(s)))^{2\bexp'_{b,0}+\sexp_{a,0}} \ ds\\
\leq &\  \lf(\drbd'^{*}_{b,0}\lf( \overline {B(\drft_0,r_0)}\ri)\ri)^2\dibd^*_{a,0}(\dfpsp_0) \int_0^1V(X^\vep(\vr_\vep(s)))^{2\bexp'_{b,0}+\sexp_{a,0}} \ ds \equiv F^\vep.
\end{align*}
Here $\drbd'^{*}_{b,0}\lf( \overline {B(\drft_0,r_0)}\ri) \equiv \sup\limits_{ \drft \in  \overline {B(\drft_0,r_0)} } \drbd'_{b,0}(\drft) < \infty$ and $\dibd^{*}_{a,0}\lf( \dfpsp_0\ri) \equiv \sup\limits_{ \diff \in \dfpsp_0} \dibd_{a,0}(\diff) < \infty$ by the assumption of local boundedness of $\drbd'_{b,0}$ and $\dibd_{a,0}$.
%

By  Proposition \ref{prop:diff-est}, $\dst F^\vep \ert  \lf(\drbd'^{*}_{b,0}\lf( \overline {B(\drft_0,r_0)}\ri)\ri)^2\dibd^*_{a,0}(\dfpsp_0)\int_{R^d}V(x)^{2\bexp'_{b,0}+\sexp_{a,0}}\inv_{\para_0}(dx)$.
It now follows by the (generalized) dominated convergence theorem that $\int_0^1 G^\vep(u)\ du  \stackrel{\vep \rt 0}\rt \Sigma_{\para_0}.$ Since $\Sigma_{\para_0}$ is non-singular, Lemma \ref{lem:aux-conv}, \eqref{eq:CLT-prelim-eq} and \eqref{eq:conv-mart} prove that $ \vep^{-1/2} (\hat\drft^\vep-\drft_0)$ is tight and converges in distribution to $\Sigma_{\para_0}^{-1/2}\xi$.
\end{proof}

We now work towards establishing limit theorems for the estimator of the diffusion parameter. The following lemma is crucial for consistency of $\hat\diff^\vep$. It is important to note that this lemma does not assume any specific forms of the diffusion coefficient $\dffun(\diff, \cdot)$ and thus will be important for developing asymptotics of the estimator $\hat \diff^\vep$ for other forms of $\dffun(\diff, \cdot)$ not discussed in the paper.

\begin{lemma}\label{lem:diff-const}
 Let $\mqd{X^\vep}^{D,\vr_\vep}$ denote the discretized quadratic variation of $X$ with respect to the partition $\{t_i\}$, as defined in \eqref{eq:quad-disc-X-ep}. Suppose that Condition \ref{cond:LLN-diff} of Theorem \ref{th:const-diff-0} hold, and $\Delta(\vep)/\vep \ert 0$. Then 
 $$\sup_{0\leq t\leq 1}\lf\|\vep \mqd{X^\vep}^{D,\vr_\vep}_t -\int_0^t a(\diff_0,X^\vep(\vr_\vep(s))\ ds\ri\| \ert 0 \mbox{ in } L^2(\PP_{\para_0}).$$  
 In particular, $\vep \mqd{X^\vep}^{D,\vr_\vep}_1 -\int_0^1 a(\diff_0,X^\vep(\vr_\vep(s))\ ds \ert 0$ in $L^2(\PP_{\para_0})$.
\end{lemma}

\begin{proof}
Notice that 
\begin{align*}
\mqd{X^\vep}^{D,\vr_\vep}_t \equiv&\ \sum_{i=1}^{m}\lf(X^\vep(t_i\wedge t) - X^\vep(t_{i-1}\wedge t)\ri)\lf(X^\vep(t_i\wedge t) - X^\vep(t_{i-1}\wedge t)\ri)^T\\
=&\ X^\vep(t)X^\vep(t)^T - X^\vep(0)X^\vep(0)^T - \int_0^t dX^\vep(s) X^\vep(\vr_\vep(s))^T - \int_0^t X^\vep(\vr_\vep(s))dX^\vep(s)^T.
\end{align*}
Hence, under $\PP_{\para_0}$, ($\para_0 = (\drft_0,\diff_0)$),
\begin{align*}
\mqd{X^\vep}^{D,\vr_\vep}_t =&\ X^\vep(t)X^\vep(t)^T - X^\vep(0)X^\vep(0)^T -2\Big(\vep^{-1}\int_0^t X^\vep(\vr_\vep(s)) b^T(\drft_0, X^\vep(s)) ds \\
& \ + \vep^{-1/2}\int_0^t X^\vep(\vr_\vep(s)) dW^T(s) \dffun^T(\diff_0, X^\vep(s))  \Big)_{\text{sym}}\\
=&\  X^\vep(t)X^\vep(t)^T - X^\vep(0)X^\vep(0)^T -2\Big(\vep^{-1}\int_0^t X^\vep(s) b^T(\drft_0, X^\vep(s)) ds \\
& \ + \vep^{-1/2}\int_0^t X^\vep(s) dW^T(s)\dffun^T(\diff_0, X^\vep(s)) \Big)_{\text{sym}}+2\vep^{-1}\lf(R^\vep_1(t)\ri)_{\text{sym}}+2\vep^{-1/2}\lf(R^\vep_2(t)\ri)_{\text{sym}}\\
=&\ \mqd{X^\vep}_t+ +2\vep^{-1}\lf(R^\vep_1(t)\ri)_{\text{sym}}+2\vep^{-1/2}\lf(R^\vep_2(t)\ri)_{\text{sym}}\\
=&\ \vep^{-1} \int_0^t  a(\diff_0,X^\vep(s))ds+2\vep^{-1}\lf(R^\vep_1(t)\ri)_{\text{sym}}+2\vep^{-1/2}\lf(R^\vep_2(t)\ri)_{\text{sym}}\\
=&\ \vep^{-1} \int_0^t  a(\diff_0,X^\vep(\vr_\vep(s)))ds+\vep^{-1} R^\vep_0(t) + +2\vep^{-1}\lf(R^\vep_1(t)\ri)_{\text{sym}}+2\vep^{-1/2}\lf(R^\vep_2(t)\ri)_{\text{sym}}.
\end{align*}
Here 
\begin{align*}
R^\vep_0(t)=&\ \int_0^t  a(\diff_0,X^\vep(s)) - a(\diff_0, X^\vep(\vr_\vep(s))) ds, \\
R^\vep_1(t)=&\ \int_0^t \lf(X^\vep(s) - X^\vep(\vr_\vep(s))\ri) b^T(\drft_0, X^\vep(s)) ds,\\
R^\vep_2(t)=&\  \int_0^t \lf(X^\vep(s) - X^\vep(\vr_\vep(s))\ri) dW^T(s) \dffun(\diff_0,X^\vep(s)),
\end{align*}
and recall for a square matrix $A$, $A_{\text{sym}} = (A+A^T)/2$.
Thus
\begin{align}\label{eq:diff-err}
\vep \mqd{X^\vep}^{D,\vr_\vep}_t -\int_0^t a(\diff_0,X^\vep(\vr_\vep(s)) = \lf(R^\vep_0(t)+2\lf(R^\vep_1(t)\ri)_{\text{sym}}+2\vep^{1/2}\lf(R^\vep_2(t)\ri)_{\text{sym}}\ri).
\end{align}

\np
Notice that by Proposition Corollary \ref{cor:diff-est-1} (also see Lemma \ref{lem:fun-diff-est-2}), for some constant $\const_{\ref*{lem:diff-const},0} \geq 0$,  
\begin{align*}
\EE_{\para_0}\lf[\sup_{0\leq t\leq 1}\|R^\vep_0(t)\|^2\ri] \leq &\ \const_{\ref*{lem:diff-const},0} (\Delta(\vep)/\vep)^{\holex_\dffun} \rt 0,\quad 
\EE_{\para_0}\lf[\sup_{0\leq t\leq 1}\|R^\vep_1(t)\|^2\ri]  \leq \ \const_{\ref*{lem:diff-const},0} (\Delta(\vep)/\vep) \ert 0,
\end{align*}
and finally, an additional use of BDG inequality shows that for some constant $\const_{\ref*{lem:diff-const},1} \geq 0$
\begin{align*}
\vep\EE_{\para_0}\lf[\sup_{0\leq t\leq 1}\|R^\vep_2(t)\|^2\ri] \leq \vep \dibd^2(\diff_0)\|\EE_{\para_0}\int_0^1\|X^\vep(s) - X^\vep(\vr_\vep(s)\|^2 V(X^\vep(s))^{2\sexp_0}\ ds \leq \const_{\ref*{lem:diff-const},1} \Delta(\vep) \ert 0.
\end{align*}
\end{proof}

\begin{proof}[{\bf Proof of Theorem \ref{th:const-diff-0}}] For the proof we operate in the scaling regime introduced in Section \ref{sec:new-scaling}. As mentioned at the end of Section \ref{sec:new-scaling}, proving Theorem \ref{th:const-diff-0} is equivalent to showing that $\hat \diff^\vep \equiv \hat\diff^\vep(\BX^\vep_{ t_0: t_m}) \stackrel{\PP_{\theta_0}} \Rt \diff_0$ as $\vep \rt 0$ under  Condition \ref{cond:LLN-diff} of that theorem and the assumption that $\Delta(\vep)/\vep \ert 0$.

We first consider the case when $\diff$ is of Form 1 (see Section \ref{sec:diff-est}), and for specificity we assume $\dffun(\diff,x) = \lf(a_0(x)\diff\ri)^{1/2}$ (that is, $a(\diff,x) \equiv \dffun(\diff,x)\dffun^T(\diff,x) = a_0(x)\diff$), in which case $\hat \diff^\vep \equiv \hat\diff^\vep(\BX^\vep_{ t_0: t_m})$ is given by
\eqref{eq:diff-est-1-sc}.  We have
\begin{align}\label{eq:diff-diff-est}
\hat \diff^\vep -\diff_0 = \lf(\int_0^1 a_0(X^\vep(\vr_\vep(s))\ ds\ri)^{-1}\lf(\vep \mqd{X^\vep}^{D,\vr_\vep}_1 -\int_0^1 a_0(X^\vep(\vr_\vep(s))) \diff_0\ ds\ri),
\end{align}
and the result follows from Lemma \ref{lem:diff-const} and the observation that by Proposition \ref{prop:diff-est}-\ref{item:diff-est:conv-1}
\begin{align} \label{eq:diff-erg}
\int_0^1 a_0(X^\vep(\vr_\vep(s))\ ds \prt \int a_0(x) \pi^{st}_{\para_0}(dx).
\end{align} 

When $\diff$ is of Form 2 notice that by the property of $\ve_{d\times d}$ operator
\begin{align}\label{eq:diff-diff-est-2}
\begin{aligned}
    \ve_{d\times d}(\hat\diff^\vep-\diff_0) =&\ \ve_{d\times d}(\hat\diff^\vep)-\ve_{d\times d}(\diff_0) = \ \lf(\int_0^{1} \dffun_0(X^\vep(\vr_\vep(s))) \ot \dffun_0(X^\vep(\vr_\vep(s)))ds\ri)^{-1}\\
    & \hs{.3cm} \lf(\ve_{d\times d}(\vep \mqd{X^\vep}^{D,\vr_\vep}_{1}) - \int_0^{1} \dffun_0(X^\vep(\vr_\vep(s))) \ot \dffun_0(X^\vep(\vr_\vep(s)))ds\ \ve_{d\times d}(\diff_0)\ri)\\
    =&\ \lf(\int_0^{1} \dffun_0(X^\vep(\vr_\vep(s))) \ot  \dffun_0(X^\vep(\vr_\vep(s)))ds\ri)^{-1}\\
    & \hs{.3cm} \ve_{d\times d}\lf(\vep \mqd{X^\vep}^{D,\vr_\vep}_{1} - \int_0^1 \dffun_0(X^\vep(\vr_\vep(s)) \diff_0 \dffun^T_0(X^\vep(\vr_\vep(s))\ ds\ri)\\
    =&\ \lf(\int_0^{1} \dffun_0(X^\vep(\vr_\vep(s))) \ot  \dffun_0(X^\vep(\vr_\vep(s)))ds\ri)^{-1}\\
    & \hs{.3cm} \ve_{d\times d}\lf(\vep \mqd{X^\vep}^{D,\vr_\vep}_{1} - \int_0^1 a(\diff_0,X^\vep(\vr_\vep(s)) \ ds\ri).
    \end{aligned}
\end{align}
Again the result follows from Lemma \ref{lem:diff-const} and the observation that by Proposition \ref{prop:diff-est}-\ref{item:diff-est:conv-1}
\begin{align}  \label{eq:diff-erg-2}
\int_0^1 \dffun_0(X^\vep(\vr_\vep(s))) \ot  \dffun_0(X^\vep(\vr_\vep(s)))\ ds \prt \int \dffun_0(x) \ot  \dffun_0(x) \pi^{st}_{\para_0}(dx).
\end{align} 
\end{proof}

\begin{proposition} \label{prop:clt-diff-err}
Suppose that Condition \ref{cond:CLT-diff} of Theorem \ref{th:clt-diff-0} hold, and $\dst \f{(\Delta(\vep)/\vep)^{\holex_b \wedge \holex'_a}}{\vep} = \Delta(\vep)^{\holex_b \wedge \holex'_a}/\vep^{1+\holex_b \wedge \holex'_a} \ert 0$. Let $\zeta$ be the $\R^{d\times d}$-valued random variable defined in Theorem \ref{th:clt-diff-0}. Then  as $\vep \rt 0$  $$\Delta(\vep)^{-1/2}\lf(\vep \mqd{X^\vep}^{D,\vr_\vep}_1 -\int_0^1 a(\diff_0,X^\vep(\vr_\vep(s))\ri) \RT 2(\zeta)_{\mathrm{sym}}.$$
\end{proposition}

\begin{proof}
From \eqref{eq:diff-err},
\begin{align}\label{eq:diff-err-clt}
\Delta(\vep)^{-1/2}\lf(\vep \mqd{X^\vep}^{D,\vr_\vep}_t -\int_0^t a(\diff_0,X^\vep(\vr_\vep(s))\ri) = \Delta(\vep)^{-1/2}\lf( R^\vep_0(t)+2\lf(R^\vep_1(t)\ri)_{\text{sym}}+2\vep^{1/2}\lf(R^\vep_2(t)\ri)_{\text{sym}}\ri).
\end{align}

The proof of CLT needs a more subtle analysis of the error terms $R^\vep_l, l=0,1,2$. We first focus on the term $R^\vep_2$. Under $\PP_{\para_0}$
\begin{align*}
X^\vep(s) - X^\vep(\vr_\vep(s)) = \vep^{-1} \int_{\vr_\vep(s)}^s b(\drft_0, X^\vep(r))dr+\vep^{-1/2}\int_{\vr_\vep(s)}^s  \dffun(\diff_0,X^\vep(r))dW(r),
\end{align*}
and hence
\begin{align*}
\vep^{1/2} \Delta^{-1/2}(\vep)R^\vep_2(t)=&\  \lf(\vep/\Delta(\vep)\ri)^{1/2}\int \lf(X^\vep(s) - X^\vep(\vr_\vep(s))\ri) dW^T(s) \dffun(\diff_0,X^\vep(s))\\
=&\ \lf(\vep\Delta(\vep)\ri)^{-1/2}\int_0^t\lf(\int_{\vr_\vep(s)}^s b(\diff_0, X^\vep(r))dr\ri)dW^T(s) \dffun(\diff_0,X^\vep(s))\\
& \ + \Delta(\vep)^{-1/2} \int_0^t\lf(\int_{\vr_\vep(s)}^s \dffun(\diff_0,X^\vep(r))dW(r)\ri)dW^T(s) \dffun(\diff_0,X^\vep(s))\\
\equiv&\ \lf(\vep\Delta(\vep)\ri)^{-1/2} R^\vep_{2,1}(t)+\Delta(\vep)^{-1/2} R^\vep_{2,2}(t)+\Delta(\vep)^{-1/2} M^\vep_{2}(t).
\end{align*}
where 
\begin{align*}
R^\vep_{2,1}(t) \equiv &\ \int_0^t\lf(\int_{\vr_\vep(s)}^s b(\diff_0, X^\vep(r))dr\ri)dW^T(s) \dffun(\diff_0,X^\vep(s)),\\
R^\vep_{2,2}(t)=&\ \int_0^t\lf(\int_{\vr_\vep(s)}^s \Big(\dffun(\diff_0,X^\vep(r))-  \dffun(\diff_0, X^\vep(\vr_\vep(r)))\Big)dW(r)\ri)dW^T(s) \dffun(\diff_0,X^\vep(s))\\
&\ + \int_0^t\lf(\int_{\vr_\vep(s)}^s   \dffun(\diff_0,X^\vep(\vr_\vep(r)))dW(r)\ri)dW^T(s) \Big(\dffun(\diff_0,X^\vep(s)) -\dffun(\diff_0,X^\vep(\vr_\ep(s)))\Big),\\
M^\vep_{2}(t)=&\ \int_0^t\lf(\int_{\vr_\vep(s)}^s  \dffun(\diff_0,X^\vep(\vr_\vep(r))dW(r)\ri)dW^T(s) \dffun(\diff_0,X^\vep(\vr_\vep(s)))\\
=&\ \int_0^t\dffun(\diff_0,X^\vep(\vr_\vep(s)))\big(W(s) - W(\vr_\vep(s))\big)dW^T(s) \dffun(\diff_0,X^\vep(\vr_\vep(s))).
\end{align*}

\np
{\em Convergence of $M^\vep_2$:} Note that $M^\vep_2$ is a martingale, and its $j$-th column is given by
\begin{align*}
M^\vep_{2,*,j}(t)=&\ \int_0^t\dffun(\diff_0,X^\vep(\vr_\vep(s)))\big(W(s) - W(\vr_\vep(s))\big)dW^T(s) \dffun(\diff_0,X^\vep(\vr_\vep(s)))_{*,j}\\
=&\ \int_0^t\dffun(\diff_0,X^\vep(\vr_\vep(s)))\big(W(s) - W(\vr_\vep(s))\big)\dffun(\diff_0,X^\vep(\vr_\vep(s)))_{j,*}\ dW(s) 
\end{align*}
and the matrix covariation between its $j$-th and $k$-th columns is given by 
\begin{align*}
\mqd{M^\vep_{2,*,j}, M^\vep_{2,*,l}}_t =&\ \int_0^t\dffun(\diff_0,X^\vep(\vr_\vep(s)))\big(W(s) - W(\vr_\vep(s))\big)\dffun(\diff_0,X^\vep(\vr_\vep(s)))_{j.*}\\
& \ \Big(\dffun(\diff_0,X^\vep(\vr_\vep(s)))\big(W(s) - W(\vr_\vep(s))\big)\dffun(\diff_0,X^\vep(\vr_\vep(s)))_{l,*}\Big)^T ds\\
=&\ \int_0^t\dffun(\diff_0,X^\vep(\vr_\vep(s)))\big(W(s) - W(\vr_\vep(s))\big)\dffun(\diff_0,X^\vep(\vr_\vep(s)))_{j,*}\\
& \ \dffun(\diff_0,X^\vep(\vr_\vep(s)))_{*,l}\big(W(s) - W(\vr_\vep(s))\big)^T\dffun(\diff_0,X^\vep(\vr_\vep(s))) ds\\
=&\ \int_0^t a(\diff_0,X^\vep(\vr_\vep(s)))_{j,l} \dffun(\diff_0,X^\vep(\vr_\vep(s)))\big(W(s) - W(\vr_\vep(s))\big)\\
& \ \big(W(s) - W(\vr_\vep(s))\big)^T\dffun(\diff_0,X^\vep(\vr_\vep(s))) ds
\equiv Q^\vep_{j,l}(t)+ \tilde R^\vep_{2,j,l}(t),
\end{align*}
where
\begin{align*}
Q^\vep_{j,l}(t)\equiv &\ \int_0^t a(\diff_0,X^\vep(\vr_\vep(s)))_{j,l} a(\diff_0,X^\vep(\vr_\vep(s))) (s-\vr_\vep(s)) \ ds,\\
\tilde R^\vep_{2,j,l}(t) \equiv  &\ \int_0^t a(\diff_0,X^\vep(\vr_\vep(s)))_{j,l} \dffun(\diff_0,X^\vep(\vr_\vep(s)))
 \Big\{\big(W(s) - W(\vr_\vep(s))\big)
\big(W(s) - W(\vr_\vep(s))\big)^T\\
&\ -(s-\vr_\vep(s))I\Big\}\dffun(\diff_0,X^\vep(\vr_\vep(s))) ds \equiv  \ [M^\vep_{2,j}, M^\vep_{2,l}]^{\ot}_t - \SC{B}^\vep_{j,l}(t).
\end{align*}
Notice that
\begin{align*}
Q^\vep_{j,l}(1)= &\ \sum_{k=0}^{m-1}\lf(a(\diff_0,X^\vep(t_k)\ri)_{j,l} a(\diff_0,X^\vep(t_k))\int_{t_k}^{t_{k+1}}(s-t_k)ds\\
=&\ \sum_{k=0}^{m-1}\lf(a(\diff_0,X^\vep(t_k)\ri)_{j,l} a(\diff_0,X^\vep(t_k))\Delta(\vep)^2/2\\
=&\ \f{\Delta(\vep)}{2}\int_0^1\lf(a(\diff_0,X^\vep(\vr_\vep(s))\ri)_{j,l} a(\diff_0,X^\vep(\vr_\vep(s)))ds.
\end{align*}
It follows by Proposition \ref{prop:diff-est} that as $\vep \rt 0$
\begin{align*}
Q^\vep_{j,l}(1)/\Delta(\vep) \prt \f{1}{2} \int \lf(a(\diff_0,x)\ri)_{j,l} a(\diff_0,x) \pi^{st}_{\para_0}(dx).
\end{align*}
Next, observe that the process $\{\tilde R^\vep_{2,j,l}(t_k):k=0,1,2,\hdots,m\}$ is a discrete-time martingale. 
Consequently, by Cauchy-Schwarz inequality
\begin{align*}
\EE_{\para_0}[\|\tilde R^\vep_{2,j,l}(1)\|^2] =&\ \sum_{k=0}^{m-1}\EE_{\para_0} \lf\|\int_{t_k}^{t_{k+1}}a(\diff_0,X^\vep(t_k))_{j,l}\dffun(\diff_0,X^\vep(t_k))
 \Big\{\big(W(s) - W(t_k)\big)
\big(W(s) - W(t_k)\big)^T\ri.\\
&\ \lf.-(s-t_k)I\Big\}\dffun(\diff_0,X^\vep(t_k) ds\ri\|^2\\
 &\leq  \Delta(\vep)\sum_{k=0}^{m-1} |\lf(a(\diff_0,X^\vep(t_k)\ri)_{j,l}|\lf\|a(\diff_0,X^\vep(t_k))\ri\|\\
 & \ \times \int_{t_k}^{t_{k+1}}\EE_{\para_0} \lf\|\big(W(s) - W(t_k)\big)
 \big(W(s) - W(t_k)\big)^T-(s-t_k)I\ri\|^2ds\\
 \leq&\ \const_{\ref*{prop:clt-diff-err},0} \Delta(\vep)^4 \sum_{k=0}^{m-1} |\lf(a(\diff_0,X^\vep(t_k))\ri)_{j,l}|\lf\|a(\diff_0,X^\vep(t_k))\ri\|\\
 =&\ \const_{\ref*{prop:clt-diff-err},0} \Delta(\vep)^3 \int_0^1 |\lf(a(\diff_0,X^\vep(\vr_\vep(s)))\ri)_{j,l}|\lf\|a(\diff_0,X^\vep(\vr_\vep(s)))\ri\|\ ds.
 \end{align*}
It follows that
\begin{align*}
\Delta(\vep)^{-2}\EE_{\para_0}[\|\tilde R^\vep_{2,j,l}(1)\|^2] \ert 0.
\end{align*}
Hence 
\begin{align*}
\Delta(\vep)^{-1}\mqd{M^\vep_{2,*,j}, M^\vep_{2,*,l}}_1 \prt \f{1}{2} \int \lf(a(\diff_0,x)\ri)_{j,l} a(\diff_0,x) \pi^{st}_{\para_0}(dx),
\end{align*}
that is,
$$\Delta(\vep)^{-1}\mqd{\ve(M^\vep_{2})}_1\prt  \f{1}{2} \int a(\diff_0,x) \ot a(\diff_0,x)\pi^{st}_{\para_0}(dx).$$
By Martingale CLT, as $\vep \rt 0$
\begin{align*}
\Delta(\vep)^{-1/2}\ve(M^\vep_{2}(1)) \RT \lf(\f{1}{2}\int a(\diff_0,x) \ot a(\diff_0,x) \pi^{st}_{\para_0}(dx)\ri)^{1/2} \No_{d^2}(0,I).
\end{align*}

\np
{\em Convergence of $\lf(\vep\Delta(\vep)\ri)^{-1/2} R^\vep_{2,1}$:}
Notice that by BDG inequality and Lemma \ref{lem:simp-int-bd},
\begin{align*}
\EE_{\para_0}\lf(\|R^\vep_{2,1}(t)\|^2\ri)\leq &\ \EE_{\para_0}\int_0^t\lf\|\int_{\vr_\vep(s)}^s b(\diff_0, X^\vep(r))dr\ri\|^2\|a(\diff_0,X^\vep(s))\|\ ds\\
\leq & \ \lf(\int_0^t\EE_{\para_0}\lf\|\int_{\vr_\vep(s)}^s b(\diff_0, X^\vep(r))dr\ri\|^4 ds\ri)^{1/2}\lf(\EE_{\para_0} \int_0^t\|a(\diff_0,X^\vep(s))\|^2\ ds\ri)^{1/2}\\
\leq &\ \const_{\ref*{prop:clt-diff-err},1}\Delta(\vep)^{3/2}\lf(\int_0^t\EE_{\para_0}\int_{\vr_\vep(s)}^s \|b(\diff_0, X^\vep(r))\|^4\ dr\ ds\ri)^{1/2}\\
\leq &\ \const_{\ref*{prop:clt-diff-err},1}\Delta(\vep)^{2}\lf(\EE_{\para_0}\int_0^t \|b(\diff_0, X^\vep(r))\|^4\ dr\ri)^{1/2} \leq \const_{\ref*{prop:clt-diff-err},2} \Delta(\vep)^{2}.
\end{align*}
It follows that $\lf(\vep\Delta(\vep)\ri)^{-1} \EE_{\para_0}\lf(\|R^\vep_{2,1}(1)\|^2\ri)\leq \const_{\ref*{prop:clt-diff-err},2}\Delta(\vep)/\vep \ert 0.$

\np
{\em Convergence of $\Delta(\vep)^{-1/2} R^\vep_{2,2}$:}
Similar steps using a combination of Burkholder and Cauchy-Schwarz inequality and Corollary \ref{cor:diff-est-1}-\ref{item:diff-est:diff-1} show that for some constant $\const_{\ref*{prop:clt-diff-err},3}$
\begin{align*}
\Delta(\vep)^{-1}\EE_{\para_0}\lf(\|R^\vep_{2,2}(1)\|^2\ri) \leq \const_{\ref*{prop:clt-diff-err},3} \Delta(\vep)^{-1}\Delta(\vep)\lf(\Delta(\vep)/\vep\ri)^{\holex_\dffun} = \const_{\ref*{prop:clt-diff-err},3} \lf(\Delta(\vep)/\vep\ri)^{\holex_\dffun} \ert 0.
\end{align*}
Thus as $\vep \rt 0$, 
\begin{align*}
\vep^{1/2}\Delta(\vep)^{-1/2}\ve(R^\vep_{2}(1)) \RT \lf(\f{1}{2}\int a(\diff_0,x) \ot a(\diff_0,x) \pi^{st}_{\para_0}(dx)\ri)^{1/2} \No_{d^2}(0,I).
\end{align*}

Finally we consider the terms involving $R^\vep_0$ and $R^\vep_1$ in \eqref{eq:diff-err-clt}, and complete the proof by showing they converge to zero.

\np
{\em Convergence of $ \Delta^{-1/2} R^\vep_{1}$ and $ \Delta^{-1/2} R^\vep_{0}$ :}
The estimates for both these quantities follow from Lemma \ref{lem:fun-diff-est-2} which shows that for some constant $\const_{\ref*{prop:clt-diff-err},4}$
\begin{align*}
 \Delta(\vep)^{-1} \EE_{\para_0}\lf[\|R^\vep_{1}(1)\|^2\ri] \leq &\ \const_{\ref*{prop:clt-diff-err},4}\Delta(\vep)^{-1} (\Delta(\vep)/\vep)^{1+\holex_b} = \const_{\ref*{prop:clt-diff-err},4} \Delta(\vep)^{\holex_b}/\vep^{1+\holex_b}\\
 \Delta(\vep)^{-1}\EE_{\para_0}\lf[\|R^\vep_{0}(1)\|^2\ri]  \leq &\  \const_{\ref*{prop:clt-diff-err},4}\Delta(\vep)^{-1} (\Delta(\vep)/\vep)^{1+\holex'_a} = \const_{\ref*{prop:clt-diff-err},4} \Delta(\vep)^{\holex'_a}/\vep^{1+\holex'_a}
  \ert 0.
\end{align*}
\end{proof}

\begin{proof}[Proof of Theorem \ref{th:clt-diff-0}]
As before, we operate in the scaling regime introduced in Section \ref{sec:new-scaling}. As mentioned at the end of Section \ref{sec:new-scaling}, proving Theorem \ref{th:clt-diff-0} is equivalent to showing that the limit of $\Delta(\vep)^{-1/2}(\hat \diff^\vep -\diff_0)$ described in \ref{item:clt-diff-form1} and \ref{item:clt-diff-form2} hold for $\hat \diff^\vep \equiv \hat \diff^\vep(\BX^\vep_{t_0:t_m})$  under  Condition \ref{cond:CLT-diff} of that theorem and the assumption that $\dst\f{(\Delta(\vep)/\vep)^{\holex_b \wedge \holex'_a}}{\vep} = \Delta(\vep)^{\holex_b \wedge \holex'_a}/\vep^{1+\holex_b \wedge \holex'_a} \ert 0$.

Now observe that 
\ref{item:clt-diff-form1} readily follows from \eqref{eq:diff-diff-est},   \eqref{eq:diff-erg} and Proposition \ref{prop:clt-diff-err}, while
\ref{item:clt-diff-form2} follows from \eqref{eq:diff-diff-est-2}, \eqref{eq:diff-erg-2} and Proposition \ref{prop:clt-diff-err}. 
\end{proof}

\setcounter{section}{0}
\setcounter{theorem}{0}
\setcounter{equation}{0}
\renewcommand{\theequation}{\thesection.\arabic{equation}}

\appendix

\section{} \label{sec:appendix}

\begin{lemma}\label{lem:simp-int-bd}
    Let $\{t_k\}$ be a partition of $[0,T]$ with corresponding step function $\vr$ defined by  $\vr(s) =  t_{k},$ if $ t_k \leq s < t_{k+1}$. Then for any integrable function $h: [0,\infty) \rt [0,\infty)$,
    \begin{align*}
        \int_0^T\int_{\vr(s)}^s h(r)\ dr \leq \Delta \int_0^T h(r)\ dr.
    \end{align*}
\end{lemma}

\begin{proof}
   This easily follows from a change in the order of integration:
   	\begin{align*}
		\int_0^T \int_{\vr(s)}^sh(r)dr\ ds = &\ \sum_k \int_{ t_k}^{ t_{k+1}}\int_{ t_k}^sh(r)dr\ ds
		= \ \sum_k \int_{ t_k}^{ t_{k+1}}\int_{r}^{ t_{k+1}}h(r)ds\ dr\\ 
		\leq &\  \Delta \sum_k \int_{ t_k}^{ t_{k+1}}h(r)dr =   \Delta \int_{0}^{T}|h(r)| dr. 
	 \end{align*}
\end{proof}

The following lemma provides a slight generalization of the convergence of integrals with respect to probability measures under a uniform integrability type condition. The proof is included for completeness.
\begin{lemma}\label{lem:uni-int}
	Let $\{\mu_n\}$ a sequence of probability measures on $\R^d$ and  $\mu_n \RT \mu_0$ as $n\rt \infty$. Let $ f: \R^d \rt [0,\infty)$ be a lower semicontinuous function such that
	\begin{align}\label{assum-1}
		K_0 \doteq \sup_n \int_{\R^d} f(x)\ \mu_n(dx) <\infty.
	\end{align}
	Suppose that $h:\Theta\times \R^d \rt \R^{d_1}$  is a function satisfying the following conditions:
	\begin{enumerate}[label=(\roman*), ref=(\roman*)]
		\item if $\theta_n \rt \theta_0$, then $h(\theta_n,\cdot) \rt h(\theta_0,\cdot)$ uniformly on compact sets of $\R^d$ as $n\rt \infty$; specifically, for any compact sets $\SC{K} \subset \R^d$, $\dst \sup_{x \in \SC{K}}\|h(\theta_n,x) -h(\theta_0,x)\|  \stackrel{n\rt \infty}\rt  0$ 
		
		\item $\|h(\theta,x))\| \leq J(\theta)U(x)$ for some locally bounded function $J:\Theta \rt [0,\infty)$ and $U :\R^d \rt [0,\infty)$  satisfying $U(x) / f(x) \rt 0$ as $\|x\| \rt \infty$.
	\end{enumerate}
	Let $\{\theta_n\}$ be a sequence such that $\theta_n \stackrel{n\rt \infty}\rt \theta_0$. Then as $n \rt \infty$,
	\begin{align*}
		\int_{\R^d} h(\theta_n,x) \mu_n(dx) \rt \int_{\R^d} h(\theta_0,x) \mu(dx).
	\end{align*}	
\end{lemma}

\begin{proof}
	First observe that from \eqref{assum-1}, by  lower semicontinuity of $f$ and a generalization of Fatou's lemma (see \cite[Theorem 1.1]{FKZ14}), we have
	$$ \int_{\R^d} f(x) \mu(dx) \leq K_0.$$
	Fix an $\vep>0$. Since $\{\mu_n\}$ is tight and $U(x) / f(x) \rt 0$ as $\|x\| \rt \infty$, there exists an $R_0>0$ such that 
	\begin{align*}
		\mu_n\{x \in \R^d: \|x\| > R_0\} \leq \vep, \quad \sup_{\|x\| > R_0} U(x)/f(x) \leq \vep.
	\end{align*}		
	Writing $\int_{\R^d} h(\theta_n,x) \mu_n(dx) = \int_{\R^d} h(\theta_0,x) \mu_n(dx) + \SC{R}_n$, we first show that the term
	$\SC{R}_n \equiv \int_{\R^d} \lf(h(\theta_n,x) - h(\theta_0,x)\ri) \mu_n(dx) \stackrel{n\rt \infty}\rt 0.$  To see this, notice that
	\begin{align*}
		|\SC{R}_n| \leq &\ \sup_{\|x\| \leq R_0}\|h(\theta_n,x) - h(\theta_0,x)\| + \int_{\|x\| >R_0} \lf(|h(\theta_n,x)| + |h(\theta_0,x)|\ri) \mu_n(dx)\\
		\leq &\ \sup_{\|x\| \leq R_0}\|h(\theta_n,x) - h(\theta_0,x)\|  + J(\theta_n) \vee J(\theta_0)  \int_{\|x\| >R_0} U(x) \mu_n (dx)\\
		\leq &\ \sup_{\|x\| \leq R_0}\|h(\theta_n,x) - h(\theta_0,x)\|  + \vep J(\theta_n) \vee J(\theta_0)  \int_{\|x\| >R_0} f(x) \mu_n (dx)\\
		\leq &\ \sup_{\|x\| \leq R_0}\|h(\theta_n,x) - h(\theta_0,x)\| + \vep J^* K_0 \stackrel{n\rt \infty}\rt \vep J^*K_0.
	\end{align*} 
 Here $J^* \dfeq J(\para_0) \vee \sup_n J(\para_n) < \infty$ since $J$ is locally bounded.
	Since $\vep$ is arbitrary, it follows that $\SC{R}_n \stackrel{n\rt \infty}\rt 0$. The convergence,  $\int_{\R^d} h(\theta_0,x) \mu_n(dx) \stackrel{n\rt \infty}\rt \int_{\R^d} h(\theta_0,x) \mu_0(dx)$, follows from standard result on uniform integrability. An easy way to see this is to invoke Skorohod representation theorem to get a probability space $(\tilde \Omega, \tilde{\SC{F}}, \tilde \PP)$, and random variables $\tilde X_0, \tilde X_n, n\geq 1$ on $(\tilde \Omega, \tilde{\SC{F}}, \tilde \PP)$ such that $\tilde X_n \sim \mu_n, \ \tilde X_0 \sim \mu_0$ and $\tilde X_n \rt \tilde X_0$ $\tilde \PP$-a.s. The conditions on $h$ then imply that 
	the sequence $\{h(\theta_0,\tilde X_n)\}$ is uniformly integrable.
\end{proof}

\begin{lemma}\label{lem:aux-conv}
	Let $\{H_n\}, \{Z_n\}$ be sequences of $d_0\times d_0$ and $d_0\times d_1$ real matrices. Let $U_n = H_nZ_n$. Suppose that as $n\rt \infty$, $H_n \rt H_0, \ U_n \rt U_0$, where $H_0$ is non-singular. Then  $Z_n \rt Z_0 = H_0^{-1}U_0$. 
\end{lemma}

\begin{proof} The main point here is to observe that $H_n$ will be non-singular for sufficiently large $n$.
	Indeed, since $\det: \R^{d_0\times d_0}$ is a continuous function and $\det(H_0) \neq 0$, there exists an open set $\mathbb{O} \subset \R^{d_0\times d_0}$ containing $H_0$, such that for any $M \in \mathbb{O}$, $\det(M) \neq 0$; that is, $M$ is non-singular. Since $H_n \rt H_0$, $H_n \in \mathbb{O}$ for all $n \geq N_0$. Thus for all $n \geq N_0$, $Z_n = H_n^{-1}U_n$. The rest now follows from the continuity of the mapping $A \in \operatorname{GL}(d_0,\R) \rt A^{-1} \in \operatorname{GL}(d_0,\R)$. Here $\operatorname{GL}(d_0,\R)$ is the space of all invertible $d_0\times d_0$ matrices.
\end{proof}

\bibliographystyle{plainnat}
\renewcommand{\bibfont}{\footnotesize} 
\bibliography{Ref-ML}
\end{document}